\documentclass[preprint,11pt]{elsarticle}

\usepackage{graphicx}
\usepackage{pifont,latexsym,ifthen,amsthm,rotating,calc,textcase,booktabs}
\usepackage{amsfonts,amssymb,amsbsy,amsmath,extarrows,bm}
\allowdisplaybreaks
\usepackage{color}
\numberwithin{equation}{section}
\usepackage[dvipsnames, svgnames, x11names]{xcolor}
\usepackage{pgf}
\usepackage{CJK}
\usepackage{graphicx}
\usepackage{tikz}
\usepackage{stmaryrd}
\usepackage{graphicx}
\usepackage{hyperref}
\usepackage[latin1]{inputenc}
\usepackage[T1]{fontenc}
\usepackage[english]{babel}
\usepackage{listings}
\usepackage{xcolor,mathrsfs,url}
\usepackage{ifthen}
\usepackage{fancyhdr}
\usepackage{verbatim}
\newtheorem{theorem}{Theorem}[section]
\newtheorem{lemma}{Lemma}[section]
\newtheorem{corollary}{Corollary}[section]
\newtheorem{Proposition}{Proposition}[section]
\newtheorem{RHP}{RH problem}[section]

\newtheorem{remark}[theorem]{Remark}
\usepackage {caption}
\usepackage{float}
\usepackage{subfigure}
\usepackage{natbib}
\hypersetup{hypertex=true,
            colorlinks=true,
            linkcolor=blue,
            anchorcolor=blue,
            citecolor=red}
\numberwithin {figure} {section}
\DeclareMathOperator*{\res}{Res}
\DeclareMathOperator*{\im}{Im}
\DeclareMathOperator*{\re}{Re}
\begin{document}

\begin{frontmatter}
\title{The Cauchy problem  for the  Novikov equation under a nonzero background: Painlev\'e  asymptotics
 in a transition zone }

\author[inst2]{Zhaoyu Wang}

\author[inst3]{Xuan Zhou}

\author[inst2]{Engui Fan$^{*,}$  }

\address[inst2]{ School of Mathematical Sciences and Key Laboratory of Mathematics for Nonlinear Science, Fudan University, Shanghai,200433, China}
\address[inst3]{ College of Mathematics and Systems Science, Shandong University of Science and Technology, Qingdao, 266061,  China\\
* Corresponding author and e-mail address: faneg@fudan.edu.cn  }

\begin{abstract}

In this paper, we   investigate the Painlev\'e  asymptotics  in a transition zone  for  the  solutions to  the Cauchy problem of  the Novikov equation under a  nonzero background
\begin{align}
	&u_{t}-u_{txx}+4 u_{x}=3uu_xu_{xx}+u^2u_{xxx}, \nonumber \\
	&u(x, 0)=u_{0}(x),\nonumber
\end{align}	
where  $u_0(x)\rightarrow \kappa>0, \ x\rightarrow \pm \infty$  and $u_0(x)-\kappa$ is
assumed in the Schwarz space.
 This  result is established  by performing   the   $\overline\partial$-steepest descent analysis to
  a  Riemann-Hilbert problem associated  with   the the Cauchy problem  in a new spatial scale
  \begin{equation*}
  	y = x - \int_{x}^{\infty} \left((u-u_{xx}+1)^{2/3}-1\right)ds,
  \end{equation*}
   for large times in   the
 transition zone $y/t \approx -1/8 $.
It is shown that the leading order term of the asymptotic approximation
comes from the contribution of solitons, while the sub-leading term is related to   the solution of  the Painlev\'e  \uppercase\expandafter{\romannumeral2} equation.

\end{abstract}


\end{frontmatter}

\tableofcontents

\quad

\section {Introduction}

In this paper, we are concerned with the Painlev\'e asymptotics  of solutions to  the Cauchy problem for the Novikov equation   under a nonzero background
 \begin{align}
	&u_{t}-u_{txx}+4 u_{x}=3uu_xu_{xx}+u^2u_{xxx},\label{Novikov1}\\
	&u(x, 0)=u_{0}(x), \quad x \in \mathbb{R},\  t>0,\\
	&u_{0}(x)\to\kappa>0,\ x\to\pm\infty,\label{Novikov3}
\end{align}	
where    $u=u(x,t)$  is a real-valued function of   $x$ and $t$. By introducing the momentum variable $m=u-u_{x x}$,  the Novikov equation (\ref{Novikov1}) can be rewritten as the conversation law form
\begin{align}
	&(m^{2/3})_{t}+\left(u^2m^{2/3}\right)_x=0.\label{Novikov0}
\end{align}
The Novikov equation (\ref{Novikov1}) as a new integrable system was derived from the  classification of integrable
generalized Camassa-Holm equations of the form
\begin{align}
	(1-\partial_x^2)u_t=F(u,u_x,u_{xx},...)
\end{align}
possessing infinite hierarchies of higher symmetries.

The Novikov equation possesses a    scalar    Lax pair  involving  the third order derivative with respect to $x$,
which has been provided \cite{39,37}.
 Furthermore, by employing the prolongation algebra method,
Hone and Wang introduced  a $3\times 3$ matrix Lax pair and  established a bi-Hamiltonian structure for the Novikov equation (\ref{Novikov1}) \cite{HW29}. This Lax pair was used
to   explicitly construct    peakon solutions  on a zero background,  replicating  a feature characterizing
  the waves of great height-waves of largest amplitude that were exact solutions of the governing equations for water waves   \cite{HW29,CA1,CA2,CA3,TJF}.
   Hone et al. further  derived the
explicit formulas for multipeakon solutions of the Novikov equation (\ref{Novikov1})   \cite{HW28}.  Matsuno, using  the  Hirota bilinear method,
     presented parametric representations of smooth multisoliton solutions
  for  the Novikov equation (\ref{Novikov1}) on a nonzero constant  background \cite{M36}.
He also demonstrated that a smooth soliton converges to a peakon in the limit where the constant background approaches zero  while the velocity of the soliton is fixed.
Wu  et al.  obtained   $N$-soliton solutions for the Novikov equation  through Darboux transformations \cite{Wu6}.    Recently, Chang et al.
applied Pfaffian technique  to investigate multipeakons of the Novikov equation, establishing
a connection between the Novikov peakons and the finite Toda lattice of BKP type, as well as employing Hermite-Pade approximation to address the Novikov peakon problem
\cite{CH2018, CH2022}.  There exists   a unique global solution $u(x, t)$  of the Novikov equation (\ref{Novikov1}),
 such that  $u(x,t)\to0$ as $x\to\pm\infty$ for all $t>0$  \cite{32}. Boutet de Monvel et al.  developed the inverse scattering theory  to
 the Novikov equation  (\ref{Novikov1}) with a nonzero constant background.
 They proved that under a  transformation
\begin{align}
		u(x, t)\rightarrow \kappa \tilde{u}(x-\kappa^2t, \kappa^2t)+\kappa,
\end{align}
the  Cauchy problem  (\ref{Novikov1})--(\ref{Novikov3})  can be reduced into the following   Cauchy
problem on zero background
\begin{align}
	&(\tilde{m}^{2/3})_{t}+\left(\tilde{m}^{2/3}\left(u^{2}+2u\right)\right)_{x}=0, \  \tilde{m}=u-u_{x x}+1,\label{Novikov}\\
	&u(x,0)=u_0(x),\label{Novikov2}
\end{align}
where $u_0(x)$ satisfies the sign condition
\begin{equation*}
	u_0(x)- u_{0,xx}(x)+1>0.
\end{equation*}
 Building upon the above characteristics, a Riemann-Hilbert (RH) formalism for the Cauchy problem (\ref{Novikov})-(\ref{Novikov2}) has been established   \cite{RHP} .

The Novikov equation and  DP equation  possess  numerous common characteristics in their RH problem  and   face   some    difficulties.
One of the difficulties is the Lax pair associated with (\ref{Novikov}) has six spectral singularities at  $\varkappa_n=e^{\frac{n\pi i}{3}}$ for $n=1,\cdots,6$, which means that the corresponding RH problem also exhibits spectral singularities.
During the  large time asymptotic analysis  for  the   DP equation,
  Boutet de Monvel et al. considered  a row vector RH problem   to  avoid the impact of singularities \cite{Monvel3,Monvel2}.
  On the contrary, in the case of the Novikov equation, the situation is different. The solution of the similar row vector RH problem cannot be directly used to recover $e^{x(y,t)-y}$ \cite{RHP}. Recently,
  this difficulty was   overcome  by  the establishment of a small-norm RH problem near the singular points \cite{novYF}.
  Furthermore, the   large time asymptotic expansions    for the solution to the Cauchy problem (\ref{Novikov})-(\ref{Novikov2}) of
   the  Novikov equation
in  four different  space-time regions (See   Figure  \ref{result1})
   $${\rm I.} \ \xi<-1/8; \ {\rm II.} \ -1/8<\xi<0;  \ {\rm III.}\   0<\xi <1; \ \ {\rm IV.} \ \xi>1, \quad \xi:=y/t$$
were obtained with
   the  $\bar{\partial}$  nonlinear steepest approach, which was  originally   introduced by McLaughlin-Miller \cite{MandM2006,MandM2008}.   This method
  has been successfully applied to analyze the large time asymptotics and soliton resolution  of  integrable systems   \cite{DandMNLS,fNLS,Liu3,SandRNLS,YF3,WF}.
 The remaining question is: How to describe   the  asymptotics of the solution to  the Cauchy problem (\ref{Novikov})-(\ref{Novikov2})
in the transition zones ${\rm V.} \ \xi \approx  -1/8$ and ${\rm VI.} \  \xi \approx  1$ as illustrated in Figure  \ref{result1} ?

   The aim of the present work is  to  present  the large time asymptotics of the Novikov equation in the transition zone ${\rm V.} \ \xi \approx  -1/8$.
  The leading term of the asymptotic expansion  is influenced by the discrete spectrum  and the sub-leading term is in terms of the solutions  of the Painlev\'e  \uppercase\expandafter{\romannumeral2} equation. Our main result  is  stated as follows.

\begin{figure}
\begin{center}
\begin{tikzpicture}
\draw [ -latex ] (-4.2,0)--(4.5,0);
\draw [ -latex ](0,0)--(0,3.5);


	\fill[blue!10] (0,0) -- plot[dotted,thick,domain=0:1.52,smooth,variable=\x,blue] ({0.8*\x + 1.2*\x*\x},{0.866*\x + 0.5*\x*\x}) -- plot[dotted,thick,domain=0:1.12,smooth,variable=\x,blue] ({2*\x + 1.5*\x*\x},{0.5*\x + 0.866*\x*\x}) -- cycle;
	\fill[blue!10] (0,0)--(4.1,1.6)--(4.1,2.545)--(0,0);


   	\fill[blue!10] (0, 0) -- plot[dotted,thick,domain=0:1.25,smooth,variable=\x,blue] ({-1.2*\x - 1.6*\x*\x},{0.86*\x + 0.25*\x*\x}) -- plot[dotted,thick,domain=0:0.92,smooth,variable=\x,blue] ({-2.5*\x - 2*\x*\x},{0.12*\x + 0.866*\x*\x}) -- (0, 0) -- cycle;
   	\fill[blue!10] (0,0)--(-4,0.8)--(-4,1.48)--(0,0);

   \draw [ red  ](0,0 )--(4.1,2.05);
   \draw [ red  ](0,0 )--(-4,1.12);

\node    at (0,-0.3)  {$0$};
\node    at (5,0)  {y};
\node    at (0,3.8 )  {t};
\node  [below]  at (1.1,1.8) {\footnotesize $II$};
\node  [below]  at (-1.5,1.8) {\footnotesize $III$};
\node  [below]  at (2.5,0.6) {\footnotesize $I$};
\node  [below]  at (-3.3,0.45) {\footnotesize $ IV $};
\node  [below]  at  (3.7,2.3) {\footnotesize $ V $};
\node  [below]  at  (-3.6,1.41) {\footnotesize $ VI $};
\node  [below]  at (-4.8,1.5) {\footnotesize $ \xi=-1/8 $};
\node  [below]  at (4.7,2.3) {\footnotesize $ \xi=1 $};
\end{tikzpicture}
\end{center}
\caption{\footnotesize The different asymptotic regions
   of the  Novikov equation in the $(y,t)$-half plane, where $\xi=y/t$.   }
\label{result1}
\end{figure}
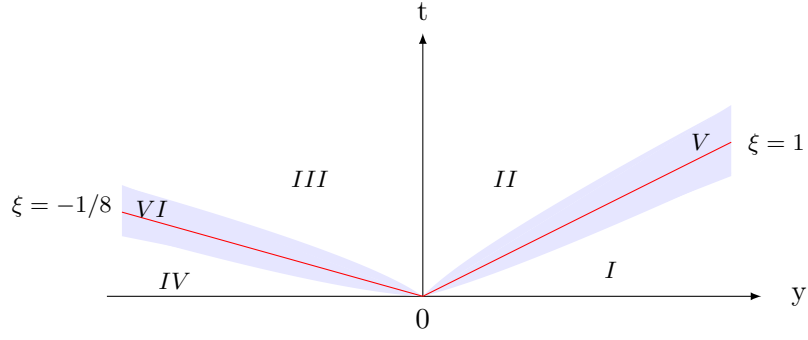

\begin{theorem}\label{last}   Let $u(x,t)=u(y(x,t),t)$ be the solution for  the Novikov equation  (\ref{Novikov})  with generic initial data   $u_0\in \mathcal{S}(\mathbb{R})$ and
 its associated scatting data $\left\lbrace  r(z),\left\lbrace \zeta_n,c_n\right\rbrace^{6N_0}_{n=1}\right\rbrace$. Let
  $u^\lozenge(y,t)$ be the $\mathcal{N}(\lozenge)$-soliton solution corresponding to the  scattering data
$\tilde{\mathcal{D}_\lozenge}=   \left\lbrace \zeta_n,c_nT^2(\zeta_n)\right\rbrace_{n\in\lozenge} $ shown in Corollary \ref{p1urxrsol}.  Then in the  transition   zone $|\xi+\frac{1}{8}|t^{2/3}<C$ with
$\xi=y/t$,  there exists a large constant $T_1$ such that for all $t>T_1$, we have
\begin{align}
	u(y,t)&= u^\lozenge(y,t;\tilde{\mathcal{D}}_\lozenge)\left( T_1(e^{\frac{\pi i}{6}})T_3(e^{\frac{\pi i}{6}})\right) ^{-1/2}-1\nonumber\\
	&+\frac{1}{2}\left( T_1(e^{\frac{\pi i}{6}})T_3(e^{\frac{\pi i}{6}})\right) ^{-1/2} f_{11}t^{-1/3}+\mathcal{O}(t^{-2/3+2\delta_1}),\label{p1consu} \\
	x(y,t)
	&=x^\lozenge(y,t;\tilde{\mathcal{D}}_\lozenge)+\frac{1}{2} \ln T_{13}(e^{\frac{\pi i}{6}})+ \frac{1}{2} f_{12} t^{-1/3} +\mathcal{O}(t^{-2/3+2\delta_1}), \label{p1constx}
\end{align}
where $u^\lozenge(y,t;\tilde{\mathcal{D}}_\lozenge)$ and $x^\lozenge(y,t;\tilde{\mathcal{D}}_\lozenge)$ are defined in Corollary \ref{p1urxrsol}, $T_{13}$, $f_{11}$, and $f_{12}$ are given by   \eqref{Tij}, \eqref{p1f11}, and \eqref{p1f12}
respectively, while functions $f_{11}$ and $f_{12}$ are related to the solution of   the Painlev\'{e} \uppercase\expandafter{\romannumeral2}  equation
\begin{equation}\label{p1pain2}
	v_{ss} = 2v^3 +sv, \quad s \in \mathbb{R},
\end{equation}
with asymptotics
\begin{equation}
v(s) \sim \left| r\left(\frac{\sqrt{7}+\sqrt{3}}{2}\right)\right| {\rm Ai}(s), \quad s \to -\infty.
\end{equation}

\end{theorem}

\begin{remark}\label{remar}
In the subsequent  paper,  we will provide  the large time asymptotic analysis of the Novikov equation in the    transition zone  ${\rm VI.} \  \xi \approx  1$.  For   this case, the critical points  and spectral singularities are the same, which  implies that singularities will emerge in the local model,
 requiring the use of alternative methods rather than the approach described in this paper.
\end{remark}

The organization of our paper is as follows:

  In Section \ref{sec2},   we quickly review  some basic results, especially the construction of a  RH  problem for $M(z)$   related to the Cauchy problem \eqref{Novikov1}-\eqref{Novikov3},  which will be used  to analyze  the large time asymptotics  of the Novikov equation.  For details,   refer to     \cite{RHP,novYF}.

  In Section \ref{sec3},   we convert the original RH problem for  $M(z)$ to a   RH problem for  $M^{(2)}(z)$, whose jump contour can be opened with
    different factorizations of the jump matrix \eqref{jumpv0} according to the sign of the phase functions \eqref{theta}.

In Section \ref{sec4}, we focus on the large time asymptotic analysis   in the  transition zone $|\xi+1/8|t^{2/3}<C$ with the following steps.
First of all,  introducing a matrix-valued  function  $\mathcal{R}^{(3)}(z)$  to make a continuous extension of  $M^{(2)}(z)$
into    a hybrid  $\bar{\partial}$-RH problem  for  $M^{(3)}(z)$, which can be
  decomposed     into a
 pure RH   problem  for  $M^{R}(z)$ and a  pure $\bar{\partial}$-problem for  $M^{(4)}(z)$.
 We observe that the contribution to $M^{R}(z)$ arises from three  distinct components:
 \begin{itemize}
 \item The first  contribution originates from the discrete spectrum, where  a  modified reflectionless RH problem for $M^{O}(z)$  is  constructed  in Subsection \ref{sec6}.
 \item The second  contribution   comes from jump contours near the critical points generated by phase points after colliding. The corresponding local parametrix for $M^L(z)$
can be   approximated by  the modified Painlev\'{e} \uppercase\expandafter{\romannumeral2}  RH problem near the  critical points in Subsection \ref{p1loc}.
 \item The third contribution  generates from a pure jump RH problem near the singularities in Subsection   \ref{rhsing}, which has an order  of $\mathcal{O}(t^{-1})$.
The residual error function results from a   small norm RH problem outside the neighborhood of critical points in Subsection \ref{sec7}.

\item  In Subsection \ref{sec8}, we analyze
 the    contribution    associated with  the $\bar{\partial}$-problem  for $M^{(4)}(z)$.

\item Finally, in  Subsection \ref{p1sec9}, based on  the result obtained above,we provide the proof of Theorem \ref{last}.

 \end{itemize}

\section {Inverse  Scattering   and  RH Problem}\label{sec2}

In this section, we  state some  main
results on  the   inverse scattering transform and the RH problem  associated with the Cauchy problem
 (\ref{Novikov1})--(\ref{Novikov3}).   The details   can be found in
\cite{RHP,novYF}.

\subsection{The Lax pair and spectral analysis}

\quad The Novikov equation (\ref{Novikov})      admits the Lax pair   \cite{HW29,RHP}
\begin{equation}
\breve{\Phi}_x = \breve{X} \breve{\Phi},\hspace{0.5cm}\breve{\Phi}_t =\breve{T} \breve{\Phi}, \label{lax0}
\end{equation}
where $k$ is a  spectral parameter, $\breve{\Phi}=\breve{\Phi} (k;x,t) $ is a $3 \times 3 $ matrix valued eigenfunction, the  matrices $\breve{X} $ and $\breve{T} $ are defined by
\begin{align}
	&\breve{X}=\left(\begin{array}{ccc}
		0&	k\tilde{m} & 1 \\
		0&0 & k\tilde{m}\\
		1&0&0
	\end{array}\right),\nonumber\\
	&\breve{T}=\left(\begin{array}{ccc}
		-(u+1)u_x+\frac{1}{3k^2}&	\frac{u_x}{k}-(u^2+2u)k\tilde{m} & u_x^2+1 \\
		\frac{u+1}{k}&-\frac{2}{3k^2} & -\frac{u_x}{k}-(u^2+2u)k\tilde{m}\\
		-u^2-2u & \frac{u+1}{k} &(u+1)u_x+\frac{1}{3k^2}
	\end{array}\right).
 \nonumber
\end{align}

Denote
$$k^2(z)  =\frac{1}{3\sqrt{3}}\left(z^3+\frac{1}{z^3} \right), \ \omega=e^{\frac{2i\pi}{3}}, $$
then the following  algebraic equation
\begin{align}
	\lambda^3(z)-\lambda (z)=k^2(z)
\end{align}
admits    three  roots  in the form
\begin{align}
	&\lambda_j(z) =\frac{1}{\sqrt{3}}\left(\omega^j z+\frac{1}{\omega^j z} \right), \ j=1,2,3.\label{lambda}
\end{align}
Moreover, $k(\kappa_n)=0$, for $\kappa_n=e^{\frac{i\pi}{6}+\frac{i\pi(n-1)}{3}}$, $n=1,\cdots,6$.

%
In order to control the large $k$ behavior of the solutions of the  Lax pair (\ref{lax0}), we  define
\begin{align}
	&D(x,t)=\text{diag}\{q,q^{-1},1\},\quad q:= \tilde{m}^{1/3}(x,t),\nonumber\\
	&P(z)=\left(\begin{array}{ccc}
		\lambda_1^2(z)&	\lambda_2^2(z) & \lambda_3^2(z) \\
		k (z) &k(z) & k(z)\\
			\lambda_1(z)&	\lambda_2(z) & \lambda_3(z)
	\end{array}\right),\label{P}\\
&P^{-1}(z)=\left(\begin{array}{ccc}
	\frac{1}{3\lambda_1^2(z)-1}&	0 & 0 \\
	0&\frac{1}{3\lambda_2^2(z)-1} & 0\\
	0&	0 & \frac{1}{3\lambda_3^2(z)-1}
\end{array}\right)\left(\begin{array}{ccc}
1&	\frac{z}{\lambda_1(z)} & \lambda_1(z) \\
1&\frac{z}{\lambda_2(z)} & \lambda_2(z)\\
1&	\frac{z}{\lambda_3(z)} & \lambda_3(z)
\end{array}\right)\label{P-1}.
\end{align}
Then the new function
\begin{align}
	\Phi = P^{-1}(z)D^{-1}(x,t)\breve{\Phi}  \label{ed22}
\end{align}
 satisfies the Lax pair
\begin{align}
	& \Phi_x-q^2\Lambda(z)\Phi=U\Phi,\\
	&\Phi_t+\left[(u^2+2u)q^2\Lambda(z)-A(z) \right] \Phi=V\Phi,\label{laxphi}
\end{align}
where
\begin{align}
	&\Lambda(z)=\text{diag}\{\lambda_1(z),\lambda_2(z),\lambda_3(z)\},  \ \ A(z)=\frac{1}{3k^2}+\Lambda(z)^{-1},\nonumber\\
	&U= U_1U_2, \ \ V= U_1(V_1+V_2\Lambda),\label{laxU}\\
	&U_1=\text{diag} \left\{\frac{1}{3\lambda_1^2-1},\frac{1}{3\lambda_2^2-1},\frac{1}{3\lambda_3^2-1}\right\}, \nonumber\\
	&U_2=\left(\begin{array}{ccc}
		c_2\lambda_1&	c_1(\lambda_1\lambda_2-\lambda_2^2)+c_2\lambda_2 & c_1(\lambda_1\lambda_3-\lambda_3^2)+c_2\lambda_3 \\
		c_1(\lambda_1\lambda_2-\lambda_1^2)+c_2\lambda_1& c_2\lambda_2 & c_1(\lambda_3\lambda_2-\lambda_3^2)+c_2\lambda_3\\
		c_1(\lambda_1\lambda_3-\lambda_3^2)+c_2\lambda_1&	c_1(\lambda_3\lambda_2-\lambda_2^2)+c_2\lambda_2 & c_2\lambda_3
	\end{array}\right), \nonumber
\end{align}
with  $c_1=\frac{q_x}{q}$ and $c_2=q^{-2}-q^2$. While  $V_1$ has same  form of $U_1$ with $c_1$ and $c_2$ replaced by $c_3=-(u^2+2u)\frac{q_x}{q}$ and $c_4=(u^2+2u)q^2+\frac{u_x^2+1}{q^2}-1$ , respectively; $V_2$ is given by
\begin{align}
	V_2=[V_2^{(jl)}]_{3\times3},\quad V_2^{(jl)}= c_5\left(\frac{1}{\lambda_l}-\frac{1}{\lambda_j} \right) +c_6\left( \frac{\lambda_j}{\lambda_l}+\frac{\lambda_l}{\lambda_j} \right).
\end{align}
with   $c_5=\frac{u_x}{q}$, $c_6=(u+1)q-1$,

Introduce a new eigenfunction $\mu=\mu(z;x,t)$ satisfying
\begin{align}
	\mu=\Phi  e ^{-Q},\label{transmu}
\end{align}
where $Q=Q(z;x,t)$ is a $3\times3$ diagonal function
$$Q=y(x,t)\Lambda(z)+tA(z),$$
with
\begin{align}
&Q_x=q^2\Lambda, \quad Q_t=-\left(u^2+2u \right) q^2\Lambda-A, \nonumber\\
	&y(x,t)=x-\int_{x}^{\infty}\left( q^2(s,t)-1\right) ds\label{y}.
\end{align}
Then the Lax pair (\ref{laxphi}) is changed to
\begin{align}
	\mu_x-[Q_x,\mu]=U\mu,\quad \mu_t-[Q_t,\mu]=V\mu,\label{laxmu}
\end{align}
whose  solutions  satisfy   the  Fredholm integral equations
\begin{equation}
	\mu^{\pm}(z;x,t)=I-\int_{x}^{\pm \infty}e^{-\hat{\Lambda}(z)\int_{x}^{s}q^2(v,t)dv}[U\mu_\pm(s,t;z)]ds\label{intmu}.
\end{equation}
We  define six  rays at $z=0$,
$$\Sigma =\cup_{n=1}^6 L_n, \ \  \ L_n=e^{\frac{\pi(n-1)i}{3}}\mathbb{R}^+, \ n=1,\cdots,6, $$
which  divide the complex plane $ \mathbb{C}$   into  six open cones
$$S_n=\{z\in\mathbb{C};\arg z\in(   {(n-1)\pi}/{3}, { n \pi }/{3})\}, \   n=1, \cdots, 6,$$
see Figure \ref{figC}.   Denote the matrix
$$\mu^{\pm}=\left(  \mu^{\pm}_1, \mu^{\pm}_2, \mu^{\pm}_3 \right), $$
where  the scripts $ 1, 2$ and $3$ denote
the first, second and third column of $ \mu^{\pm}(z)$  respectively.
Then  from  (\ref{intmu}),   we can show that  $\left(\mu^{+}_1, \mu^{+}_2, \mu^{+}_3 \right)$ is  analytical  in   the domains
$$\left(\bar{S}_1\cup \bar{S}_2,\ \bar{S}_5\cup \bar{S}_6,\ \bar{S}_3\cup \bar{S}_4 \right),$$
while  $\left(\mu^{-}_1, \mu^{-}_2, \mu^{-}_3 \right)$ is  analytical  in the domains
$$\left(\bar{S}_4\cup \bar{S}_5,\ \bar{S}_3\cup \bar{S}_2,\ \bar{S}_1\cup \bar{S}_6 \right).$$
 Here, $\bar{S}_j$ denotes the  closure of $S_j$ for $j=1,\cdots,6$, respectively.
The initial points of integration $\infty_{j,l}$ are specified as follows for each matrix entry $(j,l)$ for $j,l=1,2,3$:
\begin{align}
	\infty_{j,l}=\left\{ \begin{array}{ll}
		+\infty,   &\text{if Re}\lambda_j\geq \text{Re}\lambda_l,\\[12pt]
		-\infty  , &\text{if Re}\lambda_j< \text{Re}\lambda_l.\\
	\end{array}\right.
\end{align}
Then (\ref{intmu}) can be rewritten as a system of   Fredholm integral equations
\begin{equation}
	\mu_{jl}(z;x,t)=I_{jl}-\int_{x}^{ \infty_{jl}}e^{-\int_{x}^{s}q^2(v,t)dv (\lambda_j(z)-\lambda_l(z))}[U\mu(s,t;z)]_{jl}ds\label{intmujl}.
\end{equation}

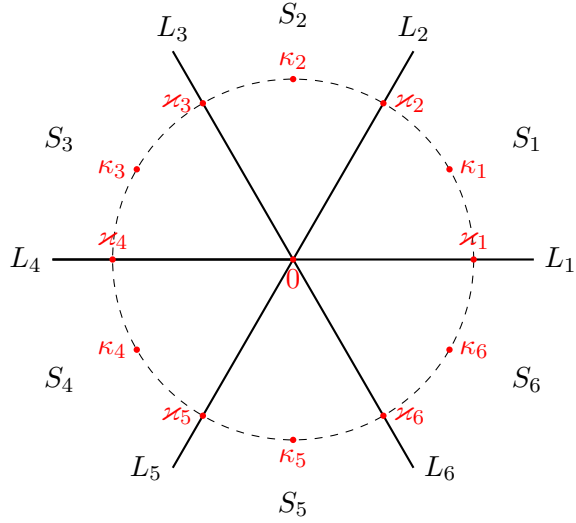
\begin{figure}[t]
	\centering
	\begin{tikzpicture}[scale=0.8]
		\draw[thick](-4,0)--(4,0)node[right]{$L_1$};
		\draw[thick](0,0)--(2,3.464)node[above]{$L_2$};
		\draw[dashed] (3,0) arc (0:360:3);
		\draw[thick](0,0)--(-2,3.464)node[above]{$L_3$};
		\draw[thick](0,0)--(-4,0)node[left]{$L_4$};
		\draw[thick ](0,0)--(2,-3.464)node[right]{$L_6$};
		\draw[thick](0,0)--(-2,-3.464)node[left]{$L_5$};
		\coordinate (A) at (1.5,2.598);
		\coordinate (B) at (1.5,-2.598);
		\coordinate (C) at (-1.5,2.598);
		\coordinate (D) at (-1.5,-2.598);
		\coordinate (E) at (3,0);
		\coordinate (F) at (-3,0);
		\coordinate (G) at (-2.598,1.5);
		\coordinate (H) at (-2.598,-1.5);
		\coordinate (J) at (2.598,1.5);
		\coordinate (K) at (2.598,-1.5);
		\coordinate (L) at (0,3);
		\coordinate (M) at (0,-3);
		\coordinate (v) at (0,0);
		\fill[red] (v) circle (1.5pt) node[below] {$0$};
		\fill[red] (A) circle (1.5pt) node[right] {$\varkappa_2$};
		\fill[red] (B) circle (1.5pt) node[right] {$\varkappa_6$};
		\fill[red] (C) circle (1.5pt) node[left] {$\varkappa_3$};
		\fill[red] (D) circle (1.5pt) node[left] {$\varkappa_5$};
		\fill[red] (E) circle (1.5pt) node[above] {$\varkappa_1$};
		\fill[red] (F) circle (1.5pt) node[above] {$\varkappa_4$};
		\fill[red] (G) circle (1.5pt) node[left] {$\kappa_3$};
		\fill[red] (H) circle (1.5pt) node[left] {$\kappa_4$};
		\fill[red] (J) circle (1.5pt) node[right] {$\kappa_1$};
		\fill[red] (K) circle (1.5pt) node[right] {$\kappa_6$};
		\fill[red] (L) circle (1.5pt) node[above] {$\kappa_2$};
		\fill[red] (M) circle (1.5pt) node[below] {$\kappa_5$};
		\coordinate (a) at (-3.464,2);
		\coordinate (s) at (-3.464,-2);
		\coordinate (d) at (3.464,2);
		\coordinate (f) at (3.464,-2);
		\coordinate (g) at (0,2);
		\coordinate (h) at (0,-2);
		\fill (a) circle (0pt) node[left] {$S_3$};
		\fill (s) circle (0pt) node[left] {$S_4$};
		\fill (d) circle (0pt) node[right] {$S_1$};
		\fill (f) circle (0pt) node[right] {$S_6$};
       \node  [above]  at (0,3.7) {$S_2$};
       \node  [below]  at (0,-3.7) {$S_5$};
	\end{tikzpicture}
	\caption{\footnotesize  The $z$-plane is divided into six analytical  domains $S_n, n=1,\cdots,6,$ by
 six   rays  $L_n$.  The red points  $\varkappa_n$ and $\kappa_n$    are spectral   singularities.}
	\label{figC}
\end{figure}
 It was shown  that the eigenfunction  $\mu(z)$ defined by (\ref{transmu}) has the following properties  \cite{RHP}.

\begin{Proposition}\label{propermu}
	 The	equations (\ref{intmu}) uniquely define a $3\times3$-matrix valued solutions $\mu (z) := \mu(z;x,t)$ of (\ref{laxmu})
	with the following properties:
\begin{itemize}
	\item [{\rm (1)}] det $\mu(z)$=1.

\item [{\rm (2)}]   $\mu(z)$ is piecewise meromorphic with respect to $ \mathbb{C}\setminus\Sigma $, as a function of the spectral parameter $z$.

\item [{\rm (3)}]  $\mu(z)$ obeys the  symmetries
	\begin{align}
\mu(z)=\Gamma_1\overline{\mu(\bar{z})}\Gamma_1=\Gamma_2\overline{\mu(\omega^2\bar{z})}\Gamma_2=\Gamma_3\overline{\mu(\omega\bar{z})}\Gamma_3=\overline{\mu(\bar{z}^{-1})},\nonumber
	\end{align}
where
	\begin{align}\label{gamma123}
		\Gamma_1=\left(\begin{array}{ccc}
			0&	1 & 0 \\
			1&0 & 0\\
			0&	0 & 1
		\end{array}\right),\quad \Gamma_2=\left(\begin{array}{ccc}
		0&	0 & 1 \\
		0&1 & 0\\
		1&	0 & 0
	\end{array}\right),\quad 					\Gamma_3=\left(\begin{array}{ccc}
	1&	0 & 0 \\
	0&0 & 1\\
	0&	1 & 0
\end{array}\right).
	\end{align}
\item [{\rm (4)}]  $\mu(z)$ has pole singularities at $\varkappa_n=e^{\frac{n\pi i}{3}}$ with $\mu=\mathcal{O}(\frac{1}{z-\varkappa_n})$ as $z\to\varkappa_n$.

\item[{\rm (5)}]  $\mu(z) \to I$ as $z\to\infty$, and for $z\in\mathbb{C}\setminus\Sigma $, $\mu(z)$ is bounded as $x\to-\infty$ and $\mu (z) \to I$ as $x\to+\infty$.
\end{itemize}
\end{Proposition}

\begin{remark}
From the symmetries in Proposition \ref{propermu}(3), it follows that the values of $\mu$ at $z$ and at $\omega z$ are related by
\begin{equation}\label{mwzc}
\mu(\omega z) = C^{-1} \mu(z) C, \ \ {\rm where}\ \  C=\left(\begin{array}{ccc}
			0&	0 &1 \\
			1&0 & 0\\
			0&	1 & 0
		\end{array}\right).
\end{equation}
\end{remark}

Denote  $\mu_{\pm}(z)$ as   the limiting values of $\mu (z')$ as $z'\to z $  from the
positive or negative side of $L_n$,    then they   are related as follows
\begin{align}
	\mu_{+}(z)=\mu_{-}(z)e^Q V_n(z)e^{-Q},\quad z\in L_n,\ n=1,\cdots,6,
\end{align}
where $V_n(z)$ only depends on $z$ and  is completely determined by $u(x, 0)$, i.e., by the initial data for the Cauchy problem (\ref{Novikov}). Take  $L_1=\mathbb{R}^+$ and $L_4=\mathbb{R}^-$ as an example,  $V_n$ for $n=1,4$ has a   special matrix structure
\begin{align}
	V_n(z)= \left(\begin{array}{ccc}
		1&	-r_\pm(z) & 0 \\
		\bar{r}_\pm(z)& 1-|r_\pm(z)|^2 & 0\\
		0&	0 & 1
	\end{array}\right),\quad z\in\mathbb{R}^\pm.	
\end{align}
where  $r_\pm(z)$ are  single scalar
functions,  with $r_\pm(z)\in L^{\infty}(\mathbb{R}^\pm)$, and $r_\pm(z)=\mathcal{O}(z^{-1})$ as $z\to\pm\infty$. The symmetry of $\mu (z) $ gives that $r_\pm(z)=\overline{r_\pm(z^{-1})}$,
therefore it also has $r_\pm(z)\in L^{2}(\mathbb{R}^\pm)$ and $\lim_{z\to 0^\pm}r_\pm(z)=0$.
Moreover, the singularities at $\pm1$ give that $r_\pm(\pm1)=0$. So we define the  \emph{reflection coefficient}
\begin{align}
	r(z)=\left\{ \begin{array}{ll}
		r_\pm(z),   & z\in \mathbb{R}^\pm,\\[12pt]
		0  , & z=0.\\
	\end{array}\right.
\end{align}
Then $r\in L^\infty(\mathbb{R})\cap L^2(\mathbb{R})$ and $r(z)=\mathcal{O}(z^{-1})$ as $z\to\infty$.  In references \cite{RHP, BC1984},  it was shown  that there exist at most a finite
number of simple poles $z_n$ of $\mu(z)$ lying in $S_1\cap\{z\in\mathbb{C};\ |z|>1\}$ and $w_m$  lying in $S_1\cap\{z\in\mathbb{C};\ |z|=1\}$.
And there are no poles except $\pm1$, $\pm\omega$ and $\pm\omega^2$ on the contour $\Sigma $.

 To differentiate this two types of poles,  we denote them as $z_n$, $z_n^A$ and $w_m$, $w_m^A$, respectively.
Denote $N_1$, $N_1^A$, $N_2$ and $N_2^A$ as the number of $z_n$, $z_n^A$, $w_m$,  and $w_m^A$, respectively. The  symmetries  of $\mu(z)$  imply  $\bar{z}_n^{-1}$ and $\frac{1}{\bar{z}_n^A}$ are also the poles of $\mu(z)$ in $S_1$.  It is convenient to define  $\zeta_n=z_n$, and $\zeta_{n+N_1}=\bar{z}_n^{-1}$ for $n=1,\cdot\cdot\cdot,N_1$; $\zeta_{m+2N_1}=w_m$  for $m=1,\cdot\cdot\cdot,N_2$; $\zeta_{j+2N_1+N_2}=z_j^A$, and $\zeta_{j+2N_1+N_2+N_1^A}=\frac{1}{\bar{z}_n^A}$ for $j=1,\cdot\cdot\cdot,N_1^A$; $\zeta_{m+2N_1+2N_1^A+N_2}=w_m^A$  for $m=1,\cdot\cdot\cdot,N_2^A$. For the sake of brevity, let
\begin{align*}
	N_0=2N_1+2N_1^A+N_2+N_2^A.
\end{align*}
Moreover,  $\omega \zeta_n$, $\omega^2 \zeta_n$, $\bar{\zeta}_n$, $\omega\bar{\zeta}_n$,   $\omega^2\bar{\zeta}_n$  are also poles of $\mu(z)$ in $S_j$, $j=2,\cdots,6$. For convenience, let $\zeta_{n+N_0}=\omega\bar{\zeta}_n$, $\zeta_{n+2N_0}=\omega\zeta_n,$ $\zeta_{n+3N_0}=\omega^2\bar{\zeta}_n$, $\zeta_{n+4N_0}=\omega^2\zeta_n$ and  $\zeta_{n+5N_0}=\bar{\zeta}_n$ for $n=1,\cdots,N_0$.
 Therefore,  define the discrete spectrum  $\mathcal{Z}$ as
\begin{equation}
	\mathcal{Z}=\left\{ \zeta_n\right\}_{n=1}^{6N_0}, \label{spectrals}
\end{equation}
with $\zeta_n\in S_1$ and $\bar{\zeta}_n\in S_6$ whose distribution on the $z$-plane   is shown  in Figure \ref{fig:figure1}.
\begin{figure}[H]
	\centering
	\begin{tikzpicture}[node distance=2cm]
		\draw[thick](-4,0)--(4,0)node[right]{$L_1$};
		\draw[thick](0,0)--(2,3.464)node[above]{$L_2$};
		\draw[thick](0,0)--(-2,3.464)node[above]{$L_3$};
		\draw[thick](0,0)--(-4,0)node[left]{$L_4$};
		\draw[thick](0,0)--(2,-3.464)node[right]{$L_6$};
		\draw[thick](0,0)--(-2,-3.464)node[left]{$L_5$};
		\draw[thick] (2,0) arc (0:360:2);
		\coordinate (A) at (2.2,2.2);
		\coordinate (B) at (2.2,-2.2);
		\coordinate (C) at (-0.8,3);
		\coordinate (D) at (-0.8,-3);
		\coordinate (E) at (0.9,0.9);
		\coordinate (F) at (0.9,-0.9);
		\coordinate (G) at (-3,0.8);
		\coordinate (H) at (-3,-0.8);
		\coordinate (J) at (1.7570508075688774,0.956);
		\coordinate (K) at (1.7570508075688774,-0.956);
		\coordinate (L) at (-1.7570508075688774,0.956);
		\coordinate (M) at (-1.7570508075688774,-0.956);
		\coordinate (a) at (0,2);
		\fill[blue] (a) circle (1pt) node[above] {$\omega \bar{w}_m$};
		\coordinate (s) at (0,-2);
		\fill[blue] (s) circle (1pt) node[below] {$\omega ^2w_m$};
		\coordinate (d) at (-0.33,1.23);
		\fill[red] (d) circle (1pt) node[right] {$\frac{\omega}{z_n} $};
		\coordinate (f) at (-0.33,-1.23);
		\fill[red] (f) circle (1pt) node[right] {$\frac{\omega^2}{\bar{z}_n}$};
		\coordinate (g) at (-1.23,0.33);
		\fill[red] (g) circle (1pt) node[right] {$\frac{\omega}{\bar{z}_n} $};
		\coordinate (h) at (-1.23,-0.33);
		\fill[red] (h) circle (1pt) node[right] {$\frac{\omega^2}{z_n}$};
		\fill[red] (A) circle (1pt) node[right] {$z_n$};
		\fill[red] (B) circle (1pt) node[right] {$\bar{z}_n$};
		\fill[red] (C) circle (1pt) node[left] {$\omega \bar{z}_n$};
		\fill[red] (D) circle (1pt) node[right] {$\omega^2 z_n$};
		\fill[red] (E) circle (1pt) node[right] {$\frac{1}{\bar{z}_n}$};
		\fill[red] (F) circle (1pt) node[right] {$\frac{1}{z_n}$};
		\fill[red] (G) circle (1pt) node[left] {$\omega z_n$};
		\fill[red] (H) circle (1pt) node[left] {$\omega^2\bar{z}_n$};
		\fill[blue] (J) circle (1pt) node[right] {$w_m$};
		\fill[blue] (K) circle (1pt) node[right] {$\bar{w}_m$};
		\fill[blue] (L) circle (1pt) node[left] {$\omega w_m$};
		\fill[blue] (M) circle (1pt) node[below] {$\omega^2\bar{w}_m$};
		\coordinate (a) at (-3.464,2);
		\coordinate (s) at (-3.464,-2);
		\coordinate (d) at (3.464,2);
		\coordinate (f) at (3.464,-2);
		\coordinate (g) at (0,3.5);
		\coordinate (h) at (0,-3.5);
		\fill (a) circle (0pt) node[left] {$S_3$};
		\fill (s) circle (0pt) node[left] {$S_4$};
		\fill (d) circle (0pt) node[right] {$S_1$};
		\fill (f) circle (0pt) node[right] {$S_6$};
		\fill (g) circle (0pt) node[above] {$S_2$};
		\fill (h) circle (0pt) node[below] {$S_5$};
	\end{tikzpicture}
	\caption{\footnotesize Distribution of the discrete spectrum $\mathcal{Z}$ in the $z$-plane. }
	\label{fig:figure1}
\end{figure}
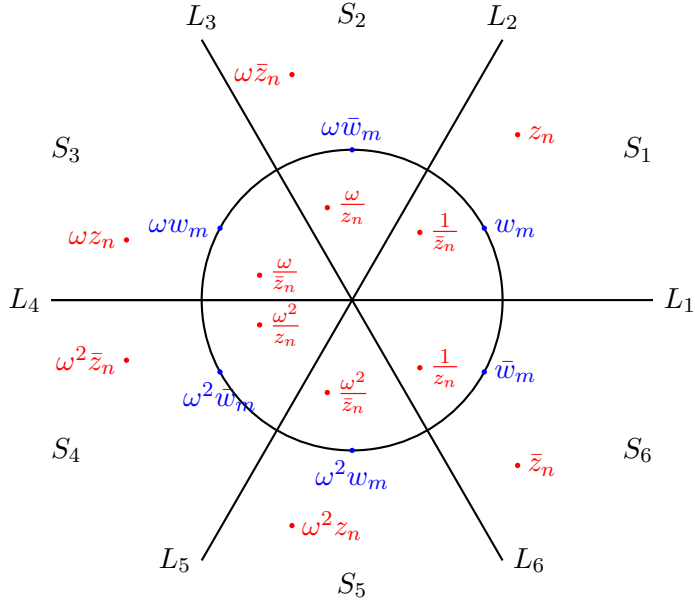
As shown in  \cite{RHP},  denote the  \emph{norming constant}    $c_n$   and    residue conditions as
\begin{align}
	&\res_{z=\zeta_n}\mu(z) =\lim_{z\to \zeta_n}\mu(z) e^Q\left(\begin{array}{ccc}
		0&	-c_n & 0 \\
		0& 0 & 0\\
		0&	0 & 0
	\end{array}\right)e^{-Q},
\end{align}
for $n=1,\cdots,2N_1+
N_2$ and
\begin{align}
	&\res_{z=\zeta_n}\mu(z)=\lim_{z\to \zeta_j}\mu(z) e^Q\left(\begin{array}{ccc}
		0&	0 & 0 \\
		0& 0 & -c_{n+2N_1+N_2}\\
		0&	0 & 0
	\end{array}\right)e^{-Q},\label{resrelation1}
\end{align}
for $n=1+2N_1+N_2,\cdots,N_0$.
In addition,  the collection
$\sigma_d= \left\lbrace \zeta_n,C_n\right\rbrace^{6N_0}_{n=1}$
is called the \emph{scattering data} with $C_n=c_n$, $C_{n+N_0}=\omega\bar{c}_n$, $C_{n+2N_0}=\omega c_n,$ $C_{n+3N_0}=\omega^2\bar{c}_n$, $C_{n+4N_0}=\omega^2c_n$ and  $C_{n+5N_0}=\bar{c}_n$ for $n=1,\cdots,N_0$.

 To deal with our following work,
we assume our initial data satisfies that $u_0\in\mathcal{S}(\mathbb{R})$ to generate  generic scattering data such  that $\mu(z)$  has no the poles
 on $L_n\setminus\{\varkappa_n\}$, $n=1,\cdots,6$  and  at the  point  $z= e^{\frac{\pi i}{6}}$.
For the reflection coefficient $r(z)$, we have the following proposition \cite{novYF}.
\begin{Proposition}\label{pror}
	If the initial data $u_0 (x) \in  \mathcal{S}(\mathbb{R})$, then $r(z)$ belongs to $\mathcal{S}(\mathbb{R})$. There exist   fixed  constants $C_{1,r}>0$ and   $C_{2,r}$
such that  if $u_0- u_{0,xx}> C_{2,r}>-1$, $ \parallel u_0- u_{0,xx}\parallel_{L^1}<C_{1,r} $ and $ \parallel u_0\parallel_{W^{3,1} \cup W^{3,\infty}} <C_{1,r} $,   then $|r(z)|<1$ for $z\in\mathbb{R}$.
Especially, $|r(\pm1)|=0$.
\end{Proposition}

\subsection{An RH characterization}

\quad
We replace the  variable  $ x $ with $ y $  defined by  (\ref{y}).
The price to pay for this is that the solution of the initial problem can be given only implicitly,
or parametrically.   Let $\xi=\frac{y}{t}$.
By the definition of the new scale $y(x, t)$, we denote the  phase function
\begin{equation}
	\theta_{jl}(z)=-i\left[ \xi (\lambda_j(z)-\lambda_l(z))+\left(\frac{1}{\lambda_j(z)}-\frac{1}{\lambda_l(z)} \right) \right] .\label{theta}
\end{equation}
Especially,
\begin{equation}\label{deftheta12}
	\theta_{12}(z)=\sqrt{3}\left(z-\frac{1}{z} \right) \left[ \xi-\frac{1}{z^2-1+z^{-2}} \right], \quad
\end{equation}
with $\theta_{23}(z)=\theta_{12}(\omega z),\text{ and }\theta_{31}(z)=\theta_{12}(\omega^2 z)$.

Define
\begin{equation}
	M(z;y,t):= \mu(z;x(y,t),t),
\end{equation}
which solves the following RH problem.
\begin{RHP}\label{RHP1}
	 Find a matrix-valued function $	M(z):= M(z;y,t)$ which satisfies
    \begin{itemize}
	
	\item  $M(z)$ is meromorphic in $\mathbb{C}\setminus \Sigma $ and has finite single poles.
	
	\item
$M(z)=\Gamma_1\overline{M(\bar{z})}\Gamma_1=\Gamma_2\overline{M(\omega^2\bar{z})}\Gamma_2=\Gamma_3\overline{M(\omega\bar{z})}\Gamma_3=\overline{M(\bar{z}^{-1})}.$

	\item $M(z)$ has continuous boundary values $M_\pm(z)$ on $\Sigma$, and
	\begin{equation}
		M_+(z)=M_-(z)V (z),\hspace{0.3cm} z \in \Sigma, \label{jumpv}
	\end{equation}
	where
	\begin{align} \label{jumpv0}
		V (z)=\begin{cases}
   \left(\begin{array}{ccc}
			1&	-r(z)e^{it\theta_{12}} & 0 \\
			\bar{r}(z)e^{-it\theta_{12}} & 1-|r(z)|^2 & 0\\
			0&	0 & 1
		\end{array}\right), \ z\in L_1,\\	
	\left(\begin{array}{ccc}
		1&	0 & 0 \\
		0 & 1 & -r(\omega z)e^{it\theta_{23}}\\
		0&	\bar{r}(\omega z)e^{-it\theta_{23}} & 1-|r(\omega z)|^2
	\end{array}\right), \ z\in L_2,\\
	\left(\begin{array}{ccc}
	1-|r(\omega^2 z)|^2&	0 & \bar{r}(\omega^2 z)e^{it\theta_{13}} \\
	0 & 1 & 0\\
	-r(\omega^2 z)e^{-it\theta_{13}}	&	0 & 1
	\end{array}\right), \ z\in L_3,\\
    \left(\begin{array}{ccc}
		1&	-r(z)e^{it\theta_{12}} & 0 \\
		\bar{r}(z)e^{-it\theta_{12}} & 1-|r(z)|^2 & 0\\
		0&	0 & 1
	\end{array}\right), \ z\in L_4,\\	
    \left(\begin{array}{ccc}
	1&	0 & 0 \\
	0 & 1 & -r(\omega z)e^{it\theta_{23}}\\
	0&	\bar{r}(\omega z)e^{-it\theta_{23}} & 1-|r(\omega z)|^2
\end{array}\right), \ z\in L_5,\\
    \left(\begin{array}{ccc}
	1-|r(\omega^2 z)|^2&	0 & \bar{r}(\omega^2 z)e^{it\theta_{13}} \\
	0 & 1 & 0\\
	-r(\omega^2 z)e^{-it\theta_{13}}	&	0 & 1
\end{array}\right),  \ z\in L_6.
     \end{cases}
	\end{align}
	
	\item $M(z) = I+\mathcal{O}(z^{-1}),\hspace{0.5cm}z \rightarrow \infty$.
	
	\item  As   $z\to\varkappa_n =e^{\frac{i\pi(n-1)}{3}}$, $n = 1,\cdots,6$,
 the  limit  of $M(z)$   has pole   singularities
	\begin{align}
		&M(z)=\frac{1}{z\mp1}\left(\begin{array}{ccc}
			\alpha_\pm &	\alpha_\pm & \beta_\pm \\
			-\alpha_\pm & -\alpha_\pm & -\beta_\pm\\
			0	&	0 & 0
		\end{array}\right)+\mathcal{O}(1),\ z\to\pm 1,\label{asyM1}\\
	&M(z)=\frac{\pm \omega^2}{z\mp\omega^2}\left(\begin{array}{ccc}
		0 &	0 &  0\\
	\beta_\pm	 & \alpha_\pm &\alpha_\pm \\
		-\beta_\pm	&	-\alpha_\pm & -\alpha_\pm
	\end{array}\right)+\mathcal{O}(1),\ z\to\pm \omega^2,\\
	&M(z)=\frac{\pm \omega}{z\mp\omega}\left(\begin{array}{ccc}
	-\alpha_\pm &	-\beta_\pm & -\alpha_\pm\\
	0	 & 0 &0 \\
	\alpha_\pm &	\beta_\pm & \alpha_\pm
	\end{array}\right)+\mathcal{O}(1),\ z\to\pm \omega\label{asymo},
	\end{align}
	with $\alpha_\pm=\alpha_\pm(y,t)=-\bar{\alpha}_\pm$, $\beta_\pm=\beta_\pm(y,t)=-\bar{\beta}_\pm$ and $M^{-1}(z)$ has same specific
	matrix structure  with $\alpha_\pm$, $\beta_\pm$ replaced by $\tilde{\alpha}_\pm$, $\tilde{\beta}_\pm$. Moreover, $\left( \alpha_\pm,\ \beta_\pm\right) \neq0$ iff $\left( \tilde{\alpha}_\pm,\ \tilde{\beta}_\pm\right) \neq0$.
	
	\item $M(z)$ has simple poles at each point in $ \mathcal{Z}$ with
\begin{equation}\label{resM11}
\begin{aligned}
&\res\limits_{k=\zeta_n} M(k)=\lim_{k\rightarrow \zeta_n}M(k)B_n,\\
&\res\limits_{k=\omega \bar{\zeta}_n} M(k)=\lim_{k\rightarrow \omega\bar{\zeta}_n} M(k)\Gamma_3(\omega \bar{B}_n)\Gamma_3:=\lim_{k\rightarrow \zeta_n}  M(k)B_{n+N},\\
&\res\limits_{k=\omega \zeta_n} M(k)=\lim_{k\rightarrow \omega \zeta_n} M(k)C^2(\omega^2B_n)C^{-2}:=\lim_{k\rightarrow \zeta_n} M(k)B_{n+2N},\\
&\res\limits_{k=\omega^2\bar{\zeta}_n} M(k)=\lim_{k\rightarrow\omega^2\bar{\zeta}_n} M(k)\Gamma_2(\omega^2 \bar{B}_n)\Gamma_2:=\lim_{k\rightarrow \zeta_n} M(k)B_{n+3N},\\
&\res\limits_{k=\omega^2\zeta_n} M(k)=\lim_{k\rightarrow\omega^2\zeta_n} M(k)C(\omega^2 B_n)C^{-1}:=\lim_{k\rightarrow \zeta_n} M(k)B_{n+4N},\\
&\res\limits_{k=\bar{\zeta}_n} M(k)=\lim_{k\rightarrow \bar{\zeta}_n}M(k)\Gamma_1\bar{B}_n\Gamma_1:=\lim_{k\rightarrow \zeta_n}M(k)B_{n+5N},
\end{aligned}
\end{equation}
where
\begin{align*}
   B_n=\begin{cases}
    \left(\begin{array}{ccc} 0 & -c_ne^{\mathrm{i}t\theta_{12}(\zeta_n)} & 0 \\ 0 & 0 & 0 \\ 0 & 0 & 0 \end{array}\right), \ n=1,\cdots,2N_1+N_2,\\
   \left(\begin{array}{ccc} 0 & 0 & 0 \\ 0 & 0 & -c_ne^{\mathrm{i}t\theta_{23}(\zeta_n)} \\ 0 & 0 & 0 \end{array}\right), \ n=2N_1+N_2+1,\cdots,2N_1+N_2+2N_1^A+N_2^A.
   \end{cases}
\end{align*}

\end{itemize}
\end{RHP}
Denote $M(z;y,t)= (M_{jl}(z;y,t))_{jl=1}^3$.
Then the  solution of  Novikov equation (\ref{Novikov}) can be  obtained by the following reconstruction formula
\begin{align}
	u(x,t)=u(y(x,t),t)=&\frac{1}{2}\tilde{m}_1(y,t)\left(\frac{M_{33}(e^{\frac{i\pi}{6}};y,t)}{M_{11}(e^{\frac{i\pi}{6}};y,t)} \right)^{1/2}\nonumber\\
	&+ \frac{1}{2}\tilde{m}_3(y,t)\left(\frac{M_{33}(e^{\frac{i\pi}{6}};y,t)}{M_{11}(e^{\frac{i\pi}{6}};y,t)} \right)^{-1/2}-1 ,\label{recons u}
\end{align}
where
\begin{equation}
	x(y,t)=y+\frac{1}{2} \ln\frac{M_{33}(e^{\frac{i\pi}{6}};y,t)}{M_{11}(e^{\frac{i\pi}{6}};y,t)} ,\label{recons x}
\end{equation}
and
\begin{align}
	\tilde{m}_l:=\sum_{j=1}^3M_{jl}(e^{\frac{i\pi}{6}};y,t),\ l=1,2,3. \nonumber
\end{align}

\section{Interpolation and Conjugation}\label{sec3}

In this section,  we aim to convert the original RH problem to a new RH problem which satisfies the following conditions:
\begin{itemize}
	\item  It is well behaved as $t\to \infty$ with $\xi$ fixed.
	\item  Different factorizations of the jump matrix \eqref{jumpv0} should be taken for different transition sectors.
\end{itemize}

For convenience, we denote
$$\mathcal{N}:=\left\lbrace 1,\cdots,N_0\right\rbrace, \   \tilde{\mathcal{N}}:=\left\lbrace 1,\cdots,2N_1+N_2\right\rbrace, \ \tilde{\mathcal{N}}^A:=\left\lbrace 1+2N_1+N_2,\cdots,N_0\right\rbrace.$$
Further,  to distinguish different  types of zeros, we introduce a small positive constant $\delta_0$ to give the  partitions $\Delta,\nabla$ and $\lozenge$  of $\mathcal{N}$   as follows:
\begin{align}
&\nabla_1=\left\lbrace j \in \tilde{\mathcal{N}}  ;\text{Im}\theta_{12}(\zeta_j)> \delta_0\right\rbrace, \
\Delta_1=\left\lbrace j \in \tilde{\mathcal{N}} ;\text{Im}\theta_{12}(\zeta_j)< 0\right\rbrace,\nonumber\\
&\nabla_2=\left\lbrace i\in \tilde{\mathcal{N}}^A  ;\text{Im}\theta_{23}(\zeta_i)> \delta_0\right\rbrace, \
\Delta_2=\left\lbrace i \in \tilde{\mathcal{N}}^A ;\text{Im}\theta_{23}(\zeta_i)< 0\right\rbrace, \\
&\lozenge_1=\left\lbrace j_0 \in \tilde{\mathcal{N}} ;0\leq\text{Im}\theta_{12}(\zeta_{j_0})\leq\delta_0\right\rbrace,\ \lozenge_2=\left\lbrace i_0 \in \tilde{\mathcal{N}}^A  ;0\leq\text{Im}\theta_{23}(\zeta_{i_0})\leq\delta_0\right\rbrace, \\	
&\nabla=\nabla_1\cup\nabla_2, \
\Delta=\Delta_1\cup\Delta_2, \ \lozenge=\lozenge_1\cup\lozenge_2.\label{devide}
\end{align}
For $\zeta_n$ with $n\in\Delta$, the residue of $M(z)$ at $\zeta_n$ in RH problem  \ref{RHP1}  grows without bound as $t\to\infty$. However, for $\zeta_n$ with $n\in\nabla$, the residue   decays to $0$. Denote two constants $\mathcal{N}(\lozenge)=|\lozenge|$ and
\begin{equation}
	\rho_0=\min_{n\in\mathcal{N}\setminus \lozenge}\left\lbrace |\text{Im}\theta_{12}(\zeta_n )|,\ |\text{Im}\theta_{23}(\zeta_n )| \right\rbrace >0.\label{rho0}
\end{equation}

For the poles $\zeta_n$ with $ n \in \mathcal{N}\setminus \lozenge$, we want to convert the residue of these poles into
the jumps along small closed loops enclosing themselves,
respectively.  The jump matrix $V(z)$  on  $ \Sigma$  in (\ref{jumpv})   has the following  factorizations: \\
On the contour $\mathbb{R}$,
\begin{align}
	&V (z)=\left(\begin{array}{ccc}
	1 & 0 &0\\
	\bar{r} e^{-it\theta_{12}} & 1 &0\\
	0 & 0 &1
\end{array}\right)\left(\begin{array}{ccc}
1 & -r e^{it\theta_{12}} &0\\
0 & 1&0\\
0 & 0 &1
\end{array}\right)\\
&=\left(\begin{array}{ccc}
	1 & \frac{-r e^{it\theta_{12}}}{1-|r|^2}&0\\
	0 & 1&0\\
	0 & 0 &1
\end{array}\right)\left(\begin{array}{ccc}
\frac{1 }{1-|r|^2} & 0 & 0\\
0 & 1-|r|^2 & 0\\
0 & 0 &1
\end{array}\right)\left(\begin{array}{ccc}
1 & 0 & 0\\
\frac{\bar{r} e^{-it\theta_{12}}}{1-|r|^2} & 1 & 0\\
0 & 0 &1
\end{array}\right);\nonumber
\end{align}
On the contour $\omega^2\mathbb{R}$,
\begin{align}
	&V (z)=\left(\begin{array}{ccc}
		1 & 0 &0\\
		0 & 1 &0\\
		0 & \bar{r}(\omega z) e^{-it\theta_{23}} &1
	\end{array}\right)\left(\begin{array}{ccc}
		1 & 0 &0\\
		0 & 1&-r(\omega z) e^{it\theta_{23}}\\
		0 & 0 &1
	\end{array}\right)\\
	&=\left(\begin{array}{ccc}
		1 & 0&0\\
		0 & 1&\frac{-r(\omega z) e^{it\theta_{23}}}{1-|r(\omega z)|^2}\\
		0 & 0 &1
	\end{array}\right)\left(\begin{array}{ccc}
		1 & 0 & 0\\
		0 & \frac{1 }{1-|r(\omega z)|^2} & 0\\
		0 & 0 &1-|r(\omega z)|^2
	\end{array}\right)\left(\begin{array}{ccc}
		1 & 0 & 0\\
	0 & 1 & 0\\
		0 & 	\frac{\bar{r}(\omega z) e^{-it\theta_{23}}}{1-|r(\omega z)|^2} &1
	\end{array}\right);\nonumber
\end{align}
On the contour $\omega\mathbb{R}$,
\begin{align}
	&V (z)=\left(\begin{array}{ccc}
		1 & 0 &\bar{r}(\omega^2 z) e^{it\theta_{13}}\\
		0 & 1 &0\\
		0 & 0 &1
	\end{array}\right)\left(\begin{array}{ccc}
		1 & 0 &0\\
		0 & 1&0\\
		-r(\omega^2 z) e^{-it\theta_{13}} & 0 &1
	\end{array}\right)\\
	&=\left(\begin{array}{ccc}
		1 & 0&0\\
		0 & 1&0\\
		\frac{-r(\omega^2 z) e^{-it\theta_{13}}}{1-|r(\omega^2 z)|^2} & 0 &1
	\end{array}\right)\left(\begin{array}{ccc}
		1-|r(\omega^2 z)|^2 & 0 & 0\\
		0 & 1 & 0\\
		0 & 0 &\frac{1}{1-|r(\omega^2 z)|^2}
	\end{array}\right)\left(\begin{array}{ccc}
		1 & 0 & 	\frac{\bar{r}(\omega^2 z) e^{it\theta_{13}}}{1-|r(\omega^2 z)|^2}\\
		0 & 1 & 0\\
		0 & 0 &1
	\end{array}\right).\nonumber
\end{align}

To achieve our goal, we utilize these factorizations to deform the jump contours so that the oscillating factor $e^{\pm it\theta_{12}}$, which is determined  by the sign of $\im \theta_{12}$, are decaying in the corresponding regions, respectively,.
 We consider the signature table of   $\im \theta_{12}$
  \begin{align}
  	&\im \theta_{12}=\sqrt{3}\text{Im}z\left(1+|z|^{-2}  \right)\xi-\nonumber\\
  	& \dfrac{\sqrt{3}\text{Im}z\left(1+|z|^{-2}\right) \left( -|z|^6-|z|^4+4\text{Re}^2z|z|^2-|z|^2\right) }{|z|^8+1+2[(\text{Re}^2z-\text{Im}^2z)^2-4\text{Re}^2z\text{Im}^2z]-2(1+|z|^4)(\text{Re}^2z-\text{Im}^2z)+|z|^4 }  .\label{Reitheta}
  \end{align}
  which is depicted in Figure \ref{figpha}.


\begin{figure}[htbp]
     \centering
     \subfigure[$\xi<-\frac{1}{8}$]{
      \begin{minipage}[t]{0.32\linewidth}
       \centering
       \includegraphics[width=1.52in]{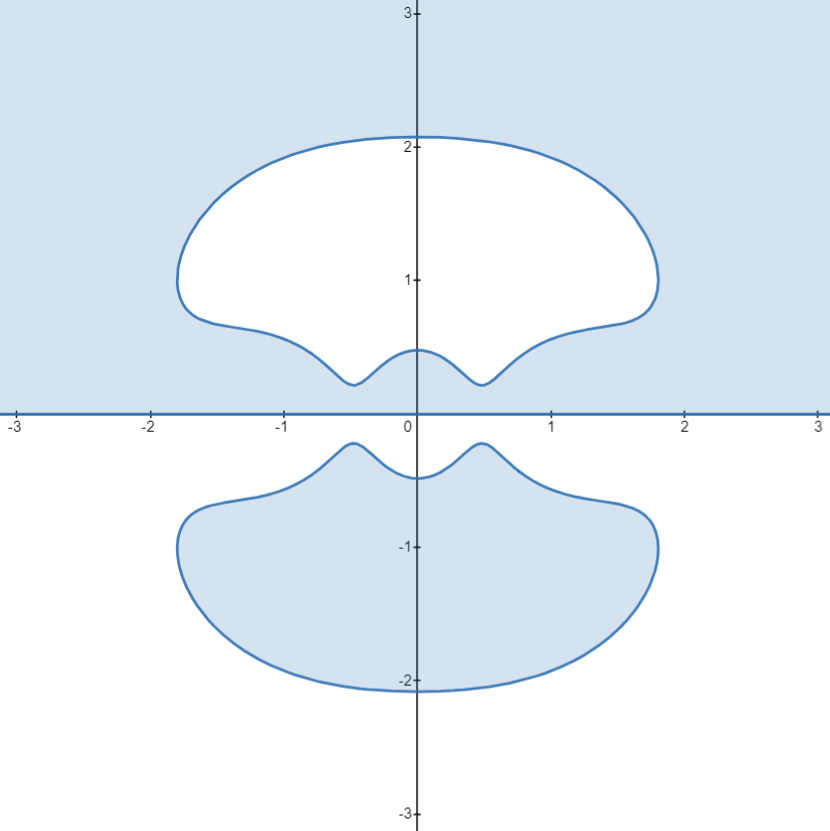}
      \end{minipage}%
     }%
    \subfigure[$\xi=-\frac{1}{8}$]{\label{figureb}
     \begin{minipage}[t]{0.32\linewidth}
      \centering
      \includegraphics[width=1.52in]{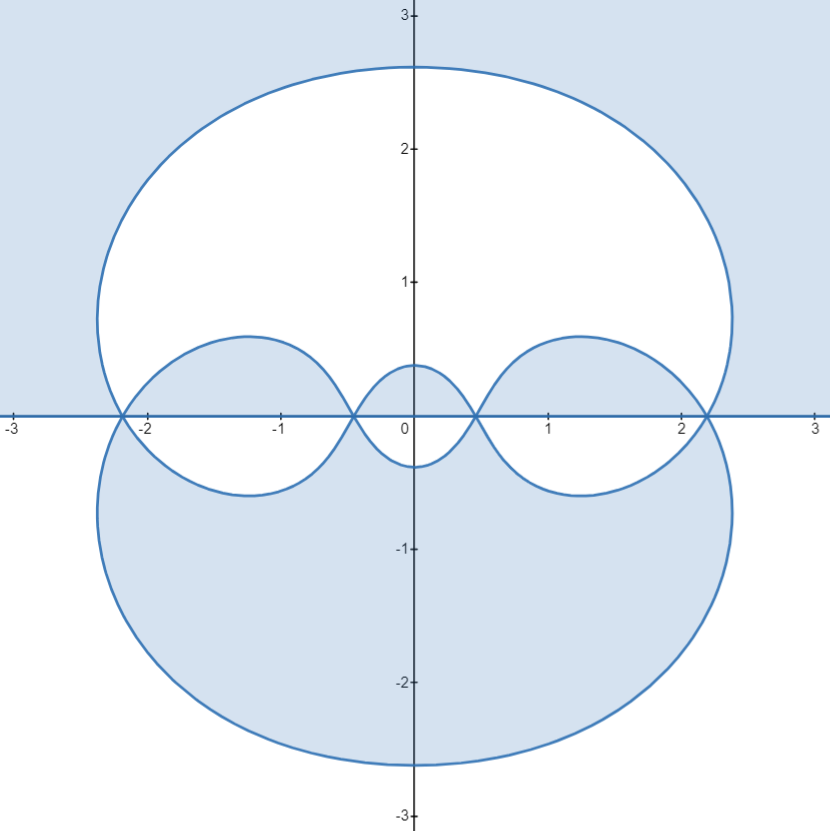}
     \end{minipage}
    }%
    \subfigure[$-\frac{1}{8}<\xi<0$]{\label{figurec}
     \begin{minipage}[t]{0.32\linewidth}
      \centering
      \includegraphics[width=1.52in]{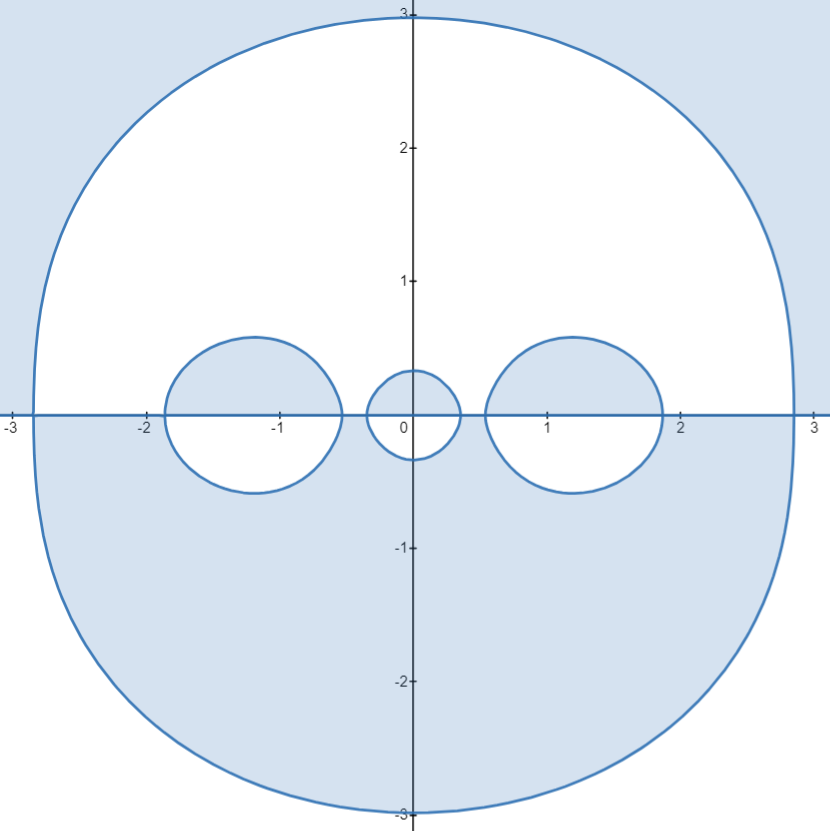}
     \end{minipage}
    }%

    \subfigure[$0\le\xi<1$]{\label{figured}
 \begin{minipage}[t]{0.32\linewidth}
  \centering
  \includegraphics[width=1.52in]{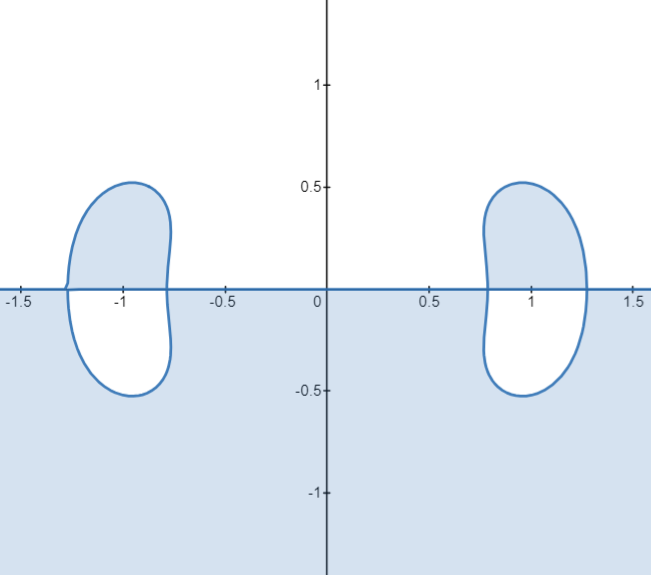}
 \end{minipage}
 }%
  \subfigure[$ \xi=1$]{
      \begin{minipage}[t]{0.32\linewidth}
       \centering
       \includegraphics[width=1.52in]{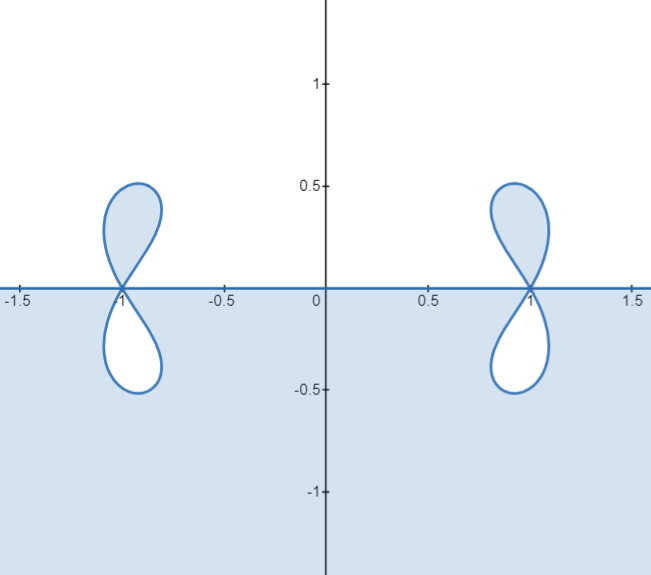}
      \end{minipage}
     }%
  \subfigure[$\xi>1$]{
      \begin{minipage}[t]{0.32\linewidth}
       \centering
       \includegraphics[width=1.52in]{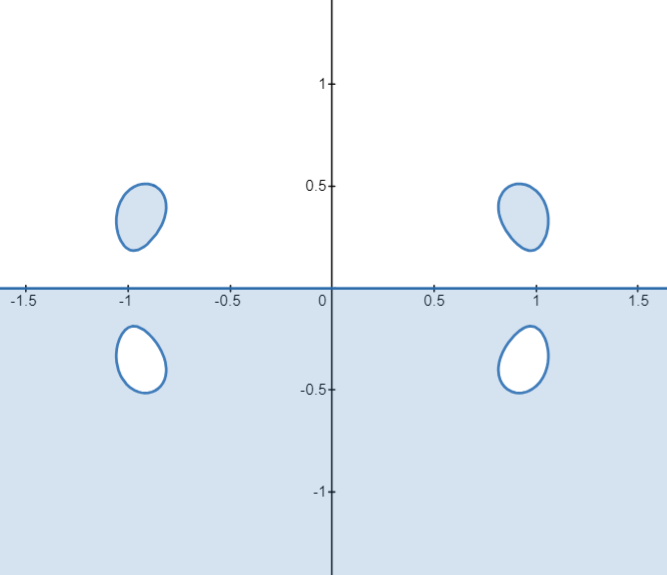}
      \end{minipage}
     }%
     \centering
     \caption{\footnotesize The signature table of $\im \theta_{12}$ with different $\xi$:
    $\textbf{(a)}$ $\xi<-\frac{1}{8}$,
    $\textbf{(b)}$ $\xi=-\frac{1}{8}$,
    $\textbf{(c)}$ $-\frac{1}{8}<\xi<0$,
    $\textbf{(d)}$ $0\le \xi<1$,
    $\textbf{(e)}$ $\xi=1$,
    $\textbf{(f)}$ $\xi>1$.
   In the blue region, $\im \theta_{12}<0$ and then $|e^{-\mathrm{i}t\theta_{12}}|\rightarrow 0$ as $t\rightarrow\infty$; In the white region, $\im \theta_{12}>0$ and then $|e^{\mathrm{i}t\theta_{12}}|\rightarrow 0$ as $t\rightarrow\infty$. Moreover, on the blue curve, $\im \theta_{12}=0$.}
         \label{figpha}
\end{figure}


To proceed, we introduce the following scalar RH problem for the transition zone (as shown in Figure \ref{figureb}.
\begin{RHP} Find a scalar function  $\delta_1(z)$ with
\begin{itemize}
	\item  $\delta_1(z)$	is analytic in $\mathbb{C}\setminus\mathbb{R}$.
\item $\delta_1(z)$ has jump relation:
\begin{align}
& \delta_{1,+} (z)=\delta_{1,-}(z)(1-|r(z)|^2), \ z\in \mathbb{R},\nonumber
\end{align}
\item $\delta_1(z)\to 1$ as $z\to\infty$.
\end{itemize}
\end{RHP}
 By the Plemelj  formula, this RH problem admits an unique solution
\begin{align}
	\delta_1 (z)& =\exp\left(-i\int _{\mathbb{R}}\dfrac{\nu(s) ds}{s-z}\right),
\end{align}
where $\nu(s) =-\frac{1}{2\pi }\log (1-|r(s)|^2).$

Moreover,  define
\begin{align}
H(z)&=\prod_{n\in \Delta_1}\dfrac{z-\zeta_n}{z-\bar{\zeta}_n}\prod_{m\in \Delta_2}\dfrac{z-\omega\zeta_m}{z-\omega^2\bar{\zeta}_m}\delta_1 (z,\xi)^{-1};\label{Pi}\\
T_1(z)&=T_1(z,\xi)= \frac{H(\omega^2z)}{H(z)};\label{T}\\ T_2(z)&=T_2(z,\xi)=T_1(\omega z)= \frac{H(z)}{H(\omega z)};\\ T_3(z)&=T_3(z,\xi)=T_1(\omega^2 z)= \frac{H(\omega z)}{H(\omega^2 z)};\label{T3}\\
T_{ij}(z)&=T_{ij}(z,\xi)=\frac{T_i(z)}{T_j(z)},\ i,j=1,2,3\label{Tij}.
\end{align}
In  the above formulas, we choose the principal branch of power and logarithm functions. 

\begin{Proposition}\label{proT}
	The function defined by (\ref{T}) and (\ref{Tij}) has the  following properties
\begin{itemize}
\item[{\rm (a)}]  $T_1(z)$ is meromorphic in $\mathbb{C}\setminus \mathbb{R}$. For each $n\in\Delta_1$, $T_1(z)$  exhibits  a simple pole at $\zeta_n$ and a simple zero
 at $\bar{\zeta}_n$; while for each $m\in\Delta_2$, $T_1(z)$ possesses a simple pole at $\omega\zeta_m$ and a simple zero at $\omega\bar{\zeta}_m$.

\item[{\rm (b)}]  $\overline{T_1(\bar{z})}=T_1(\omega z)=T_1(z^{-1})$.

\item[{\rm (c)}]  $T_1(z)$ satisfies the jump condition
	\begin{align}
	& T_{1,-}(z)=(1-|r(z)|^2) T_{1,+ }(z),\hspace{0.5cm}z\in \mathbb{R},\\
&T_{1,+ }(z)=(1-|r(\omega^2z)|^2)T_{1,- }(z),\hspace{0.5cm}z\in \omega \mathbb{R}.
	\end{align}

	\item[{\rm (d)}]  $\lim_{z\to \infty}T_1(z):= T_1(\infty)$ with $T_1(\infty)=1$.

\item[{\rm (e)}]  $T_1(e^{\frac{i\pi}{6}})$ exists as a constants.

\item[{\rm (f)}]  $T_1(z)$ is continuous at $z=0$ with $T_1(0)=1$.
\end{itemize}
\end{Proposition}

For
\begin{align}
	\varrho:=\frac{1}{4}\min&\left\lbrace \min_{j\in \mathcal{N}} |\text{Im}\zeta_j| ,\min_{j\in \mathcal{N},\; \arg(z)=\frac{i \pi }{3}}|\zeta_j-z|,\min_{j\in \mathcal{N}\setminus\lozenge, \;\text{Im}\theta_{ik}(z)=0}|\zeta_j-z|,\right.\nonumber\\
	&\left. \min_{  j\in \mathcal{N}}|\zeta_j-e^{\frac{i\pi}{6}}|,\min_{  j\neq k\in \mathcal{N}}|\zeta_j-\zeta_k|\right\rbrace ,
\end{align}
we also define
$$\mathbb{D}_n:=\mathbb{D}(\zeta_n,\varrho)=\{ z: |z-\zeta_n| \le \varrho\}, \ n \in \mathcal{N},$$
to be small disks, which are pairwise disjoint,  also  disjoint  with  critical lines
$\left\lbrace z\in \mathbb{C};\text{Im} \theta(z)=0 \right\rbrace $, as well as  the contours  $\mathbb{R}$, $\omega\mathbb{R}$ and $\omega^2\mathbb{R}$.
Besides, $e^{\frac{i\pi}{6}}\notin \mathbb{D}_n$.
By the above definition and the symmetry of poles and $\theta_{jk}(z)$, for every $n$, there exists  a $k\in\{0,\cdots,5\}$ such that $n-kN_0\in\mathcal{N}\setminus\lozenge$.
Denote
 a piecewise matrix function
\begin{equation}
	G(z)=\left\{ \begin{array}{ll}
		I-\frac{B_n}{z-\zeta_n},   & z\in\mathbb{D}_n,n-kN_0\in\nabla,k\in\{0,\cdots,5\},\\[12pt]
		\left(\begin{array}{ccc}
			1& 0& 0 \\
			-\frac{z-\zeta_n}{C_{n}e^{it\theta_{12}(\zeta_{n})}}& 1 & 0\\
			0&	0 & 1
		\end{array}\right),   & z\in\mathbb{D}_n,n\in\Delta_1 \text{ or }n-2N_0\in\Delta_2,\\[12pt]
	\left(\begin{array}{ccc}
		1&	0 & 0 \\
		0& 1 & 0\\
		0& -\frac{z-\zeta_n}{C_{n}e^{it\theta_{23}(\zeta_{n})}} & 1
	\end{array}\right),   & z\in\mathbb{D}_n,n-N_0\in\Delta_1 \text{ or }n-5N_0\in\Delta_2,\\[12pt]
	\left(\begin{array}{ccc}
		1&	0 & -\frac{z-\zeta_n}{C_{n}e^{-it\theta_{13}(\zeta_{n})}} \\
		0& 1 & 0\\
		0&	0 & 1
	\end{array}\right),   & z\in\mathbb{D}_n,n-2N_0\in\Delta_1 \text{ or }n-4N_0\in\Delta_2,\\[12pt]
	\left(\begin{array}{ccc}
		1&	0 &  0\\
		0& 1 &  -\frac{z-\zeta_n}{C_{n}e^{-it\theta_{23}(\zeta_{n})}}\\
		0&	0 & 1
	\end{array}\right),   & z\in\mathbb{D}_n,n-3N_0\in\Delta_1 \text{ or }n-N_0\in\Delta_2,\\[12pt]
	\left(\begin{array}{ccc}
		1&	0 & 0 \\
		0& 1 & 0\\
		0 &-\frac{z-\zeta_n}{C_{n}e^{it\theta_{23}(\zeta_{n})}}	 & 1
	\end{array}\right),   & z\in\mathbb{D}_n,n-4N_0\in\Delta_1 \text{ or }n\in\Delta_2,\\[12pt]
	\left(\begin{array}{ccc}
	1&	-\frac{z-\zeta_n}{C_{n}e^{-it\theta_{12}(\zeta_{n})}} & 0 \\
	0& 1 & 0\\
	0&	0 & 1
	\end{array}\right),   & z\in\mathbb{D}_n,n-5N_0\in\Delta_1\text{ or } n-3N_0\in\Delta_2,\\[12pt]
	I ,& 	z \text{ in elsewhere},
	\end{array}\right.\label{funcG}
\end{equation}
where $B_n, \ n=1,\cdots,2N_1+N_2+2N_1^A +N_2^A$ are defined in RH problem \ref{RHP1}, and define
$$T(z)=\text{diag}\{T_1(z),T_2(z),T_3(z)\}.$$
Now we introduce the following transformation to construct a regular RH problem
\begin{equation}
M^{(1)}(z ):=M^{(1)}(z;y,t)=M(z)G(z)T(z),\label{transm1}
\end{equation}
which then satisfies the following RH problem.

\begin{RHP}\label{RHP3}
	Find a matrix-valued function  $  M^{(1)}(z )$ which satisfies:
	\begin{itemize}
\item $M^{(1)}(z)$ is meromorphic in $\mathbb{C}\setminus \Sigma^{(1)}$, where $	\Sigma^{(1)}=\Sigma \cup\left( \underset{n\in\mathcal{N}\setminus\lozenge,\;k=0,..,5}{\cup}\partial\mathbb{D}_{n+kN_0} \right)$.
	
	\item  $M^{(1)}(z)=\Gamma_1\overline{M^{(1)}(\bar{z})}\Gamma_1=\Gamma_2\overline{M^{(1)}(\omega^2\bar{z})}\Gamma_2=\Gamma_3\overline{M^{(1)}(\omega\bar{z})}\Gamma_3 =\overline{M^{(1)}(\bar{z}^{-1})}$.

\item $	M^{(1)}_+(z)=M^{(1)}_-(z)V^{(1)}(z),\hspace{0.5cm}z \in \Sigma^{(1)},$
	where
	\begin{equation}
		V^{(1)}(z)=\left\{ \begin{array}{ll}
			\left(\begin{array}{ccc}
				1 & -\frac{rT_{21} e^{it\theta_{12}}}{1-|r|^2}&0\\
				0 & 1&0\\
				0 & 0 &1
			\end{array}\right)\left(\begin{array}{ccc}
				1 & 0 & 0\\
				\frac{\bar{r}T_{12} e^{-it\theta_{12}}}{1-|r|^2} & 1 & 0\\
				0 & 0 &1
			\end{array}\right),   & z\in \mathbb{R},\\[12pt]
		\left(\begin{array}{ccc}
			1 & 0&0\\
			0 & 1&\frac{-r(\omega z)T_{32} e^{it\theta_{23}}}{1-|r(\omega z)|^2}\\
			0 & 0 &1
		\end{array}\right)\left(\begin{array}{ccc}
			1 & 0 & 0\\
			0 & 1 & 0\\
			0 & 	\frac{\bar{r}(\omega z) e^{-it\theta_{23}}T_{23}}{1-|r(\omega z)|^2} &1
		\end{array}\right),   & z\in \omega\mathbb{R},\\[12pt]
	\left(\begin{array}{ccc}
		1 & 0&0\\
		0 & 1&0\\
		\frac{-r(\omega^2 z)T_{13} e^{-it\theta_{13}}}{1-|r(\omega^2 z)|^2} & 0 &1
	\end{array}\right)\left(\begin{array}{ccc}
		1 & 0 & 	\frac{\bar{r}(\omega^2 z)T_{31} e^{it\theta_{13}}}{1-|r(\omega^2 z)|^2}\\
		0 & 1 & 0\\
		0 & 0 &1
	\end{array}\right),   & z\in \omega^2 \mathbb{R},\\[12pt]
	T^{-1}(z)G(z)T(z),   & 	z\in\partial\mathbb{D}_n\cap (\underset{k=1}{\cup^3}S_{2k}),\\[12pt]
		T^{-1}(z)G^{-1}(z)T(z),    & z\in\partial\mathbb{D}_n\cap (\underset{k=1}{\cup^3}S_{2k-1}).\\
		\end{array}\right. \label{jumpv1}
	\end{equation}

\item  	$M^{(1)}(z) = I+\mathcal{O}(z^{-1}),\hspace{0.5cm}z \rightarrow \infty$.

\item As $z\to \varkappa_l=e^{\frac{i\pi(l-1)}{3}}$, $l = 1,\cdots,6$,   the limit of $M^{(1)}(z)$  has  the pole singularities
	\begin{align}
		&M^{(1)}(z)=\frac{1}{z\mp1}\left(\begin{array}{ccc}
			\alpha_\pm^{(1)} &	\alpha_\pm^{(1)} & \beta_\pm^{(1)} \\
			-\alpha_\pm^{(1)} & -\alpha_\pm^{(1)} & -\beta_\pm^{(1)}\\
			0	&	0 & 0
		\end{array}\right)T(\pm1)+\mathcal{O}(1),\ z\to\pm 1,\\
		&M^{(1)}(z)=\frac{1}{z\mp\omega^2}\left(\begin{array}{ccc}
			0 &	0 &  0\\
			\beta_\pm^{(1)}	 & \alpha_\pm^{(1)} &\alpha_\pm^{(1)} \\
			-\beta_\pm^{(1)}	&	-\alpha_\pm^{(1)} & -\alpha_\pm^{(1)}
		\end{array}\right)T(\pm\omega^2)+\mathcal{O}(1),\ z\to\pm \omega^2,\\
		&M^{(1)}(z)=\frac{1}{z\mp\omega}\left(\begin{array}{ccc}
			-\alpha_\pm^{(1)} &	-\beta_\pm^{(1)} & -\alpha_\pm^{(1)}\\
			0	 & 0 &0 \\
			\alpha_\pm^{(1)} &	\beta_\pm^{(1)} & \alpha_\pm^{(1)}
		\end{array}\right)T(\pm\omega)+\mathcal{O}(1),\ z\to\pm \omega,
	\end{align}
	with $\alpha_\pm^{(1)}=\alpha_\pm^{(1)}(y,t)=-\bar{\alpha}_\pm^{(1)}$, $\beta_\pm^{(1)}=\beta_\pm^{(1)}(y,t)=-\bar{\beta}_\pm^{(1)}$ and $M^{(1)}(z)^{-1}$ has same specific
	matrix structure with $\alpha_\pm^{(1)}$, $\beta_\pm^{(1)}$ replaced by $\tilde{\alpha}_\pm^{(1)}$, $\tilde{\beta}_\pm^{(1)}$.
	
	\item $M^{(1)}(z)$ has simple poles at each point $\zeta_n$ for $n-kN_0\in\lozenge$ with
	\begin{align}
		&\res_{z=\zeta_{n}}M^{(1)}(z)=\lim_{z\to \zeta_{n}}M^{(1)}(z)\left[ T^{-1}(z)B_nT(z)\right] .
	\end{align}
\end{itemize}
\end{RHP}

\begin{proof}
The above properties of RH problem \ref{RHP3} can be directly obtained by the properties of RH problem \ref{RHP1}, Proposition \ref{proT}, and \eqref{funcG}- \eqref{transm1}.
\end{proof}

Since the jump matrices on the disks $\mathbb{D}_n, n -k N_0 \in \mathcal{N}\setminus\lozenge$ decay exponentially to the identity matrix as $t \to \infty$, it follows that the RH problem is asymptotically equivalent to the following RH problem.

\begin{RHP}\label{RHP31}
	Find a matrix-valued function  $  M^{(2)}(z )$ which satisfies:
\begin{itemize}
  \item $M^{(2)}(z)$ is meromorphic in $\mathbb{C}\setminus \Sigma^{(2)}$, where $	\Sigma^{(2)}=\Sigma$.
      \item $M^{(2)}(z)$ obeys the symmetries \begin{equation}\label{symm2}
          M^{(2)}(z)=\Gamma_1\overline{M^{(2)}(\bar{z})}\Gamma_1=\Gamma_2\overline{M^{(2)}(\omega^2\bar{z})}\Gamma_2=\Gamma_3\overline{M^{(2)}(\omega\bar{z})}\Gamma_3 =\overline{M^{(2)}(\bar{z}^{-1})}.
          \end{equation}
\item $	M^{(2)}_+(z)=M^{(2)}_-(z)V^{(2)}(z),\hspace{0.5cm}z \in \Sigma^{(2)},$
	where
	\begin{equation}
		V^{(2)}(z)=\left\{\begin{array}{ll}	
			\left(\begin{array}{ccc}
				1 & -\frac{rT_{21} e^{it\theta_{12}}}{1-|r|^2}&0\\
				0 & 1&0\\
				0 & 0 &1
			\end{array}\right)\left(\begin{array}{ccc}
				1 & 0 & 0\\
				\frac{\bar{r}T_{12} e^{-it\theta_{12}}}{1-|r|^2} & 1 & 0\\
				0 & 0 &1
			\end{array}\right),   & z\in \mathbb{R},\\[12pt]
		\left(\begin{array}{ccc}
			1 & 0&0\\
			0 & 1&\frac{-r(\omega z)T_{32} e^{it\theta_{23}}}{1-|r(\omega z)|^2}\\
			0 & 0 &1
		\end{array}\right)\left(\begin{array}{ccc}
			1 & 0 & 0\\
			0 & 1 & 0\\
			0 & 	\frac{\bar{r}(\omega z) e^{-it\theta_{23}}T_{23}}{1-|r(\omega z)|^2} &1
		\end{array}\right),   & z\in \omega \mathbb{R},\\[12pt]
	\left(\begin{array}{ccc}
		1 & 0&0\\
		0 & 1&0\\
		\frac{-r(\omega^2 z)T_{13} e^{-it\theta_{13}}}{1-|r(\omega^2 z)|^2} & 0 &1
	\end{array}\right)\left(\begin{array}{ccc}
		1 & 0 & 	\frac{\bar{r}(\omega^2 z)T_{31} e^{it\theta_{13}}}{1-|r(\omega^2 z)|^2}\\
		0 & 1 & 0\\
		0 & 0 &1
	\end{array}\right),   & z\in \omega^2 \mathbb{R},\\
		\end{array}\right. \label{jumpv02}
	\end{equation}

\item  	$M^{(2)}(z) = I+\mathcal{O}(z^{-1}),\hspace{0.5cm}z \rightarrow \infty$.

\item As $z\to \varkappa_l=e^{\frac{i\pi(l-1)}{3}}$, $l = 1,\cdots,6$,
	\begin{align}
		&M^{(2)}(z)=\frac{1}{z\mp1}\left(\begin{array}{ccc}
			\alpha_\pm^{(2)} &	\alpha_\pm^{(2)} & \beta_\pm^{(2)} \\
			-\alpha_\pm^{(2)} & -\alpha_\pm^{(2)} & -\beta_\pm^{(2)}\\
			0	&	0 & 0
		\end{array}\right)T(\pm1)+\mathcal{O}(1),\ z\to\pm 1,\label{m2sin1}\\
		&M^{(2)}(z)=\frac{1}{z\mp\omega^2}\left(\begin{array}{ccc}
			0 &	0 &  0\\
			\beta_\pm^{(2)}	 & \alpha_\pm^{(2)} &\alpha_\pm^{(2)} \\
			-\beta_\pm^{(2)}	&	-\alpha_\pm^{(2)} & -\alpha_\pm^{(2)}
		\end{array}\right)T(\pm\omega^2)+\mathcal{O}(1),\ z\to\pm \omega^2,\label{m2sin2}\\
		&M^{(2)}(z)=\frac{1}{z\mp\omega}\left(\begin{array}{ccc}
			-\alpha_\pm^{(2)} &	-\beta_\pm^{(2)} & -\alpha_\pm^{(2)}\\
			0	 & 0 &0 \\
			\alpha_\pm^{(2)} &	\beta_\pm^{(2)} & \alpha_\pm^{(2)}
		\end{array}\right)T(\pm\omega)+\mathcal{O}(1),\ z\to\pm \omega,\label{m2sin3}
	\end{align}
	with $\alpha_\pm^{(2)}=\alpha_\pm^{(2)}(y,t)=-\bar{\alpha}_\pm^{(2)}$, $\beta_\pm^{(2)}=\beta_\pm^{(2)}(y,t)=-\bar{\beta}_\pm^{(2)}$ and $M^{(2)}(z)^{-1}$ has same specific
	matrix structure with $\alpha_\pm^{(2)}$, $\beta_\pm^{(2)}$ replaced by $\tilde{\alpha}_\pm^{(2)}$, $\tilde{\beta}_\pm^{(2)}$.
	
	\item $M^{(2)}(z)$ has simple poles at each point $\zeta_n$ for $n -k N_0\in\lozenge$ with
	\begin{align}
		&\res_{z=\zeta_n}M^{(2)}(z)=\lim_{z\to \zeta_n}M^{(2)}(z)\left[ T^{-1}(z)B_nT(z)\right] .
	\end{align}
\end{itemize}
\end{RHP}

\begin{Proposition}
The solution of RH problem \ref{RHP3} can be approximated by the solution of RH problem \ref{RHP31}
\begin{equation}\label{p1transm2}
M^{(1)}(z) = M^{(2)}(z) (I+ \mathcal{O}(e^{-ct})),
\end{equation}
where $c$ is a positive constant.

\end{Proposition}

\begin{proof}
The result is derived from the theorem of Beals-Coifman and the corresponding norm estimates.
\end{proof}

\section{Painlev\'e Asymptotics   in  Transition Zone  $y/t \approx -1/8 $ }\label{sec4}

In this section, we consider the large time  asymptotics   in the zone  $0 < (\xi+\frac{1}{8})t^{2/3}<C$ with  $C>0$ corresponding  to
Figure \ref{figurec}.    A comparable discussion can be conducted for the other half region $-C < (\xi+\frac{1}{8})t^{2/3}<0$  corresponding  to  Figure \ref{figureb}.

From \eqref{deftheta12}, the phase function $\theta_{12}$ has eight saddle points  such that $\theta'_{12}=0$.
Introducing $$\tilde{k}=z-\frac{1}{z},$$ then one gets
\begin{equation}
	\theta'_{12}(z) = \sqrt{3} \left( \xi - \frac{1-\tilde{k}^2}{(\tilde{k}^2+1)^2}\right) \left(1+\frac{1}{z^2}\right)
\end{equation}
Thus, in the $\tilde{k}$-plane the real critical points $\tilde{k}_1 \in \mathbb{R}$ are determined by the equation
\begin{equation}\label{xik1}
	 \xi = \frac{1-\tilde{k}^2}{(\tilde{k}^2+1)^2},
\end{equation}
or, equivalently, in terms of $\varpi = \tilde{k}^2+1 \ge 1$, by
\begin{equation}\label{sadpoi}
	\xi \varpi ^2 +\varpi -2 = 0.
\end{equation}
In $0<(\xi+\frac{1}{8})t^{2/3}<C$, this equation has two solutions $\ge 1$, which gives in the $\tilde{k}$-plane four real saddle points $\pm \kappa_0, \pm \kappa_1$:
\begin{align}
&	\kappa_0(\xi) = \left( \frac{\sqrt{1+8\xi}-1-2\xi}{2\xi} \right)^{\frac{1}{2}},\\
	&	\kappa_1(\xi) = \left(- \frac{\sqrt{1+8\xi}  +1+2\xi }{2\xi}\right)^{\frac{1}{2}}.
\end{align}
Applying the inverse map $\tilde{k} \leadsto z$ to obtain the corresponding saddle points in the $z$-plane we get  $\pm p_0, \pm \frac{1}{p_0}, \pm p_1, \pm \frac{1}{p_1}$:
\begin{align}
	&p_0(\xi) =\frac{ \sqrt{\kappa_0^2+4}-\kappa_0 }{2},\\
	&p_1(\xi) =\frac{ \sqrt{\kappa_1^2+4}-\kappa_1 }{2}.
\end{align}

As $\xi \to -\frac{1}{8}^+$, $\kappa_0, \kappa_1 \to \sqrt{3}$. Thus, pairs of saddle points collide in the $z$-plane
$$p_0, p_1 \to z_b=\frac{\sqrt{7}-\sqrt{3}}{2}, \quad \frac{1}{p_0}, \frac{1}{p_1} \to z_a = \frac{\sqrt{7}+\sqrt{3}}{2}.$$


By  \eqref{jumpv02}, the jump matrix for $M^{(2)}(z)$ on $\mathbb{R}$ is
\begin{equation}
V^{(2)}(z)=\left(\begin{array}{ccc}
				1 & -\overline{\tilde{r}(z)}e^{it\theta_{12}(z)}&0\\
				0 & 1&0\\
				0 & 0 &1
			\end{array}\right)\left(\begin{array}{ccc}
				1 & 0 & 0\\
				\tilde{r}(z)e^{-it\theta_{12}(z)} & 1 & 0\\
				0 & 0 &1
			\end{array}\right),
\end{equation}
where
\begin{equation}\label{deftilder}
\tilde{r}(z):= \frac{\overline{r(z)} T_{12}(z)}{1-|r(z)|^2}.
\end{equation}
The jump matrix on $\omega \mathbb{R}$ and $\omega^2 \mathbb{R}$ can be obtained by the symmetries \eqref{symm2}.


\subsection{Opening $\bar\partial$-lenses }\label{p1open}

According to the signature table of $\im \theta_{12}$ for $0<\left(\xi + \frac{1}{8}\right)t^{2/3} < C$ illustrated in Figure \ref{figurec},  we now want to remove the jump from the intervals $(-\infty,-\frac{1}{p_1}) \cup (-\frac{1}{p_0},-p_0)\cup (-p_1,p_1) \cup (p_0, \frac{1}{p_0}) $ in such a way that the new problem takes advantage of the decay/growth of $e^{\pm i t \theta_{12}(z)}$. Additionally, we want to "open the lens" in such a way that the lenses are bounded away from the disks introduced previously to remove the poles from the RH problem.

For $l=0,1,2$, define
\begin{align}
&\Omega_1^{l}: = \{z\in \mathbb{C}: 0\le \arg(z-\omega^l\frac{1}{p_1}) \le \varphi_0 \}, \label{p1spa1}\\
&\Omega_2^{l} := \{z\in \mathbb{C}: \pi -\varphi_0 \le \arg(z-\omega^l\frac{1}{p_0}) \le \pi,\ |\re(z-\omega^l\frac{1}{p_0})| \le \frac{1-p_0^2}{2p_0} \},\\
&\Omega_3^{l}:= \{z\in \mathbb{C}: 0 \le \arg(z-\omega^lp_0)  \le \varphi_0,\ |\re(z-\omega^lp_0)| \le \frac{1-p_0^2}{2p_0} \},\\
 &\Omega_4^{l} := \{z\in \mathbb{C}: \pi -\varphi_0 \le \arg(z-\omega^lp_1) \le \pi,\ |\re(z-\omega^lp_1)| \le \frac{p_1}{2} \}\label{p1spa2},
\end{align}
where $0<\varphi_0<\frac{\pi}{8}$ is a sufficiently small angle such that each $\Omega_j^{l}$ doesn't intersect the set $\{z\in\mathbb{C}: \im \theta_{12}(z) = 0 \}$ and any small disks $\mathbb{D}_n,\ n\in \mathcal{N}$.
Denote $\Omega_j^{l},\ j=5,6,7,8,$ be the regions of $\Omega_j^{l},\ j=1,2,3,4,$ symmetric about the imaginary axis.

Moreover,  we use $\Sigma_j^{l},\ j=1,\cdots,8, \ l=0,1,2,$ to denote the boundary of  $\Omega_j^{l},\ j=1,\cdots,8, \ l=0,1,2,$ in the upper half plane and set
\begin{align}
&\Sigma^{l}_{2,3} = \left\{  z\in \mathbb{C}:  \frac{z}{\omega^l}=  \frac{1+p_0^2}{2p_0} + i \rho, \ \rho \in (0, \frac{1-p_0^2}{2p_0 \tan \varphi_0})     \right\},\\
&\Sigma^{l}_{6,7} = \left\{  z\in \mathbb{C}:  \frac{z}{\omega^l}=  -\frac{1+p_0^2}{2p_0} + i \rho, \ \rho \in (0, \frac{1-p_0^2}{2p_0 \tan \varphi_0})     \right\},\\
&I_1^l := \omega^l (\frac{1}{p_0}, \frac{1}{p_1}), \ I^l_2 : =  \omega^l(p_1,p_0), \ I^l_3 :=  \omega^l(-p_0,-p_1), \ I_4^l = \omega^l (-\frac{1}{p_1}, -\frac{1}{p_0}).
\end{align}

\begin{figure}[htbp]
	\centering
\begin{tikzpicture} [scale=0.6]
\draw[dotted](0,0)--(6.287,0)node[right]{ $L_1$};
\draw[dotted](0,0)--(-6.287,0)node[left]{ $L_4$};
\draw[dotted,rotate=60](0,0)--(6.287,0)node[above]{ $L_2$};
\draw[dotted,rotate=60](0,0)--(-6.287,0)node[below]{ $L_5$};
\draw[dotted,rotate=120](0,0)--(6.287,0)node[above]{ $L_3$};
\draw[dotted,rotate=120](0,0)--(-6.287,0)node[below]{ $L_6$};

\draw[thick](2.8,0)--(3.487,0.4)--(4.2,0);
\draw[thick](5.6,0)--(6.287,0.4);
\filldraw (1.4,0)circle(0.04cm);
\filldraw (2.8,0)circle(0.04cm);
\filldraw (4.2,0)circle(0.04cm);
\filldraw (5.6,0)circle(0.04cm);
\filldraw [red] (3.3,0)circle(0.04cm);
\node [below] at (3.3,-0.15) {\textcolor{red}{\tiny{$1$}}};
\node [below] at (1.4,-0.1) {\tiny{$p_1$}};
\node [below] at (2.8,-0.1) {\tiny{$p_0$}};
\node [below] at (4.2,-0.1) {\tiny{$\frac{1}{p_0}$}};
\node [below] at (5.6,-0.1) {\tiny{$\frac{1}{p_1}$}};
\draw[thick](1.4,0)--(2.8,0);
\draw[thick](4.2,0)--(5.6,0);
\draw[thick,rotate=60](1.4,0)--(2.8,0);
\draw[thick,rotate=60](4.2,0)--(5.6,0);
\draw[thick,rotate=120](1.4,0)--(2.8,0);
\draw[thick,rotate=120](4.2,0)--(5.6,0);
\draw[thick,rotate=180](1.4,0)--(2.8,0);
\draw[thick,rotate=180](4.2,0)--(5.6,0);
\draw[thick,rotate=240](1.4,0)--(2.8,0);
\draw[thick,rotate=240](4.2,0)--(5.6,0);
\draw[thick,rotate=300](1.4,0)--(2.8,0);
\draw[thick,rotate=300](4.2,0)--(5.6,0);

\draw[thick,rotate=60](2.8,0)--(3.487,0.4)--(4.2,0);
\draw[thick,rotate=60](5.6,0)--(6.287,0.4);
\filldraw [rotate=60](1.4,0)circle(0.04cm);
\filldraw [rotate=60](2.8,0)circle(0.04cm);
\filldraw [rotate=60](4.2,0)circle(0.04cm);
\filldraw [rotate=60](5.6,0)circle(0.04cm);
\filldraw [red,rotate=60] (3.3,0)circle(0.04cm);
\node [right] at (2,2.858) {\textcolor{red}{\tiny{$-\omega^2$}}};
\node [right] at (0.8,1.212) {\tiny{$-\omega^2p_1$}};
\node [right] at (1.5,2.425) {\tiny{$-\omega^2 p_0$}};
\node [right] at (2.2,3.637) {\tiny{$-\omega^2\frac{1}{p_0}$}};
\node [right] at (2.9,4.85) {\tiny{$-\omega^2\frac{1}{p_1}$}};

\draw[thick,rotate=120](2.8,0)--(3.487,0.4)--(4.2,0.07);
\draw[thick,rotate=120](5.6,0)--(6.287,0.4);
\filldraw [rotate=120](1.4,0)circle(0.04cm);
\filldraw [rotate=120](2.8,0)circle(0.04cm);
\filldraw [rotate=120](4.2,0)circle(0.04cm);
\filldraw [rotate=120](5.6,0)circle(0.04cm);
\filldraw [red,rotate=120] (3.3,0)circle(0.04cm);
\node [left] at (-2,2.858) {\textcolor{red}{\tiny{$\omega$}}};
\node [left] at (-0.8,1.212) {\tiny{$\omega p_1$}};
\node [left] at (-1.5,2.425) {\tiny{$\omega p_0$}};
\node [left] at (-2.2,3.637) {\tiny{$\omega \frac{1}{p_0}$}};
\node [left] at (-2.9,4.85) {\tiny{$\omega \frac{1}{p_1}$}};

\draw[thick,rotate=180](2.8,0)--(3.487,0.4)--(4.2,0);
\draw[thick,rotate=180](5.6,0)--(6.287,0.4);
\filldraw [rotate=180](1.4,0)circle(0.04cm);
\filldraw [rotate=180](2.8,0)circle(0.04cm);
\filldraw [rotate=180](4.2,0)circle(0.04cm);
\filldraw [rotate=180](5.6,0)circle(0.04cm);
\filldraw [red,rotate=180] (3.3,0)circle(0.04cm);
\node [below] at (-3.3,-0.15) {\textcolor{red}{\tiny{$-1$}}};
\node [below] at (-1.4,-0.1) {\tiny{$-p_1$}};
\node [below] at (-2.8,-0.1) {\tiny{$-p_0$}};
\node [below] at (-4.2,-0.1) {\tiny{$-\frac{1}{p_0}$}};
\node [below] at (-5.6,-0.1) {\tiny{$\frac{1}{p_1}$}};

\draw[thick,rotate=240](2.8,0)--(3.487,0.4)--(4.2,0);
\draw[thick,rotate=240](5.6,0)--(6.287,0.4);
\filldraw [rotate=240](1.4,0)circle(0.04cm);
\filldraw [rotate=240](2.8,0)circle(0.04cm);
\filldraw [rotate=240](4.2,0)circle(0.04cm);
\filldraw [rotate=240](5.6,0)circle(0.04cm);
\filldraw [red,rotate=240] (3.3,0)circle(0.04cm);
\node [left] at (-2,-2.858) {\textcolor{red}{\tiny{$\omega^2$}}};
\node [left] at (-0.8,-1.212) {\tiny{$\omega^2p_1$}};
\node [left] at (-1.5,-2.425) {\tiny{$\omega^2p_0$}};
\node [left] at (-2.2,-3.637) {\tiny{$\omega^2\frac{1}{p_0}$}};
\node [left] at (-2.9,-4.85) {\tiny{$\omega^2\frac{1}{p_1}$}};

\draw[thick,rotate=300](2.8,0)--(3.487,0.4)--(4.2,0);
\draw[thick,rotate=300](5.6,0)--(6.287,0.4);
\filldraw [rotate=300](1.4,0)circle(0.04cm);
\filldraw [rotate=300](2.8,0)circle(0.04cm);
\filldraw [rotate=300](4.2,0)circle(0.04cm);
\filldraw [rotate=300](5.6,0)circle(0.04cm);
\filldraw [red,rotate=300] (3.3,0)circle(0.04cm);
\node [right] at (2,-2.858) {\textcolor{red}{\tiny{$-\omega$}}};
\node [right] at (0.8,-1.212) {\tiny{$-\omega p_1$}};
\node [right] at (1.5,-2.425) {\tiny{$-\omega p_0$}};
\node [right] at (2.2,-3.637) {\tiny{$-\omega \frac{1}{p_0}$}};
\node [right] at (2.9,-4.85) {\tiny{$-\omega\frac{1}{p_1}$}};

\draw[thick](-1.4,0)--(-0.7,0.4);
\draw[thick](1.4,0)--(0.7,0.4);
\draw [thick](0,0.8)--(0,-0.8);
\draw[thick](2.8,0)--(3.487,-0.4)--(4.2,0);
\draw[thick](5.6,0)--(6.287,-0.4);
\draw [thick](3.487,-0.4)--(3.487,0.4);

\draw[thick,rotate=60](-1.4,0)--(-0.7,0.4);
\draw[thick,rotate=60](1.4,0)--(0.7,0.4);
\draw [thick,rotate=60](0,0.8)--(0,-0.8);
\draw[thick,rotate=60](2.8,0)--(3.487,-0.4)--(4.2,-0);
\draw[thick,rotate=60](5.6,0)--(6.287,-0.4);
\draw [thick,rotate=60](3.487,-0.4)--(3.487,0.4);

\draw[thick,rotate=120](-1.4,0)--(-0.7,0.4);
\draw[thick,rotate=120](1.4,0)--(0.7,0.4);
\draw [thick,rotate=120](0,0.8)--(0,-0.8);
\draw[thick,rotate=120](2.8,0)--(3.487,-0.4)--(4.2,0);
\draw[thick,rotate=120](5.6,0)--(6.287,-0.4);
\draw [thick,rotate=120](3.487,-0.4)--(3.487,0.4);

\draw[thick,rotate=180](-1.4,0)--(-0.7,0.4);
\draw[thick,rotate=180](1.4,0)--(0.7,0.4);
\draw[thick,rotate=180](2.8,0)--(3.487,-0.4)--(4.2,0);
\draw[thick,rotate=180](5.6,0)--(6.287,-0.4);
\draw [thick,rotate=180](3.487,-0.4)--(3.487,0.4);

\draw[thick,rotate=240](-1.4,0)--(-0.7,0.4);
\draw[thick,rotate=240](1.4,0)--(0.7,0.4);
\draw[thick,rotate=240](-1.4,0)--(-0.7,-0.4);
\draw[thick,rotate=240](1.4,0)--(0.7,-0.4);
\draw[thick,rotate=240](2.8,0)--(3.487,-0.4)--(4.2,0);
\draw[thick,rotate=240](5.6,0)--(6.287,-0.4);
\draw [thick,rotate=240](3.487,-0.4)--(3.487,0.4);

\draw[thick,rotate=300](-1.4,0)--(-0.7,0.4);
\draw[thick,rotate=300](1.4,0)--(0.7,0.4);
\draw[thick,rotate=300](-1.4,0)--(-0.7,-0.4);
\draw[thick,rotate=300](1.4,0)--(0.7,-0.4);
\draw[thick,rotate=300](2.8,0)--(3.487,-0.4)--(4.2,0);
\draw[thick,rotate=300](5.6,0)--(6.287,-0.4);
\draw [thick,rotate=300](3.487,-0.4)--(3.487,0.4);

\filldraw [red] (0,0) circle [radius=0.04];
\node [right] at (-0.1,-0.35) {\textcolor{red}{\tiny$0$}};
\end{tikzpicture}
\caption{\footnotesize  The contour $\Sigma^{(3)}$. }
\label{FIGE2}
\end{figure}

For convenience, we use
\begin{equation*}
	f^*(z):=\overline{f(\bar{z})}, \quad z \in \mathbb{C},
\end{equation*}
to denote the Schwartz conjugation for a complex-valued function $f(z)$.
Using  these  rays defined above, we  define  new contours   obtained when  opening   jump contours
  $\omega^l \mathbb{R} \setminus \cup_{j=1}^4 I^l_j, \ l=0,1,2$:
\begin{align*}
    & \Sigma^{l}_{p} = \Sigma^{l}_{2,3}  \cup \Sigma^{l}_{6,7}, \quad I^l = \cup_{j=1}^4 I^l_j,\\
	&\tilde{\Sigma}^{l}= ({\underset{j=1,\cdots,8}{\cup}}(\Sigma_{j}^{l} \cup (\Sigma_{j}^{l})^* ) ) \cup \Sigma^{l}_{p} \cup (\Sigma^{l}_{p})^*,\\	&\Sigma^{(3)}=\cup_{l=0,1,2} (\tilde{\Sigma}^{l} \cup I^l),
\end{align*}
as shown in Figure  \ref{FIGE2}.
Additionally, we   define the open domains along the jump contours  $\omega^l \mathbb{R} \setminus I^l, \ l=0,1,2$:
\begin{align}
	&\Omega=\underset{l=0,1,2}{\underset{j=1,\cdots,8}{\cup}} \Omega_{j}^{l} \cup (\Omega_{j}^{l})^*,\nonumber
\end{align}

From Figure \ref{figurec},  we open the contours $ \mathbb{R} \setminus I^0$ via continuous extensions of the jump matrix $V^{(3)}(z)$ by defining appropriate functions and other contours can be opened by the symmetries.
We can construct a matrix function $\mathcal{R}^{(3)}$ like in \cite{fNLS}. The difference is here  we have extra singularity on the  boundary.
Hence, to deal with the singularity at $\varkappa_k$, $k=1,\cdots,6$, we need to introduce a fixed cutoff function $\mathcal{X}(z)$   in $C^\infty_0(\mathbb{R},[0,1])$ with  support near $1$ with
\begin{align}\label{mathcalx}
	\mathcal{X}(z)=\left\{\begin{array}{llll}
		0, & |z-1|>2\varepsilon,\\[4pt]
		1,  & |z-1|<\varepsilon,
	\end{array}\right.
\end{align}
where  $\varepsilon $ is a small enough positive constant satisfying
 the   support of $\mathcal{X}(z)$ doesn't contain any of phase points with
$\varepsilon <\frac{1-p_0^2}{16 p_0}$.

We now define the continuous extension functions in this case: for $j = 1,\cdots,8$,
\begin{equation}
\mathcal{R}^{(3)}(z)=\left\{\begin{array}{lll}
\left(\begin{array}{ccc}
	1 & 0 & 0\\
	-R_{j}(z)e^{-it\theta_{12}} & 1 & 0\\
	0 & 0 & 1
\end{array}\right),  &z\in \Omega_{j}^0,\\
\left(\begin{array}{ccc}
	1 & -R_{j}^*(z)e^{it\theta_{12}} & 0\\
	0 & 1 & 0\\
	0 & 0 & 1
\end{array}\right), & z\in \Omega_{j}^{0*},\\
\left(\begin{array}{ccc}
	1 & 0 & 0\\
	0 & 1 & 0\\
	-R_{j}(\omega^2z)e^{-it\theta_{13}} & 0 & 1
\end{array}\right),  &z\in \Omega^1_j,\\
\left(\begin{array}{ccc}
	1 & 0 & -R_{j}^*(\omega^2z)e^{it\theta_{13}}\\
	0 & 1 & 0\\
	0 & 0 & 1
\end{array}\right), & z\in \Omega^{1*}_j,\\
\left(\begin{array}{ccc}
	1 & 0 & 0\\
	0 & 1 & -R_{j}(\omega z)e^{it\theta_{23}}\\
	0 & 0 & 1
\end{array}\right),  &z\in \Omega^{2}_j,\\
\left(\begin{array}{ccc}
	1 & 0 &0\\
	0 & 1 & 0\\
	0 & -R_{j}^*(\omega z)e^{-it\theta_{23}} & 1
\end{array}\right), & z\in \Omega^{2*}_j,\\
I,  &elsewhere,\\
\end{array}\right.\label{R(2)1}
\end{equation}
where  the functions $R_{j}(z)$, $j=1,\cdots,8$ are given by the following proposition.

\begin{Proposition}\label{proR1}
	Define  functions $R_{j}$: $\overline{\Omega}_{j}^0\to \mathbb{C}$, $j=1,\cdots,8$, continuous on $\overline{\Omega}_{j}^0$, with continuous first partials on  $\Omega_{j}^0$, and boundary values
\begin{align}
	R_{j}(z)=\begin{cases}
		\tilde{r}^*(z),  &z\in \mathbb{R},\\
		\tilde{r}^*(\frac{1}{p_1}),  &z\in \Sigma_1^0,\\
        \tilde{r}^*(\frac{1}{p_0}),  &z\in \Sigma_2^0,\\
        \tilde{r}^*(p_0),  &z\in \Sigma_3^0,\\
        \tilde{r}^*(p_1),  &z\in \Sigma_4^0,\\
        \tilde{r}^*(-p_1),  &z\in \Sigma_5^0,\\
        \tilde{r}^*(-p_0),  &z\in\Sigma_6^0,\\
        \tilde{r}^*(-\frac{1}{p_0}),  &z\in \Sigma_7^0,\\
        \tilde{r}^*(-\frac{1}{p_1}),  &z\in \Sigma_8^0,
      \end{cases}
	\end{align}	
	such that
	\begin{align}
&|\bar{\partial}R_{j}(z)|\lesssim|\tilde{r}^{*'}(\re z)|+|\mathcal{X}'( \re z)|+|\re z-\xi_j|^{-1/2}, \text{for all $z\in \Omega_{j}^0$},\label{dbarRj3}	\\
&|\bar{\partial}R_{j}(z)|\lesssim|\tilde{r}^{*'}(\re z)|+|\mathcal{X}'(\re z)|,\ \text{for all $z\in \Omega_{j}^0$},\label{dbarRj4}\\
	&R_{j}(z)=\bar{\partial}R_{j}(z)=0, \ \text{ for all }z\in\Omega_{j}^0 \text{ with } |\re z\pm1|<\varepsilon,\label{R11}\\	
	&\bar{\partial}R_{j}(z)=0,\hspace{0.5cm}\text{ elsewhere},
	\end{align}
where $\xi_j,\ j=1,\cdots,8$ stands for the saddle points in the corresponding regions $\Omega_j^0,\ j=1,\cdots,8$.
Setting $R: \Omega \to \mathbb{C}$ by $R(z)|_{z \in \Omega_j^0} = R_j(z)$, the extension can preserve the symmetry $R(z) = \overline{R(\bar{z}^{-1})} $.
\end{Proposition}
\begin{proof}
 Without loss of generality, we give the details for $R_1(z)$ and the other cases are easily inferred.
 The extension of $R_1(z)$ can be constructed by
 \begin{equation}\label{defr1}
 R_1(z) = (1-\mathcal{X}(\re z)) \left(\tilde{r}^*(\re z) - \tilde{r}^*(\frac{1}{p_1}) \right) \cos \left( \frac{\pi}{2\varphi_0} \arg (z-\frac{1}{p_1}) \right) + \tilde{r}^*(\frac{1}{p_1}).
 \end{equation}
Note that for $z-\frac{1}{p_1}=le^{i\varphi}=u+vi$, $l,\varphi,u,v\in\mathbb{R}$, we have  $\bar{\partial}=\frac{1}{2}\left( \partial_u+i\partial_v\right) =\frac{e^{i\varphi}}{2}\left( \partial_l+il^{-1}\partial_\varphi\right)$.
Applying $\bar \partial$ operator to \eqref{defr1}, it is readily seen that
\begin{align}
\bar{\partial}R_{1}(z)&= \frac{1}{2} \tilde{r}^{*'}(u) \cos \left( \frac{\pi \varphi}{2 \varphi_0}\right)
 -  \frac{1}{2}\mathcal{X}'(u) \left(\tilde{r}^*(u) - \tilde{r}^*(\frac{1}{p_1}) \right) \cos \left( \frac{\pi \varphi}{2 \varphi_0}\right)\nonumber\\
	& +\frac{1}{2}(1-\mathcal{X}(u)) \left(\tilde{r}^*(u) - \tilde{r}^*(\frac{1}{p_1}) \right)  \bar \partial \cos \left(\frac{\pi \varphi}{2 \varphi_0}\right).
	\end{align}
From the definition \eqref{deftilder} and   H\"{o}lder's  inequality, we obtain
	\begin{align}
	|\tilde{r}^*(u)-\tilde{r}^*(\frac{1}{p_1})|=|\int_{\frac{1}{p_1}}^{u}	\tilde{r}^{*'}(s)ds|\leq\parallel \tilde{r}^{*'}\parallel_{L^2(\mathbb{R})}|u-\frac{1}{p_1}|^{1/2},
	\end{align}
which yields \eqref{dbarRj3}. Moreover,
\eqref{dbarRj4} is obtained from the boundedness of $\tilde{r}'(z)$.

\end{proof}
In addition, $\mathcal{R}^{(3)}$ achieves the symmetry:
\begin{equation}	\mathcal{R}^{(3)}(z)=\Gamma_1\overline{\mathcal{R}^{(3)}(\bar{z})}\Gamma_1=\Gamma_2\overline{\mathcal{R}^{(3)}(\omega^2\bar{z})}\Gamma_2=\Gamma_3\overline{\mathcal{R}^{(3)}(\omega\bar{z})}\Gamma_3=\overline{\mathcal{R}^{(3)}(\bar{z}^{-1})}.
\end{equation}

\subsection{A hybrid $\bar{\partial}$-RH problem and its decomposition  }\label{sec5}

Now we  use $\mathcal{R}^{(3)}$ to define a new transformation
\begin{equation}
	M^{(3)}(z):= M^{(3)}(z;y,t) =M^{(2)}(z)\mathcal{R}^{(3)}(z)\label{transm2},
\end{equation}
which satisfies the following hybrid $\bar{\partial}$-RH problem.

\begin{RHP}\label{RHP4}
Find a matrix valued function  $ M^{(3)}(z)$ with following properties:
\begin{itemize}
\item $M^{(3)}(z)$  has  sectionally continuous first partial derivatives in
$\mathbb{C}\setminus \left( \Sigma^{(3)}\cup \left\lbrace\zeta_n \right\rbrace_{n-k N_0\in\lozenge} \right) $,  and is meromorphic outside $\bar{\Omega}$.

\item $M^{(3)}(z)=\Gamma_1\overline{M^{(3)}(\bar{z})}\Gamma_1=\Gamma_2\overline{M^{(3)}(\omega^2\bar{z})}\Gamma_2=\Gamma_3\overline{M^{(3)}(\omega\bar{z})}\Gamma_3=\overline{M^{(3)}(\bar{z}^{-1})}$.

\item $M^{(3)}(z)$ has continuous boundary
values $M^{(3)}_\pm(z)$ on $\Sigma^{(3)}$ with
\begin{equation*}
	M^{(3)}_+(z)=M^{(3)}_-(z)V^{(3)}(z),\hspace{0.5cm}z \in \Sigma^{(3)},
\end{equation*}
where
\begin{equation}\label{p1jumpv3}
V^{(3)}(z)=\left\{
\begin{array}{ll}
\left(\begin{array}{ccc}
	1 & -\overline{\tilde{r}(z)} e^{it\theta_{12}(z)}&0\\
	0 & 1&0\\
	0 & 0 &1
\end{array}\right)\left(\begin{array}{ccc}
	1 & 0 & 0\\
	\tilde{r}(z)e^{-it\theta_{12}(z)} & 1 & 0\\
	0 & 0 &1
\end{array}\right),   & z\in I^0,\\[12pt]
\left(\begin{array}{ccc}
	1 & 0&0\\
	0 & 1&-\overline{\tilde{r}(\omega z)}e^{it\theta_{23}(z)}\\
	0 & 0 &1
\end{array}\right)\left(\begin{array}{ccc}
	1 & 0 & 0\\
	0 & 1 & 0\\
	0 & 	\tilde{r}(\omega z)e^{-it\theta_{23}(z)}&1
\end{array}\right),   & z\in I^{2},\\[12pt]
\left(\begin{array}{ccc}
	1 & 0&0\\
	0 & 1&0\\
	-\overline{\tilde{r}(\omega^2 z)}e^{-it\theta_{13}(z)}& 0 &1
\end{array}\right)\left(\begin{array}{ccc}
	1 & 0 & 	\tilde{r}(\omega^2 z)we^{it\theta_{13}(z)}\\
	0 & 1 & 0\\
	0 & 0 &1
\end{array}\right),   & z\in I^{1},\\
\mathcal{R}^{(3)}(z)|_{z\in \Omega_{j+1} } - \mathcal{R}^{(2)}(z)|_{z\in \Omega_{j}}, & z \in \Sigma_{p}^{l}, l=0,1,2,\\
\mathcal{R}^{(3)}(z)|_{z\in \Omega^*_{j+1} } - \mathcal{R}^{(2)}(z)|_{z\in \Omega^*_{j}}, & z \in \Sigma_{p}^{l*}, l=0,1,2\\
\mathcal{R}^{(3)}(z),&	z\in \tilde{\Sigma}^{l},l=0,1,2\\
\mathcal{R}^{(3)}(z)^{-1},& 	z\in\tilde{\Sigma}^{l*}, l=0,1,2.
\end{array}\right.
\end{equation}	
\item $M^{(3)}(z) = I+\mathcal{O}(z^{-1}),\hspace{0.5cm}z \rightarrow \infty$.

\item For $z\in\mathbb{C}$,
\begin{align}
	\bar{\partial}M^{(3)}(z)=M^{(3)}(z)\bar{\partial}\mathcal{R}^{(3)}(z),
\end{align}
where
\begin{equation}
	\bar \partial \mathcal{R}^{(3)}(z)=\left\{\begin{array}{lll}
		\left(\begin{array}{ccc}
			0 & 0 & 0\\
			-\bar \partial  R_{j}(z)e^{-it\theta_{12}} & 0 & 0\\
			0 & 0 & 0
		\end{array}\right),  &z\in \Omega_{j}^0,\\
		\left(\begin{array}{ccc}
			0 & -\bar \partial R_{j}^*(z)e^{it\theta_{12}} & 0\\
			0 &0 & 0\\
			0 & 0 & 0
		\end{array}\right), & z\in \Omega_{j}^{0*},\\
		\left(\begin{array}{ccc}
			0 & 0 & 0\\
			0 & 0 & 0\\
			-\bar \partial R_{j}(\omega^2z)e^{-it\theta_{13}} & 0 & 1
		\end{array}\right),  &z\in \Omega^1_j,\\
		\left(\begin{array}{ccc}
			0 & 0 & -\bar \partial R_{j}^*(\omega^2z)e^{it\theta_{13}}\\
			0 & 0 & 0\\
			0 & 0 & 0
		\end{array}\right), & z\in \Omega^{1*}_j,\\
		\left(\begin{array}{ccc}
			0 & 0 & 0\\
			0 & 0 & -R_{j}(\omega z)e^{it\theta_{23}}\\
			0 & 0 & 0
		\end{array}\right),  &z\in \Omega^{2}_j,\\
		\left(\begin{array}{ccc}
			0 & 0 &0\\
			0 & 0 & 0\\
			0 & -\bar \partial R_{j}^*(\omega z)e^{-it\theta_{23}} & 0
		\end{array}\right), & z\in \Omega^{2*}_j,\\
		{\bf{0}},  &elsewhere,\\
	\end{array}\right.\label{p1dbarr3}
\end{equation}

\item $M^{(3)}(z)$ satisfies the singularity conditions in  \eqref{m2sin1}-\eqref{m2sin3} with  $M^{(3)}(z)$  replacing $M^{(2)}(z)$.

\item $M^{(3)}(z)$ has simple poles at each point $\zeta_n$  for $n -k N_0\in\lozenge$ with
\begin{align}
	&\res_{z=\zeta_n}M^{(3)}(z)=\lim_{z\to \zeta_n}M^{(3)}(z)\left[ T^{-1}(z)B_nT(z)\right].
\end{align}
\end{itemize}
\end{RHP}


 To solve the RH problem \ref{RHP4},  we decompose it into a pure   RH  problem  for $M^{R}(z):= M^{R}(z;y,t)$  with $\bar\partial \mathcal{R}^{(3)}\equiv0$   and a pure $\bar{\partial}$-problem with nonzero $\bar{\partial}$-derivatives.
By omitting the $\bar\partial$-derivative part of RH problem \ref{RHP4}, we obtain the pure   RH problem  for  $M^{R}(z)$   as follows.

\begin{RHP}\label{RHP5}

\begin{itemize}
Find a matrix-valued function  $  M^{R}(z)$ with following properties:

\item $M^{R}(z)$ is  meromorphic  in $\mathbb{C}\setminus \Sigma^{(3)}$.

\item $M^{R}(z)$ has continuous boundary values $M^{R}_\pm(z)$ on $\Sigma^{(3)}$ and
\begin{equation*}
	M^{R}_+(z)=M^{R}_-(z)V^{(3)}(z),\hspace{0.5cm}z \in \Sigma^{(3)},
\end{equation*}
where $V^{(3)}(z)$ is defined by \eqref{p1jumpv3}.
\item  $M^{R}(z)=\Gamma_1\overline{M^{R}(\bar{z})}\Gamma_1=\Gamma_2\overline{M^{R}(\omega^2\bar{z})}\Gamma_2=\Gamma_3\overline{M^{R}(\omega\bar{z})}\Gamma_3=\overline{M^{R}(\bar{z}^{-1})}$.


\item $M^{R}(z) =I+\mathcal{O}(z^{-1}),\hspace{0.5cm}z \rightarrow \infty$.

\item As $z\rightarrow \varkappa_l=e^{\frac{i\pi(l-1)}{3}}, l = 1,\cdots,6$,   the  limit  of $M^{R}(z)$ have  the pole singularities
\begin{align}
	&M^{R}(z)=\frac{1}{z\mp1}\left(\begin{array}{ccc}
		\alpha^{R}_\pm &	\alpha^{R}_\pm & \beta^{R}_\pm \\
		-\alpha^{R}_\pm & -\alpha^{R}_\pm & -\beta^{R}_\pm\\
		0	&	0 & 0
	\end{array}\right)+\mathcal{O}(1),\ z\to\pm 1, \label{resmr1}\\
	&M^{R}(z)=\frac{1}{z\mp\omega^2}\left(\begin{array}{ccc}
		0 &	0 &  0\\
		\beta^{R}_\pm	 & \alpha^{R}_\pm &\alpha^{R}_\pm \\
		-\beta^{R}_\pm	&	-\alpha^{R}_\pm & -\alpha^{R}_\pm
	\end{array}\right)+\mathcal{O}(1),\ z\to\pm \omega^2,\label{resmr2}\\
	&M^{R}(z)=\frac{1}{z\mp\omega}\left(\begin{array}{ccc}
		-\alpha^{R}_\pm &	-\beta^{R}_\pm & -\alpha^{R}_\pm\\
		0	 & 0 &0 \\
		\alpha^{R}_\pm &	\beta^{R}_\pm & \alpha^{R}_\pm
	\end{array}\right)+\mathcal{O}(1),\ z\to\pm \omega,\label{resmr3}
\end{align}
with $\alpha^{R}_\pm=\alpha^{R}_\pm(y,t)=-\bar{\alpha}^{R}_\pm$, $\beta^{R}_\pm=\beta^{R}_\pm(y,t)=-\bar{\beta}^{R}_\pm$ and $M^{R}(z)^{-1}$ has same specific
matrix structure with $\alpha^{R}_\pm$, $\beta^{R}_\pm$ replaced by $\tilde{\alpha}^{R}_\pm$, $\tilde{\beta}^{R}_\pm$.

\item $M^{R}(z)$ has the simple poles at each point $\zeta_n$  for $n -k N_0\in\lozenge$ with
\begin{align*}
	&\res_{z=\zeta_n}M^{R}(z)=\lim_{z\to \zeta_n}M^{R}(z)\left[ T^{-1}(z)B_nT(z)\right]. 
\end{align*}

\end{itemize}
\end{RHP}

To proceed, define $\mathbb{B}_j$ as the neighborhood of $\varkappa_j$, $j=1,\cdots,6$ with
\begin{equation}
\mathbb{B}_j = \{z\in \mathbb{C} \setminus \{\varkappa_j\}_{j=1}^6:  |\re(z/\varkappa_j)-1|<2 \varepsilon, |\im (z/\varkappa_j)|< 2 \varepsilon \}.
\end{equation}
For convenience,  let $z_c = -z_b$ and $z_d = -z_a$.
Then denote
\begin{align}
&U^{l}_{a}= \left\lbrace z\in \mathbb{C}:\left|z-\omega^l z_a\right|\leq \varrho^{0} \right\rbrace, \ U^{l}_{b}= \left\lbrace z\in \mathbb{C}:\left|z-\omega^lz_b\right|\leq \varrho^{0} \right\rbrace, \nonumber\\
&U^{l}_{c}= \left\lbrace z\in \mathbb{C}: \left|z-\omega^l z_c\right|\leq \varrho^{0}\right\rbrace, \ U^{l}_{d}= \left\lbrace z\in \mathbb{C}: \left|z-\omega^l z_d\right|\leq \varrho^{0} \right\rbrace, \nonumber
\end{align}
where
\begin{align*}
\varrho^{0} := \min \left\{\frac{p_1}{2}, \frac{1}{8} |p_i \pm1|, 2\left(p_i -\frac{\sqrt{7}-\sqrt{3}}{2}\right)t^{\delta_1}, 2\left(\frac{1}{p_i} -\frac{\sqrt{7}+\sqrt{3}}{2}\right)t^{\delta_1}, i=0,1 \right\}
\end{align*}
with  $\frac{1}{9}<\delta_1<\frac{1}{6}$, which makes $U_j^l$ and $\mathbb{B}_j$ are disjoint.
Denote $ U:=\underset{l=0,1,2}{\cup} \left( U^{l}_{a} \cup U^{l}_{b} \cup U^{l}_{c}\cup U^{l}_{d} \right)$ as the union set of $U^{l}_{j}$ with $j=a,b,c,d$ and $l=0,1,2$.
For $t$ large enough, we have $\omega^l\frac{1}{p_0}, \omega^l\frac{1}{p_1} \in U^{l}_{a}$, $\omega^lp_1, \omega^lp_0 \in U^{l}_{b}$, $-\omega^lp_0, -\omega^lp_1\in U^{l}_{c}$ and
$-\omega^l\frac{1}{p_1},-\omega^l\frac{1}{p_0} \in U^{l}_{d}$.

From \eqref{p1jumpv3},  we know the jump matrix $V^{(3)}(z)$ uniformly goes to $I$ on $\Sigma^{(3)}$  except in $U$, which
 enlightens  us to construct the solution $M^{R}(z)$ as follows:
\begin{equation}
M^{R}(z)=\left\{\begin{array}{ll}
E(z)M^{O}(z), & z\notin U \cup \mathbb{B}_j,\\
E(z)M^{O}(z)M^{L}(z),  &z\in U,\\
E(z)M^{O}(z)M^{B}_j(z), & z\in \mathbb{B}_j.
\end{array}\right. \label{transm4}
\end{equation}
where $M^{O}(z)$ is an outer model including  the influence of solitons, $M^{L}(z)$ is a local model which can be well approximated by the Painlev\'e \uppercase\expandafter{\romannumeral2}  RH model, $M^{B}_j(z)$ is a solution of a RH problem which only has jumps near $\varkappa_j$, and $E(z)$  is the error function which we will prove exists and bound it asymptotically.
Then we use $M^{R}(z)$ to construct  a new matrix function
\begin{equation}
M^{(4)}(z):= M^{(4)}(z;y,t)=M^{(3)}(z)M^{R}(z)^{-1}.\label{transm3}
\end{equation}
which   removes   analytical component  $M^{R}(z)$    to get  a  pure $\bar{\partial}$-problem.

\noindent\textbf{$\bar{\partial}$-problem}. Find a matrix-valued function  $ M^{(4)}(z) :=
M^{(4)}(z;y,t) $ such that

\begin{itemize}

 \item $M^{(4)}(z)$  has sectionally continuous first partial derivatives in $\mathbb{C}$.

\item  $M^{(4)}(z)=I+\mathcal{O}(z^{-1}),\hspace{0.5cm}z \rightarrow \infty$.

\item $M^{(4)}(z)$ satisfies the $\bar\partial$-equation
$$\bar{\partial}M^{(4)}(z)=M^{(4)}(z)W^{(4)}(z),\ \ z\in \mathbb{C},$$
where
\begin{equation}\label{p1w4}
W^{(4)}(z)=M^{R}(z)\bar{\partial}\mathcal{R}^{(3)}(z)M^{R}(z)^{-1}.
\end{equation}

\end{itemize}

\begin{proof}
	The proof of this pure $\bar \partial$-problem can be given similarly as shown in Section 4.3 of \cite{novYF}.

\end{proof}

The existence  and asymptotics  of   the above pure $\bar \partial$-problem for $M^{(4)}(z)$ will be shown in   Section \ref{sec7}.

\subsection{Contribution from discrete spectrum}\label{sec6}

Now we construct a model solution outside $U$ which ignores the jumps completely.
The outer model $M^{O}(z)$ satisfies the following RH problem.

\begin{RHP}\label{RHP51}
Find a matrix-valued function  $ M^{O} (z)$ with the following properties:
\begin{itemize}
\item $M^{O} (z)$ is analytical  in $\mathbb{C}\setminus \left\lbrace\zeta_n \right\rbrace_{n-k N_0\in \lozenge} $.

\item $M^{O} (z)=\Gamma_1\overline{M^{r} (\bar{z})}\Gamma_1
=\Gamma_2\overline{M^{r} (\omega^2\bar{z})}\Gamma_2=\Gamma_3\overline{M^{r} (\omega\bar{z})}\Gamma_3=\overline{M^{r} (\bar{z}^{-1})}$.

 \item $M^{O} (z) = I+\mathcal{O}(z^{-1}),\hspace{0.5cm}z \rightarrow \infty$.

 \item As $z\rightarrow \varkappa_l=e^{\frac{i\pi(l-1)}{3}}, l = 1,\cdots,6$,   the  limit  of $M^{O}(z)$ have  the pole singularities
\begin{align}
	&M^{O}(z)=\frac{1}{z\mp1}\left(\begin{array}{ccc}
		\alpha^{O}_\pm &	\alpha^{O}_\pm & \beta^{O}_\pm \\
		-\alpha^{O}_\pm & -\alpha^{O}_\pm & -\beta^{O}_\pm\\
		0	&	0 & 0
	\end{array}\right)+\mathcal{O}(1),\ z\to\pm 1, \label{mosin}\\
	&M^{O}(z)=\frac{1}{z\mp\omega^2}\left(\begin{array}{ccc}
		0 &	0 &  0\\
		\beta^{O}_\pm	 & \alpha^{O}_\pm &\alpha^{O}_\pm \\
		-\beta^{O}_\pm	&	-\alpha^{O}_\pm & -\alpha^{O}_\pm
	\end{array}\right)+\mathcal{O}(1),\ z\to\pm \omega^2,\\
	&M^{O}(z)=\frac{1}{z\mp\omega}\left(\begin{array}{ccc}
		-\alpha^{O}_\pm &	-\beta^{O}_\pm & -\alpha^{O}_\pm\\
		0	 & 0 &0 \\
		\alpha^{O}_\pm &	\beta^{O}_\pm & \alpha^{O}_\pm
	\end{array}\right)+\mathcal{O}(1),\ z\to\pm \omega,
\end{align}
with $\alpha^{O}_\pm=\alpha^{O}_\pm(y,t)=-\bar{\alpha}^{O}_\pm$, $\beta^{O}_\pm=\beta^{O}_\pm(y,t)=-\bar{\beta}^{O}_\pm$ and $M^{O}(z)^{-1}$ has same specific
matrix structure with $\alpha^{O}_\pm$, $\beta^{O}_\pm$ replaced by $\tilde{\alpha}^{O}_\pm$, $\tilde{\beta}^{O}_\pm$.

\item $M^{O} (z)$ has simple poles at each point $\zeta_n$ for $n-k N_0\in \lozenge$ with
\begin{align}
	& \underset{z=\zeta_n}{\operatorname{Res}}  M^{O} (z) =\lim_{z\to \zeta_n}M^{O} (z)\left[ T^{-1}(z)B_nT(z)\right] \label{resMrsol}.
\end{align}
\end{itemize}
\end{RHP}

The essential fact we need concerning $M^{O}$ is as follows.

\begin{Proposition} \label{p1moutf1}

The unique solution $M^{O} $ of RH problem \ref{RHP51} is given by
\begin{equation}\label{p1mo1}
M^O(z) = M^{ \lozenge}(z;\tilde{\mathcal{D}}),
\end{equation}
where $M^{ \lozenge}$ is the solution of the RH problem  \ref{RHP31} corresponding to the reflectionless
scattering data $\tilde{\mathcal{D}}=\{(\zeta_n,\tilde{C}_n)_{n-k N_0\in \lozenge} \}$. Here the modified connection coefficients are given by
\begin{equation*}
\tilde{C}_n = c_n T_{12}(\zeta_n),
\end{equation*}
with $T_{12}$ defined by \eqref{Tij}.

\end{Proposition}

\begin{proof}
	
	The detailed proof for  the existence and uniqueness of solution of the above RH problem \ref{RHP51} can be found in  Proposition 4.3 of \cite{novYF}.  From this, we know RH problem \ref{RHP51}  has an unique solution $M^O(z)$ with modified scattering data $\tilde{\mathcal{D}} =\left\{   \{\zeta_n,\tilde{C}_n \}_{n-kN_0 \in \lozenge} \}\right\}$, where $\tilde{C}_n = c_n T_{12}(\zeta_n)$.
\end{proof}

Denote $u^\lozenge(y,t;\tilde{\mathcal{D}})$ as the $\mathcal{N}(\lozenge)$-soliton with   scattering data $\tilde{\mathcal{D}}$.
By the reconstruction formula (\ref{recons u}) and (\ref{recons x}), we then have
\begin{corollary}\label{p1urxrsol}
The soliton solution   of the Novikov equation  (\ref{Novikov}) is given by
\begin{align}
	u^\lozenge(y,t;\tilde{\mathcal{D}})
	=&\frac{1}{2}\tilde{m}^\lozenge_1(y,t)\left(\frac{M^O_{33}(e^{\frac{i\pi}{6}};y,t)}{M^O_{11}(e^{\frac{i\pi}{6}};y,t)} \right)^{1/2}\nonumber\\
	&+ \frac{1}{2}\tilde{m}^\lozenge_3(y,t)\left(\frac{M^O_{33}(e^{\frac{i\pi}{6}};y,t)}{M^O_{11}(e^{\frac{i\pi}{6}};y,t)} \right)^{-1/2}-1,\label{recons ur}
\end{align}
in which
\begin{align}\label{recons xr}
&x^\lozenge(y,t;\tilde{\mathcal{D}})=y+\frac{1}{2} \ln\frac{M^O_{33}(e^{\frac{i\pi}{6}};y,t)}{M^O_{11}(e^{\frac{i\pi}{6}};y,t)},\
\tilde{m}^\lozenge_l:=\sum_{j=1}^3M^O_{jl}(e^{\frac{i\pi}{6}};y,t),\ l=1,2,3.
\end{align}
\end{corollary}

\subsection{Contribution from jump contours }\label{secpc}
Since we have ignored the jump conditions completely in the above section, now we construct the solutions which match with the jump conditions. Again by the property
\begin{equation}
	\Vert V^{(3)}(z)-I \Vert_{L^q(\Sigma^{(3)} \setminus U)} = \mathcal{O}(e^{-K_q t}), \quad t\to \infty,
\end{equation}
where $K_q$ is a positive constant and $1\le q \le +\infty$, the jump matrix is exponentially close to the identity outside of $U$ and hence we need to investigate the local properties
 near the saddle points.

\subsubsection{Local model near  critical  points } \label{p1loc}

We define a  new local contour
$$\Sigma^{L}= {\Sigma}^{L,0} \cup  {\Sigma}^{L,1}\cup {\Sigma}^{L,2}, $$
where $  {\Sigma}^{L,l}, \ l=0,1,2$ are the local contours on  jump  contours $\omega^n \mathbb{R}$, respectively
\begin{align}
	&  {\Sigma}^{L,l}= ({\underset{j=a,b,c,d}{\cup}}(\Sigma_{j}^{^l} \cup \Sigma_{j}^{l*} ) \cup I^l)\cap U,\  l=0, 1,2. \nonumber
\end{align}

 Since there are  $12$ colliding points
 $\omega^l\xi_{j}, l=0,1,2, j=a,b,c,d$,  so
the  entire  local jump contour $\Sigma^{L}$  consists     $12$  separate local  jumps  in this case.  See Figure \ref{sigma0}.
Further  denote  the  local  jump  for each colliding  point $\omega^l\xi_{j}$
\begin{align}
	& \Sigma^{L,l}_j= (\Sigma_{j}^{l} \cup \Sigma_{j}^{l*} \cup I^l_j )\cap U^{l}_j, \ l=0, 1,2,\ j=a,b,c,d. \nonumber
\end{align}
 We consider the following  local  RH problem.
\begin{figure}[htp]
	\centering
		\begin{tikzpicture}[scale=0.85]
		\draw[dashed](0,0)--(5,0)node[right]{$L_1$};
		\draw[dashed](0,0)--(-5,0)node[left]{$L_4$};
		\coordinate (I) at (0,0);
		\fill (I) circle (1pt) node[below] {$0$};
		\coordinate (c) at (-3,0);
		\draw[thick](-0.8,0)--(-0.4,0.2);
		\draw[thick](-0.8,0)--(-0.4,-0.2);
		\draw[thick](-0.8,0)--(-1.8,0);
			\coordinate (A11) at (-1.3,0);
			\fill (A11) circle (1pt) node[below] {\tiny$z_c$};
	
		\draw[thick](-1.8,0)--(-2.2,-0.2);
		\draw[thick](-1.8,0)--(-2.2,0.2);
		\draw[thick](-4.3,0)--(-4.7,0.2);
		\draw[thick](-4.3,0)--(-4.7,-0.2);
		\draw[thick](-3.2,0)--(-2.8,0.2);
		\draw[thick](-3.2,0)--(-2.8,-0.2);
		\draw[thick](-3.2,0)--(-4.3,0);
\coordinate (B1) at (-3.75,0);
			\fill (B1) circle (1pt) node[below] {\tiny$z_d$};
	
		\draw[thick](0.8,0)--(0.4,0.2);

		\draw[thick](0.8,0)--(0.4,-0.2);

		\draw[thick](1.8,0)--(2.2,-0.2);
		\draw[thick](1.8,0)--(2.2,0.2);
		\draw[thick](1.8,0)--(0.8,0);
\coordinate (C1) at (1.3,0);
			\fill (C1) circle (1pt) node[below] {\tiny$z_b$};

		\draw[thick,red](4.3,0)--(4.7,0.2);
		\draw[thick,red](4.3,0)--(4.7,-0.2);
		\draw[thick,red](3.2,0)--(2.8,0.2);
		\draw[thick,red](3.2,0)--(2.8,-0.2);
		\draw[thick,red](3.2,0)--(4.3,0);
\coordinate (D1) at (3.75,0);
			\fill[red] (D1) circle (1pt) node[below] {\tiny$z_a$};

		\draw[dashed](0,0)--(2.4,4.157)node[above]{ $L_2$};	
		\draw[dashed](0,0)--(2.4,-4.157)node[below]{ $L_6$};;
		\draw[dashed](0,0)--(-2.4,4.157)node[above]{ $L_3$};
		\draw[dashed](0,0)--(-2.4,-4.157)node[below]{ $L_5$};;
		\coordinate (I) at (0,0);
		\fill (I) circle (1pt) node[below] {$0$};
		\coordinate (A) at (-4.3,0);
		\fill (A) circle (1pt) node[below] {\tiny$-\frac{1}{p_1}$};
		\coordinate (b) at (-3.2,0);
		\fill (b) circle (1pt) node[below] {\tiny$-\frac{1}{p_0}$};
		\coordinate (C) at (-0.8,0);
		\fill (C) circle (1pt) node[below] {\tiny$-p_1$};
		\coordinate (d) at (-1.8,0);
		\fill (d) circle (1pt) node[below] {\tiny$-p_0$};
		\coordinate (E) at (4.3,0);
		\fill[red] (E) circle (1pt) node[below] {\tiny$\frac{1}{p_1}$};
		\coordinate (R) at (3.2,0);
		\fill[red] (R) circle (1pt) node[below] {\tiny$\frac{1}{p_0}$};
		\coordinate (T) at (0.8,0);
		\fill (T) circle (1pt) node[below] {\tiny$p_1$};
		\coordinate (Y) at (1.8,0);
		\fill (Y) circle (1pt) node[below] {\tiny$p_0$};
		\coordinate (A1) at (0.4,0.69);
		\fill (A1) circle (1pt) node[right] {\tiny$- \omega^2 p_1$};
		\draw[thick](0.4,0.69)--(0.32,0.3);
		\draw[thick](0.4,0.69)--(0.1,0.43);
		\coordinate (A2) at (0.4,-0.69);
		\fill (A2) circle (1pt) node[right] {\tiny$-\omega p_1$};
	
  	    \draw[thick](0.4,-0.69)--(0.32,-0.3);
		\draw[thick](0.4,-0.69)--(0.1,-0.43);
		\coordinate (A3) at (-0.4,0.69);
		\fill (A3) circle (1pt) node[left] {\tiny$\omega p_1$};
		\draw[thick](-0.4,0.69)--(-0.32,0.3);
		\draw[thick](-0.4,0.69)--(-0.1,0.43);
		\coordinate (A4) at (-0.4,-0.69);
		\fill (A4) circle (1pt) node[left] {\tiny$\omega^2 p_1$};
		\draw[thick](-0.4,-0.69)--(-0.32,-0.3);
		\draw[thick](-0.4,-0.69)--(-0.1,-0.43);
		\coordinate (n1) at (1.6,2.77);
		\fill (n1) circle (1pt) node[right] {\tiny$-\omega^2\frac{1}{p_0}$};

		\draw[thick](1.6,2.77)--(1.52,2.37);
		\draw[thick](1.6,2.77)--(1.28,2.5);
		\coordinate (m1) at (1.6,-2.77);
		\fill (m1) circle (1pt) node[right] {\tiny$-\omega\frac{1}{p_0}$};
		\coordinate (j1) at (-1.6,2.77);
		\fill (j1) circle (1pt) node[left] {\tiny$\omega\frac{1}{p_0}$};
		\draw[thick](1.6,-2.77)--(1.52,-2.37);

		\draw[thick](1.6,-2.77)--(1.28,-2.5);
		\draw[thick](-1.6,2.77)--(-1.52,2.37);
		\draw[thick](-1.6,2.77)--(-1.28,2.5);
		\coordinate (k1) at (-1.6,-2.77);
		\fill (k1) circle (1pt) node[left] {\tiny$\omega^2\frac{1}{p_0}$};
		\draw[thick](-1.6,-2.77)--(-1.52,-2.37);
		\draw[thick](-1.6,-2.77)--(-1.28,-2.5);
		\coordinate (n2) at (0.9,1.56);
		\fill[thick] (n2) circle (1pt) node[right] {\tiny$- \omega^2 p_0$};
		\draw[thick](0.9,1.56)--(1.02,2.07);
		\draw[thick](0.9,1.56)--(1.3,1.9);
		\draw[thick](-0.9,-1.56)--(-1.02,-2.07);
		\draw[thick](-0.9,-1.56)--(-1.3,-1.9);

         \draw[thick](0.9,1.56)--(0.4,0.69);
         \draw[thick](-0.9,-1.56)--(-0.4,-0.69);
         \draw[thick](0.9,-1.56)--(0.4,-0.69);
         \draw[thick](-0.9,1.56)--(-0.4,0.69);

                  \coordinate (l5) at (0.65,1.125);
			\fill (l5) circle (1pt) node[right] {\tiny$ \omega^2 z_c$};
 \coordinate (l6) at (-0.65,-1.125);
			\fill (l6) circle (1pt) node[left] {\tiny$ \omega^2z_b$};

\coordinate (l7) at (0.65,-1.125);
			\fill (l7) circle (1pt) node[right] {\tiny$ \omega z_c$};
 \coordinate (l8) at (-0.65,1.125);
			\fill (l8) circle (1pt) node[left] {\tiny$ \omega z_b$};
		\coordinate (m2) at (0.9,-1.56);
		\fill (m2) circle (1pt) node[right] {\tiny$-\omega p_0$};
		\draw[thick](0.9,-1.56)--(1.02,-2.07);
		\draw[thick](0.9,-1.56)--(1.3,-1.9);
		\coordinate (j2) at (-0.9,1.56);
		\fill (j2) circle (1pt) node[left] {\tiny$\omega p_0$};
		\coordinate[blue] (k2) at (-0.9,-1.56);
		\fill (k2) circle (1pt) node[left] {\tiny$\omega^2p_0$};
		\draw[thick](-0.9,1.56)--(-1.02,2.07);
		\draw[thick](-0.9,1.56)--(-1.3,1.9);
		\coordinate (n23) at (2.15,3.72);
		\fill (n23) circle (1pt) node[right] {\tiny$-\omega^2\frac{1}{p_1}$};
		\draw[thick](2.15,3.72)--(2.23,4.17);
		\draw[thick](2.15,3.72)--(2.48,4.03);

		\coordinate (m23) at (2.15,-3.72);
		\fill (m23) circle (1pt) node[right] {\tiny$-\omega\frac{1}{p_1}$};
		\coordinate (j23) at (-2.15,3.72);
		\fill (j23) circle (1pt) node[left] {\tiny$\omega\frac{1}{p_1}$};
		\draw[thick](2.15,-3.72)--(2.23,-4.17);
		\draw[thick](2.15,-3.72)--(2.48,-4.03);
		\draw[thick](-2.15,3.72)--(-2.23,4.17);
		\draw[thick](-2.15,3.72)--(-2.48,4.03);

		\coordinate (k3) at (-2.15,-3.72);
		\fill (k3) circle (1pt) node[left] {\tiny$\omega^2\frac{1}{p_1}$};
		\draw[thick](-2.15,-3.72)--(-2.23,-4.17);
		\draw[thick](-2.15,-3.72)--(-2.48,-4.03);

      \draw[thick](2.15,3.72)--(1.6,2.77);
         \draw[thick](-2.15,-3.72)--(-1.6,-2.77);
         \draw[thick](2.15,-3.72)--(1.6,-2.77);
         \draw[thick](-2.15,3.72)--(-1.6,2.77);
         \coordinate (l1) at (1.875,3.245);
			\fill(l1) circle (1pt) node[right] {\tiny$ \omega^2 z_d$};
 \coordinate (l2) at (-1.875,-3.245);
			\fill (l2) circle (1pt) node[left] {\tiny$ \omega^2z_a$};

\coordinate (l3) at (1.875,-3.245);
			\fill (l3) circle (1pt) node[right] {\tiny$ \omega z_d$};
 \coordinate (l4) at (-1.875,3.245);
			\fill (l4) circle (1pt) node[left] {\tiny$ \omega z_a$};

		\end{tikzpicture}
	\caption{\footnotesize The local jump contour $\Sigma^{L}$  for $0 < (\xi+\frac{3}{8})t^{2/3}<C$.}
	\label{sigma0}
\end{figure}
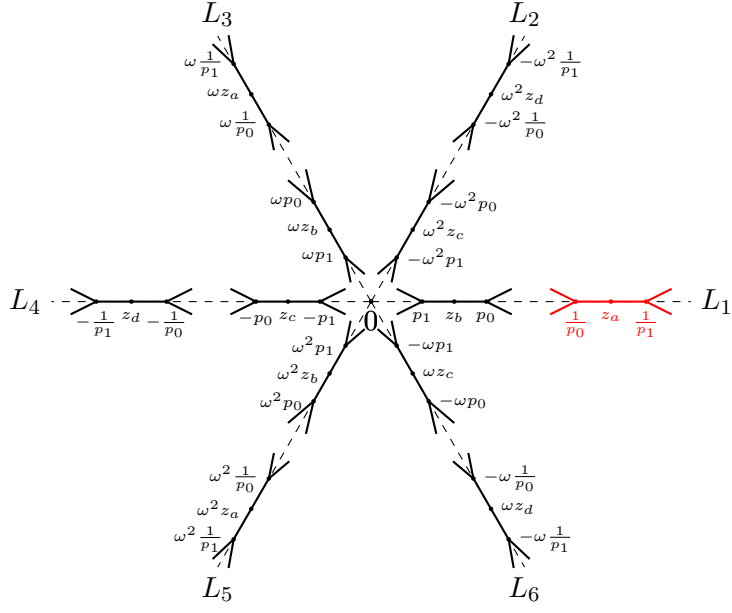

\begin{RHP} \label{rh7}
	Find a matrix-valued function  $ M^{L}(z)$ with following properties:
	\begin{itemize}
		\item  $M^{L}(z)$ is analytical  in $\mathbb{C}\setminus \Sigma^{L }$.
	
\item  $M^{L}(z)=\Gamma_1\overline{M^{L}(\bar{z})}\Gamma_1=\Gamma_2\overline{M^{L}(\omega^2\bar{z})}\Gamma_2=\Gamma_3\overline{M^{L}(\omega\bar{z})}\Gamma_3=\overline{M^{L}(\bar{z}^{-1})}$.
	
	\item $M^{L}(z)$ has continuous boundary values $M^{L}_\pm$ on $\Sigma^L $ and
	\begin{equation*}
	M^{L}_+(z)=M^{L}_-(z)V^{L}(z),\hspace{0.5cm}z \in \Sigma^{L},
	\end{equation*}
	where $ V^{L}(z)= V^{(3)}(z)\big|_{z\in\Sigma^{L}}$.

\item $M^{L}(z) = I+\mathcal{O}(z^{-1}),\hspace{0.5cm}z \rightarrow \infty.$
	\end{itemize}
\end{RHP}	

This  local  RH problem, which consists of $12$ local models on $\Sigma^{L,l}_j$ about
phase point $\omega^l\xi_{j}, \ j = a,b,c,d, \ l=0,1,2$,   has the jump condition and no poles.

\begin{RHP} \label{rh71}
	Find a matrix-valued function  $ M^{L,l}_j (z)$ with following properties:
\begin{itemize}	
 \item $M^{L,l}_j(z)$ is analytical  in $\mathbb{C}\setminus \Sigma^{L,l}_j$.

 \item $M^{L,l}_j (z)$ has continuous boundary values $M^{L,l}_{j,\pm} $  and
	\begin{equation*}
	M^{L,l}_{j,+} (z) =M^{L,l}_{j,-}(z) V^{L,l}_j(z),\hspace{0.5cm}z \in \Sigma^{L,l}_j,
	\end{equation*}
where $V^{L,l}_j(z)=V^L(z)|_{z\in\Sigma^{L,l}_j}$
 \item $M^{L,l}_j(z) =I+\mathcal{O}(z^{-1}),\hspace{0.5cm}z \rightarrow \infty$.
\end{itemize}
\end{RHP}

According to the theorem of Beals-Coifman, we know as $t \to \infty$, the solution  $M^{L}(z)$ of the RH problem \ref{rh7} is approximated by the sum of the separate local model  $M^{L,l}_j(z)$ of the RH problem \ref{rh71} in the neighborhoods of $\omega^l \xi_j$, $j=a,b,c,d, \ =0,1,2$ respectively.
As illustrative example,  we   only consider a local model  at the  point $ z_a$ on $\mathbb{R}$, whose  jump contours denotes  (red lines in Figure \ref{sigma0})
$$\Sigma^{L,0}_a= (\Sigma_{a}^{0} \cup \Sigma_{a}^{0*} \cup I^0_1)  \cap U_a^0,$$
which corresponds to the following   local RH problem.
\begin{RHP}\label{RHPlo1}
	Find a matrix-valued function  $ M^{L,0}_a(z)$ with following properties:
	\begin{itemize}
\item $M^{L,0}_a(z)$ is analytical  in $\mathbb{C}\setminus \Sigma^{L,0}_a$.
\item $M^{L,0}_a(z)$ has continuous boundary values $M^{L,0}_{a\pm}(z)$
	$$M^{L,0}_{a+}(z)=M^{L,0}_{a-}(z)V^{L,0}_a(z),\ z \in \Sigma^{L,0}_a,$$
	where
	\begin{align}
	V^{L,0}_a(z)=\left\{\begin{array}{ll}
	\left(\begin{array}{ccc}
	1 & 0& 0\\
	\tilde{r}(1/p_1)e^{-it\theta_{12}(z)}& 1 & 0\\
	0 & 0&1
	\end{array}\right),  & z\in \Sigma_{1}^0\cap U_a^0,\\[10pt]
	\left(\begin{array}{ccc}
	1 & -\overline{\tilde{r}(1/p_1)}e^{it\theta_{12}(z)}&0\\
	0 & 1 & 0\\
	0 & 0&1
	\end{array}\right),   & z\in \Sigma_{1}^{0*}\cap U_a^0,\\[10pt]
	\left(\begin{array}{ccc}
	1 & 0& 0\\
	\tilde{r}(1/p_0)e^{-it\theta_{12}(z)}& 1 & 0\\
	0 & 0&1
	\end{array}\right),   & z\in \Sigma_{2}^0\cap U_a^0,\\[10pt]
	\left(\begin{array}{ccc}
	1 &-\overline{\tilde{r}(1/p_0)}e^{it\theta_{12}(z)}&0\\
	0 & 1 & 0\\
	0 & 0&1
	\end{array}\right),   & z\in \Sigma_{2}^{0*}\cap U_a^0,\\[10pt]
\left(\begin{array}{ccc}
	1 & -\overline{\tilde{r}(z)}e^{it\theta_{12}(z)}&0\\
	0 & 1&0\\
	0 & 0 &1
\end{array}\right)\left(\begin{array}{ccc}
	1 & 0 & 0\\
\tilde{r}(z)e^{-it\theta_{12}(z)} & 1 & 0\\
	0 & 0 &1
\end{array}\right),   & z\in (1/p_0,1/p_1).
	\end{array}\right. \label{efe}
	\end{align}
	
	\item $M^{L,0}_a(z) = I+\mathcal{O}(z^{-1}),\hspace{0.5cm}z \rightarrow \infty$.
\end{itemize}
\end{RHP}

To solve the RH problem for $M^{L,0}_a(z)$, we observe that for  $z \in U_a^0$ and $t$ large enough,
\begin{align}\label{theapp1}
-t\theta_{12}(z) = -t\theta_{12}(z_a) + \frac{8}{3}\hat{z}^3 + 2\tilde{s} \hat{z} + \mathcal{O}(t^{-1/3}\hat{z}^2).
\end{align}
Here, we define the scaled spectral parameter $\hat{z}$ as
\begin{align}\label{zscale1}
\hat{z} = \left( c_a t\right)^{1/3} (z-z_a),
\end{align}
where the constant $c_a$ is given by
\begin{align}\label{hatc1}
c_a = \frac{21}{256} (14\sqrt{3}-9\sqrt{7}).
\end{align}
Additionally, we introduce the parameter $\tilde{s}$ as
\begin{align}\label{p1as}
\tilde{s} = s_1 (1+8\xi)t^{2/3}
\end{align}
with
\begin{align}
s_1  = \frac{3^{1/6} (-7+\sqrt{21})}{2^{7/3} 7^{1/3} (14\sqrt{3}- 9\sqrt{7})} <0,
\end{align}
which parametrizes the space-time region. Next we show that after scaling, RH problem  \ref{RHPlo1} can be well-approximated by the model RH problem \ref{modelp2} in \ref{appx2}, which is associated with the Painlev\'{e} \uppercase\expandafter{\romannumeral2} equation.

Through this change of variable \eqref{zscale1}, it is directly inferred that    $M^{L,0}_a(\hat{z}) = M^{L,0}_a(z(\hat{z}))$ is holomorphic for $\hat{z} \in \mathbb{C} \setminus \hat{\Sigma}_a^{L,0}$ where
\begin{align}
\hat{\Sigma}_a^{L,0} = \cup_{j=1,2} \left( \hat{\Sigma}_j^{0} \cup \hat{\Sigma}_j^{0*}\right) \cup (\hat{z}_1, \hat{z}_2),
\end{align}
with  $\hat{\Sigma}_j^{0}, j=1,2$ be the corresponding contours of $\Sigma_j^{0}, j=1,2$ after scaling, and
\begin{align}
   \hat{z}_1 = (c_a t)^{1/3} (1/p_1-z_a), \  \hat{z}_2 = (c_a t)^{1/3} (1/p_0-z_a).
\end{align}
In addition, the jump matrix  $V^{L,0}_a(\hat{z})$ satisfies
	\begin{align}\label{asvlo}
	V^{L,0}_a(\hat{z})=\left\{\begin{array}{ll}
	\left(\begin{array}{ccc}
	1 & 0& 0\\
	\tilde{r}(1/p_1)e^{-it\theta_{12}((c_at)^{-1/3}\hat{z}+z_a)}& 1 & 0\\
	0 & 0&1
	\end{array}\right),  & \hat{z}\in \hat{\Sigma}_{1}^0,\\[10pt]
	\left(\begin{array}{ccc}
	1 &-\overline{\tilde{r}(1/p_1)}e^{it\theta_{12}((c_at)^{-1/3}\hat{z}+z_a)}&0\\
	0 & 1 & 0\\
	0 & 0&1
	\end{array}\right),   & \hat{z}\in \hat{\Sigma}_{1}^{0*},\\[10pt]
	\left(\begin{array}{ccc}
	1 & 0& 0\\
	\tilde{r}(1/p_0)e^{-it\theta_{12}((c_at)^{-1/3}\hat{z}+z_a)}& 1 & 0\\
	0 & 0&1
	\end{array}\right),   & \hat{z}\in \hat{\Sigma}_{2}^0,\\[10pt]
	\left(\begin{array}{ccc}
	1 &-\overline{\tilde{r}(1/p_0)}e^{it\theta_{12}((c_at)^{-1/3}\hat{z}+z_a)}&0\\
	0 & 1 & 0\\
	0 & 0&1
	\end{array}\right),   & \hat{z}\in \hat{\Sigma}_{2}^{0*},\\[10pt]
 \mathcal{A}_1 \left(\begin{array}{ccc}
	1 &	-\overline{\tilde{r}((c_at)^{-1/3}\hat{z}+z_a)}&0\\
	0 & 1&0\\
	0 & 0 &1
\end{array}\right)\left(\begin{array}{ccc}
	1 & 0 & 0\\
\tilde{r}((c_at)^{-1/3}\hat{z}+z_a)& 1 & 0\\
	0 & 0 &1
\end{array}\right) \mathcal{A}_1^{-1},   & \hat{z}\in (\hat{z}_1,\hat{z}_2),
	\end{array}\right.
	\end{align}
where
\begin{align}
 \mathcal{A}_1 = \left(\begin{array}{ccc}
	e^{it\theta_{12}((c_at)^{-1/3}\hat{z}+z_a)/2} & 0 &0\\
	0 & e^{-it\theta_{12}((c_at)^{-1/3}\hat{z}+z_a)/2}&0\\
	0 & 0 &1
\end{array}\right).
\end{align}

Since \eqref{theapp1} holds, we will show the RH problem for $M^{L,0}_a(\hat{z})$ in the $\hat{z}$-plane can be explicitly approximated by the model RH problem for  $N^P(\hat{z})$ in  \ref{appx2}.
To do this,  we give the following basic estimates.
\begin{Proposition} \label{p1esttid1}
As $t \to +\infty$,
\begin{align}
 & |\tilde{r}((c_at)^{-1/3}\hat{z}+z_a)e^{-it\theta_{12}((c_at)^{-1/3}\hat{z}+z_a)} -R_a e^{i(8 \hat{z}^3/3 +2 \tilde{s} \hat{z})}  |\lesssim t^{-1/3+2\delta_1}, \ \hat{z} \in (\hat{z}_1,\hat{z}_2),  \\
 &|\tilde{r}(1/p_1) e^{-it\theta_{12}((c_at)^{-1/3}\hat{z}+z_a)} -R_a  e^{i(8 \hat{z}^3/3 +2 \tilde{s} \hat{z})}  |\lesssim t^{-1/3+2\delta_1},  \ \hat{z} \in \hat{\Sigma}_{1}^{0},\label{estr01}\\
 &|\tilde{r}(1/p_0) e^{-it\theta_{12}((c_at)^{-1/3}\hat{z}+z_a)} -R_a e^{i(8 \hat{z}^3/3 +2 \tilde{s} \hat{z})}  |\lesssim t^{-1/3+2\delta_1},  \ \hat{z} \in \hat{\Sigma}_{2}^{0}.\label{estr02}
 \end{align}
 where
 \begin{align}
  R_a := \tilde{r}(z_a) e^{- i t \theta_{12}(z_a)}
 \end{align}
\end{Proposition}
\begin{proof}
For $\hat{z} \in (\hat{z}_1,\hat{z}_2)$, we have
\begin{align*}
|e^{-it\theta_{12}((c_at)^{-1/3}\hat{z}+z_a)}| = |e^{i(8 \hat{z}^3/3 +2 \tilde{s} \hat{z})}|=1,
\end{align*}
and
\begin{align*}
 \left| \tilde{r}((c_at)^{-1/3}\hat{z}+z_a) -    \tilde{r}(z_a)  \right| = \left| \int_{z_a}^{(c_at)^{-1/3}\hat{z}+z_a} \tilde{r}'(\eta) d \eta \right| \le \|\tilde{r}' \|_{L^\infty(\mathbb{R})} \left| (c_at)^{-1/3}\hat{z}\right| \lesssim t^{-1/3}.
\end{align*}
For $\hat{z} \in \hat{\Sigma}_{1}^{0}$, since $ \re (i (8\hat{z}^3/3 + 2 \tilde{s} \hat{z}))<0$, it follows that $|e^{i(8 \hat{z}^3/3 +2 \tilde{s} \hat{z})}| $ is bounded and
\begin{align}\label{estr1}
 \left|  e^{-it\theta_{12}((c_at)^{-1/3}\hat{z}+z_a)} -  e^{- i t \theta_{12}(z_a)}e^{i(8 \hat{z}^3/3 +2 \tilde{s} \hat{z})} \right| = \left| e^{\mathcal{O}(t^{-1/3} \hat{z}^3)} -1 \right| \lesssim t^{-1/3+2\delta_1}.
\end{align}
On the other hand,
\begin{align}\label{estr2}
 \left| \tilde{r}(1/p_1)  -\tilde{r}(z_a)\right| =   \left|  \int_{z_a}^{1/p_1} \tilde{r}'(\eta) d\eta  \right| \le \|\tilde{r}' \|_{L^\infty(\mathbb{R})}  \left|1/p_1-z_a \right| \lesssim t^{-1/3}.
\end{align}
Then \eqref{estr01} can be derived from  \eqref{estr1} and \eqref{estr2}. For $\hat{z} \in \hat{\Sigma}_{2}^{0}$, we can use a similar method to prove \eqref{estr02}.

\end{proof}

Then we obtain the following proposition as a direct corollary of Proposition  \ref{p1esttid1}.
\begin{Proposition}
As $t \to +\infty$,
\begin{align}
 M_a^{L,0}(\hat{z}) = \hat{M}_a^{L,0} (\hat{z}) + \mathcal{O}(t^{-1/3+2\delta_1}),
\end{align}
where $ \hat{M}_a^{L,0} (\hat{z})$ satisfies the following RH problem.
\begin{RHP}\label{p1malo}
Find a matrix-valued function  $\hat{M}_a^{L,0} (\hat{z})$ with  properties:
\begin{itemize}
\item $\hat{M}^{L,0}_a(\hat{z})$ is analytical  in $\mathbb{C}\setminus \hat{\Sigma}^{L,0}_a$.
\item
	$\hat{M}^{L,0}_{a+}(\hat{z})=\hat{M}^{L,0}_{a-}(\hat{z})\hat{V}^{L,0}_a(\hat{z}),\ \hat{z} \in \hat{\Sigma}^{L,0}_a$,
	where
	\begin{align}\label{afsvlo}
	\hat{V}^{L,0}_a(\hat{z})=\left\{\begin{array}{ll}
	\left(\begin{array}{ccc}
	1 & 0& 0\\
R_a e^{i(8 \hat{z}^3/3 +2 \tilde{s} \hat{z})}& 1 & 0\\
	0 & 0&1
	\end{array}\right),  & \hat{z}\in \hat{\Sigma}_{j}^0, \ j=1,2,\\[10pt]
	\left(\begin{array}{ccc}
	1 &-\overline{R_a}e^{-i(8 \hat{z}^3/3 +2 \tilde{s} \hat{z})}&0\\
	0 & 1 & 0\\
	0 & 0&1
	\end{array}\right),   & \hat{z}\in \hat{\Sigma}_{j}^{0*},\ j=1,2, \\[10pt]
  \left(\begin{array}{ccc}
	1 &	-\overline{R_a}e^{-i(8 \hat{z}^3/3 +2 \tilde{s} \hat{z})}&0\\
	0 & 1 & 0\\
	0 & 0&1
	\end{array}\right)\left(\begin{array}{ccc}
	1 & 0& 0\\
R_ae^{i(8 \hat{z}^3/3 +2 \tilde{s} \hat{z})}& 1 & 0\\
	0 & 0&1
	\end{array}\right),   & \hat{z}\in (\hat{z}_1,\hat{z}_2).
	\end{array}\right.
	\end{align}
	\item $\hat{M}^{L,0}_a(\hat{z}) = I+\mathcal{O}(\hat{z}^{-1}),\hspace{0.5cm}\hat{z}\rightarrow \infty$.
\end{itemize}
\end{RHP}
\end{Proposition}

To match RH problem \ref{p1malo} with the model RH problem, we  need to convert the coefficients in front of the exponential terms in the jump matrix into the  pure imaginary number. Rewriting $ R_a$ as
\begin{align}
  R_a =  |R_a| e^{i\varphi_a}
\end{align}
with
\begin{align}
 &|R_a|  = |r(z_a)|, \ \varphi_a = \arg R_a =\arg \tilde{r}(z_a) -t \theta_{12}(z_a),
\end{align}
we make the following transformation
\begin{align}
\tilde{M}^{L,0}_a(\hat{z}) = \mathcal{B}_1 \hat{M}^{L,0}_a(\hat{z}) \mathcal{B}_1^{-1},
\end{align}
where
\begin{align}
\mathcal{B}_1 = \left(\begin{array}{ccc}
	e^{i(\varphi_a/2-\pi/4)}  & 0&0\\
	0 & e^{-i(\varphi_a/2-\pi/4)} & 0\\
	0 & 0&1
	\end{array}\right).
\end{align}
Then $\tilde{M}^{L,0}_a(\hat{z})$ satisfies the following RH problem.
\begin{RHP}\label{p1tmalo}
Find a matrix-valued function  $\tilde{M}_a^{L,0} (\hat{z})$ with  properties:
\begin{itemize}
\item $\tilde{M}^{L,0}_a(\hat{z})$ is analytical  in $\mathbb{C}\setminus \hat{\Sigma}^{L,0}_a$.
\item
	$\tilde{M}^{L,0}_{a+}(\hat{z})=\tilde{M}^{L,0}_{a-}(\hat{z})\tilde{V}^{L,0}_a(\hat{z}),\ \hat{z} \in \hat{\Sigma}^{L,0}_a$,
	where
	\begin{align}\label{afstvlo}
	\tilde{V}^{L,0}_a(\hat{z})=\left\{\begin{array}{ll}
	\left(\begin{array}{ccc}
	1 & 0& 0\\
	i|r(z_a)|e^{i(8 \hat{z}^3/3 +2 \tilde{s} \hat{z})}& 1 & 0\\
	0 & 0&1
	\end{array}\right),  & \hat{z}\in \hat{\Sigma}_{j}^0, \ j=1,2,\\[10pt]
	\left(\begin{array}{ccc}
	1 & i|r(z_a)|e^{-i(8 \hat{z}^3/3 +2 \tilde{s} \hat{z})}&0\\
	0 & 1 & 0\\
	0 & 0&1
	\end{array}\right),   & \hat{z}\in \hat{\Sigma}_{j}^{0*},\ j=1,2, \\[10pt]
  \left(\begin{array}{ccc}
	1 & i|r(z_a)|e^{-i(8 \hat{z}^3/3 +2 \tilde{s} \hat{z})}&0\\
	0 & 1 & 0\\
	0 & 0&1
	\end{array}\right)\left(\begin{array}{ccc}
	1 & 0& 0\\
	i|r(z_a)|e^{i(8 \hat{z}^3/3 +2 \tilde{s} \hat{z})}& 1 & 0\\
	0 & 0&1
	\end{array}\right),   & \hat{z}\in (\hat{z}_1,\hat{z}_2).
	\end{array}\right.
	\end{align}
	\item $\tilde{M}^{L,0}_a(\hat{z}) = I+\mathcal{O}(\hat{z}^{-1}),\hspace{0.5cm}\hat{z}\rightarrow \infty$.
\end{itemize}
\end{RHP}

Observing the jump matrix \eqref{afstvlo},  the RH problem \ref{p1tmalo} could be explicitly solved by using the model RH problem in \ref{appx2}.
\begin{Proposition}\label{p1mlonp}
 Let $c_1 = i|r(z_a)|$, then we obtain  $\tilde{M}^{L,0}_a(\hat{z})  = N^P(\hat{z})$.
\end{Proposition}
\begin{proof}
	Note that  $N^P(\hat{z})$ is invertible, we define
\begin{align}
   \Xi (\hat{z}) := \tilde{M}^{L,0}_a(\hat{z})  N^P(\hat{z})^{-1},
\end{align}
which satisfies the following RH problem.
\begin{RHP}\label{matchxi}
Find a matrix-valued function  $\Xi (\hat{z})$ with  properties:
\begin{itemize}
\item $\Xi(\hat{z})$ is analytical  in $\mathbb{C}\setminus \hat{\Sigma}^{L,0}_a$.
\item
	$\Xi_+(\hat{z})=\Xi_-(\hat{z})V^\Xi(\hat{z}),\ \hat{z} \in \hat{\Sigma}^{L,0}_a$,
	where
\begin{align}
V^\Xi(\hat{z})=  N^P_-(\hat{z}) \tilde{V}^{L,0}_a(\hat{z}) V^P(\hat{z})^{-1}  N^P_+(\hat{z})^{-1}.
\end{align}
	\item $\Xi(\hat{z})= I+\mathcal{O}(\hat{z}^{-1}),\hspace{0.5cm}\hat{z}\rightarrow \infty$.
\end{itemize}
\end{RHP}

Because of the boundedness of  $N^P(\hat{z})$ (see \eqref{mPbounded} below), it is sufficient to estimate the error between the jump matrices $\tilde{V}^{L,0}_a(\hat{z})$ and  $V^P(\hat{z})$.
\begin{align*}
\tilde{V}^{L,0}_a(\hat{z}) - V^P(\hat{z}) =\\
&\begin{cases}
 \begin{pmatrix} 1 & 0 &0 \\ (i|r(z_a)|-p)e^{i(8\hat{z}^2/3+2\tilde{s}\hat{z})} & 1 &0 \\ 0 & 0 & 1 \end{pmatrix}, \quad \hat{z} \in \hat{\Sigma}^{0}_{j},\ j=1,2,\\
\begin{pmatrix} 1 & (i|r(z_a)|+p^*)e^{-i(8\hat{z}^2/3+2\tilde{s}\hat{z})} &0 \\ 0 & 1 &0 \\ 0 & 0 & 1 \end{pmatrix},\quad \hat{z} \in \hat{\Sigma}^{0*}_{j}, \ j=1,2 ,\\
\begin{pmatrix} -|r(z_a)|^2 + |p|^2 & (i |r(z_a)| +  p^*)e^{-i(8\hat{z}^2/3+2\tilde{s}\hat{z})} &0 \\ (ir(z_a) -p)e^{i(8\hat{z}^2/3+2\tilde{s}\hat{z})}  & 0 &0 \\ 0 & 0 & 1 \end{pmatrix},\quad \hat{k} \in (\hat{z}_1,\hat{z}_2).
  \end{cases}
\end{align*}
Then using Proposition \ref{p1esttid1},   we have
\begin{align}
  \| V^\Xi(\hat{z}) -I \|_{L^1 \cap L^2 \cap L^\infty(\hat{\Sigma}^{L,0}_a) } \lesssim t^{-1/3+2\delta_1}.
\end{align}
Thus, the existence and uniqueness of  $ \Xi (\hat{z})$ are valid by a small norm RH problem arguments, which also yields
\begin{align}
\Xi (\hat{z}) = I + \mathcal{O}(t^{-1/3+2\delta_1}), \quad \text{as} \ t \to +\infty.
\end{align}

\end{proof}

Finally,  the following proposition follows directly from  propositions \ref{p1esttid1}-\ref{p1mlonp}.
\begin{Proposition}
As $t \to +\infty$,
\begin{align}
 M_a^{L,0}(\hat{z}) = \mathcal{B}_1^{-1} N^P(\hat{z}) \mathcal{B}_1 + \mathcal{O}(t^{-1/3+2\delta_1}).
\end{align}
Moreover, as $\hat{z} \to \infty$,
\begin{align}
M_a^{L,0}(\hat{z}) = I + \frac{M_{a1}^{L,0}}{\hat{z}} + \mathcal{O}(\hat{z}^{-2}),
\end{align}
where  $M_{a1}^{L,0}$  is the coefficient of the term $1/\hat{z}$ in the large-$\hat{z}$ expansion of  $M_a^{L,0}(\hat{z})$.
As $t\to +\infty$,
\begin{align}\label{p1ma11}
	 M_{a1}^{L,0}  =  \frac{i}{2}  \begin{pmatrix} -\int_{\tilde{s}}^\infty u(\eta)^2d\eta & u(\tilde{s})e^{-i\varphi_a} &0 \\ -u(\tilde{s})e^{i\varphi_a} & \int_{\tilde{s}}^\infty u(\eta)^2d\eta &0  \\ 0& 0 & 0 \end{pmatrix} + \mathcal{O}(t^{-1/3+2\delta_1}).
\end{align}

\end{Proposition}

The RH problem for $M_b^{L,0}$,  $M_c^{L,0} $ and $M_d^{L,0}$ can be solved in a similar manner.
For $z \in  U_b^{0}$ and $t$ large enough, we have
\begin{align}
	-t \theta_{12}(z) = 	-t \theta_{12}(z_b)  + \frac{8}{3} \check{z}^3 + 2 \tilde{s} \check{z} + \mathcal{O}(t^{-1/3}\check{z}^2),
\end{align}
where $\tilde{s}$ is defined by  \eqref{p1as} and
\begin{align}
	\check{z} = (c_b t)^{\frac{1}{3}}(z-z_b),
\end{align}
is the scaled spectral parameter in this case with
\begin{align}\label{checkc1}
c_b = \frac{21}{256}(14\sqrt{3}+9\sqrt{7}).
\end{align}
Through the analysis,  we obtain the result of the large-$\check{z}$ expansion of  $M_b^{L,0}(\check{z})$,
\begin{align}
M_b^{L,0}(\check{z}) = I + \frac{M_{b1}^{L,0}}{\check{z}} + \mathcal{O}(\check{z}^{-2}),
\end{align}
where as $t \to +\infty$,
\begin{align}\label{p1mb11}
M_{b1}^{L,0} =  \frac{i}{2}  \begin{pmatrix} -\int_{\tilde{s}}^\infty u(\eta)^2d\eta & u(\tilde{s})e^{-i\varphi_b} &0 \\ -u(\tilde{s})e^{i\varphi_b} & \int_{\tilde{s}}^\infty u(\eta)^2d\eta &0  \\ 0& 0 & 0 \end{pmatrix} + \mathcal{O}(t^{-1/3+2\delta_1}),
\end{align}
with
\begin{align}
\varphi_b = \arg \tilde{r}(z_b)  -t\theta_{12}(z_b).
\end{align}

For $z \in U_c^0$, and for sufficiently large $t$, the following relationship holds:
\begin{align}
 -t\theta_{12}(z) = -t \theta_{12}(z_c)  + \frac{8}{3} \tilde{z}^3 + 2 \tilde{s} \tilde{z} + \mathcal{O}(t^{-1/3}\tilde{z}^2).
\end{align}
Here, $\tilde{s}$ is defined in accordance with  \eqref{p1as}, and the scaled spectral parameter $\tilde{z}$ is given by:
\begin{align}
	\tilde{z} = (c_c t)^{\frac{1}{3}}(z-z_c),
\end{align}
where
\begin{align}\label{p1cccb}
c_c=c_b,
\end{align}
which is as specified  in \eqref{checkc1}. Moreover, as $\tilde{z} \to \infty$,
\begin{align}
  M_c^{L,0}(\check{z}) = I + \frac{M_{c1}^{L,0}}{\check{z}} + \mathcal{O}(\tilde{z}^{-2}),
\end{align}
where  as $t \to +\infty$,
\begin{align}\label{p1mc11}
M_{c1}^{L,0} =  \frac{i}{2}  \begin{pmatrix} -\int_{\tilde{s}}^\infty u(\eta)^2d\eta & -u(\tilde{s})e^{i\varphi_b} &0 \\ u(\tilde{s})e^{-i\varphi_b} & \int_{\tilde{s}}^\infty u(\eta)^2d\eta &0  \\ 0& 0 & 0 \end{pmatrix} + \mathcal{O}(t^{-1/3+2\delta_1}).
\end{align}
Here we use the symmetry $|r(z)|=|r(-z)|$ whose detailed proof could be given by the similar method in the proof of Proposition 4.2 in \cite{Monvel3}.

For $z \in U_d^0$, and for sufficiently large $t$, the following relationship holds:
\begin{align}
 -t\theta_{12}(z) = -t \theta_{12}(z_d)  + \frac{8}{3} \breve{z}^3 + 2 \tilde{s} \breve{z} + \mathcal{O}(t^{-1/3}\breve{z}^2).
\end{align}
Here, $\tilde{s}$ is defined in accordance with  \eqref{p1as}, and the scaled spectral parameter $\breve{z}$ is given by:
\begin{align}
	\breve{z} = (c_d t)^{\frac{1}{3}}(z-z_d),
\end{align}
where
\begin{align}\label{p1cdca}
c_d=c_a,
\end{align}
which is as specified  in \eqref{hatc1}. Moreover, as $\breve{z} \to \infty$,
\begin{align}
  M_d^{L,0}(\breve{z}) = I + \frac{M_{d1}^{L,0}}{\breve{z}} + \mathcal{O}(\breve{z}^{-2}),
\end{align}
where  as $t \to +\infty$,
\begin{align}\label{p1md11}
M_{d1}^{L,0} =  \frac{i}{2}  \begin{pmatrix} -\int_{\tilde{s}}^\infty u(\eta)^2d\eta & -u(\tilde{s})e^{i\varphi_a} &0 \\ u(\tilde{s})e^{-i\varphi_a} & \int_{\tilde{s}}^\infty u(\eta)^2d\eta &0  \\ 0& 0 & 0 \end{pmatrix} + \mathcal{O}(t^{-1/3+2\delta_1}).
\end{align}
According to the symmetries of RH problem \ref{RHP1}, we have the following proposition.
\begin{Proposition}
 As $z \to \infty$, the coefficients  $M_{j1}^{L,l}, \ j=a,b,c,d,\ l=0,1,2$ of the term of $z^{-1}$  have the following relationships:
  \begin{align}
    M_{j1}^{L,1} = \omega \Gamma_3 \overline{M_{j1}^{L,0}} \Gamma_3,  \quad  M_{j1}^{L,2} = \omega^2 \Gamma_2 \overline{M_{j1}^{L,0}} \Gamma_2.
  \end{align}
\end{Proposition}

Combining the results above,  finally we obtain the following proposition.
\begin{Proposition}\label{asymlo}
	As $t\to+\infty$,
	\begin{align} \label{p1deml}
	M^{L}(z)=I+t^{-1/3}\sum_{j=a,b,c,d } c_j^{-1/3} \left( \frac{M_{j1}^{L,0} }{z-\xi_j}+\frac{\omega\Gamma_3\overline{M_{j1}^{L,0} }\Gamma_3}{z-\omega\xi_j}+\frac{\omega^2\Gamma_2\overline{M_{j1}^{L,0} }\Gamma_2}{z-\omega^2\xi_j}\right)  +\mathcal{O}(t^{-1}),
	\end{align}
	where $c_j,\ j=a,b,c,d$ are defined by \eqref{hatc1}, \eqref{checkc1}, \eqref{p1cccb}, and \eqref{p1cdca}, and $M_{j1}^{L,0}, \ j=a,b,c,d$ are given by  \eqref{p1ma11}, \eqref{p1mb11}, \eqref{p1mc11}, and \eqref{p1md11}.

\end{Proposition}

\subsubsection{RH problem near singularities}\label{rhsing}

 RH problems near singularities $\varkappa_j$, $j=1,\cdots,6$ have the following properties.

\begin{RHP}
	Find a matrix-valued function  $  M^B_j(z)$ with following properties:
	\begin{itemize}
	\item $M^B_j(z)$ is  meromorphic  in $\mathbb{C}\setminus \left(\mathbb{B}_j\cap L_j \right) $.
	
	\item $M^B_j$ has continuous boundary values $M^B_{j\pm}$ on $\mathbb{B}_j\cap L_j$ and
	\begin{equation}
	M^B_{j+}(z)=M^B_{j-}(z)V^{B}(z),\hspace{0.5cm}z \in \mathbb{B}_j\cap L_j,\label{jumpB}
	\end{equation}
where $V^{B}(z) = V^{(3)}(z)|_{z\in\mathbb{B}_j\cap L_j}$.
		
	\item $M^{B}_j(z) =I+\mathcal{O}(z^{-1}),\hspace{0.5cm}z \rightarrow \infty$.

\end{itemize}
\end{RHP}
By the symmetry of $M^R(z)$, we obtain $M^B_3(z)=\Gamma_3\overline{M^B_1(\omega \bar{z})}\Gamma_3$ for $z\in\mathbb{B}_3$, $M^B_5(z)=\Gamma_2\overline{M^B_1(\omega^2 \bar{z})}\Gamma_2$ for $z\in\mathbb{B}_5$, $M^B_6(z)=\Gamma_3\overline{M^B_4(\omega \bar{z})}\Gamma_3$ for $z\in\mathbb{B}_6$, and $M^B_2(z)=\Gamma_2\overline{M^B_4(\omega^2 \bar{z})}\Gamma_2$ for $z\in\mathbb{B}_2$. Thus, we only need to  give the detail of RH problem for $M^{B}_1(z)$. Then $M^{B}_4(z)$ can be obtained analogously and the others  can be obtained by symmetry. $M^{B}_1(z)$ only has the jump condition on $(1-2\varepsilon,1+2\varepsilon)$ with
\begin{align*}
	V^B(z)=\left(\begin{array}{ccc}
		1 & \frac{rT_{21}\mathcal{X} e^{it\theta_{12}}}{1-|r|^2}&0\\
		0 & 1&0\\
		0 & 0 &1
	\end{array}\right)\left(\begin{array}{ccc}
		1 & 0 & 0\\
		\frac{\bar{r}T_{12}\mathcal{X} e^{-it\theta_{12}}}{1-|r|^2} & 1 & 0\\
		0 & 0 &1
	\end{array}\right).
\end{align*}

To proceed, we first give the following lemma about the properties of the imaginary part of the phase function $\theta_{12}$.
\begin{lemma}\label{p1the12}
Let $\Omega_j^l,\ j=1,\cdots,8,\ l=0,1,2$, denote the sectors defined in \eqref{p1spa1}- \eqref{p1spa2}.
The following estimates for the imaginary part of the phase function $\theta_{12}(z)$ in the transition zone $0<(\xi+1/8)t^{2/3}<C$, as defined in \eqref{deftheta12}, are valid. Similar estimates can also be provided for the imaginary parts of the phase functions $\theta_{23}(z)$ and $\theta_{31}(z)$.
\begin{itemize}
\item  If $| \re z(1-|z|^{-2}) | \le 2$, then
\begin{subequations}
\begin{align}
& \im \theta_{12}(z) \le - \frac{\sqrt{3}(7-3 \tilde{k}_1^2)}{(\tilde{k}_1+1)^2}\im z \left( \re z -\xi_j \right)^2, \quad z\in \Omega_j^0, \ j=1,\cdots,8, \label{p1theia}\\
& \im \theta_{12}(z) \ge  \frac{\sqrt{3}(7-3 \tilde{k}_1^2)}{(\tilde{k}_1+1)^2} \im z \left( \re z -\xi_j \right)^2, \quad z\in \Omega_j^{0*}, \ j=1,\cdots,8,\label{p1theib}
\end{align}
\end{subequations}
\item If $| \re z(1-|z|^{-2}) | > 2$, then
\begin{subequations}
\begin{align}
& \im \theta_{12}(z) \le - \frac{\sqrt{3}}{8}  \im z, \quad z\in \Omega_j^0, \ j=1,\cdots,8,\label{p1theoa}\\
& \im \theta_{12}(z) \ge   \frac{\sqrt{3}}{8}  \im z, \quad z\in \Omega_j^{0*}, \ j=1,\cdots,8,\label{p1theob}
\end{align}
\end{subequations}

\end{itemize}

\end{lemma}

\begin{proof}
We give the proof on the sector $\Omega_1$ and the proof on other sectors can be obtained similarly.

For $z\in \Omega_1$, denote
\begin{align*}
 z - 1/z := u+vi, \ u,v \in \mathbb{R}, \ \text{and} \ \tilde{k}_1 : = \xi_1-1/\xi_1.
\end{align*}
Then we obtain $u = \re z (1-1/|z|^2)$, $v= \im z (1+1/|z|^2)$. From \eqref{deftheta12},
\begin{align}\label{p1the}
\im \theta_{12} = \sqrt{3} v \left(\xi + F(u,v) \right), \ F(u,v): = \frac{u^2 +v^2-1}{(u^2-v^2+1)^2+4u^2v^2}.
\end{align}
It is evident that
\begin{align}\label{f1}
F(u,v) \le
\begin{cases}
&F(u,0), \quad  u^2 \le 4,\\
&0,\quad  u^2 >4.
\end{cases}
\end{align}
Moreover, from \eqref{xik1},we have
\begin{align}\label{f2}
\xi + F(u,0) = \frac{1-\tilde{k}_1^2}{(1+\tilde{k}_1^2)^2} +\frac{u^2-1}{(u^2+1)^2} = (u^2 -\tilde{k}_1^2) \frac{3+u^2-\tilde{k}_1^2(u^2-1)}{(1+\tilde{k}_1^2)^2(u^2+1)^2}
\end{align}
Inserting \eqref{f1} and \eqref{f2} into \eqref{p1the}, we obtain \eqref{p1theia} and \eqref{p1theoa} from the ranges of $u$ and $\tilde{k}_1$.

\end{proof}

To obtain the large time asymptotic behavior of $M^{B}_1(e^{\frac{\pi i}{6}})$, we transform it to a pure $\bar{\partial}$-problem by multiplying a new function $R^{B}$ defined as follow:
\begin{align*}
	R^{B}(z)=\left\{\begin{array}{lll}
		\left(\begin{array}{ccc}
			1 & R^{B}_{+}e^{it\theta_{12}} & 0\\
			0 & 1 & 0\\
			0 & 0 & 1
		\end{array}\right), & z\in \mathbb{C}^-,\\[12pt]
	\left(\begin{array}{ccc}
	1 & 0 & 0\\
	 R^{B}_{-}e^{-it\theta_{12}} & 1 & 0\\
	0 & 0 & 1
	\end{array}\right), & z\in \mathbb{C}^+,
	\end{array}\right.
\end{align*}
in which
\begin{align}
	&R^{B}_{+}(z)=\mathcal{X}(\text{Re}z)\mathcal{X}(\text{Im}z+1)f(\text{Re}z)g(z),\ \ \
	R^{B}_{-}(z)=\overline{R^{B}_{+}(\bar{z})}, \nonumber
\end{align}
with
\begin{align}
	f(z)=\frac{r(z)}{1-|r(z)|^2},\hspace{0.5cm} g(z)=T_{12}(z).
\end{align}
 Then  direct  calculations yield
\begin{align}
	|\bar{\partial}R^{B}_{+}|\lesssim |\mathcal{X}'(\text{Re}z)\mathcal{X}(\text{Im}z+1)|+|\mathcal{X}(\text{Re}z)\mathcal{X}'(\text{Im}z+1)|.
\end{align}
It is obvious that  the support of  $R^{B}_{+}$ and  $\bar{\partial}R^{B}_{+}$ are contained in $\mathbb{B}_1$.
Denote
\begin{align}
	\tilde{M}^{B}_1(z)=M^{B}_1(z)R^{B}(z).
\end{align}
Then
\begin{align}
	&\bar{\partial}\tilde{M}^{B}_1(z)=M^{B}_1(z)\bar{\partial}R^{B}(z),\ \ \  \tilde{M}^{B}_1(z)\sim I,\ z\to\infty.
\end{align}
Specially, $\tilde{M}^{B}_1$ has no jump. Therefore, its solution can be given by the following integral equation
\begin{equation}
\tilde{M}^{B}_1(z)=I+\frac{1}{\pi}\iint_\mathbb{C}\dfrac{\tilde{M}^{B}_1(\eta)\bar{\partial}R^{B} (\eta)}{\eta-z}dm(\eta).\label{tMB}
\end{equation}
Denote $C_B: L^\infty(\mathbb{C})\to L^\infty(\mathbb{C})$ be the integral  operator  as
\begin{equation}
	C_Bf(z)=\frac{1}{\pi}\iint_\mathbb{C}\dfrac{f(\eta)\bar{\partial}R^{B} (\eta)}{\eta-z}dm(\eta).\label{CB}
\end{equation}
\begin{Proposition}
	$C_B$ is a bounded integral  operator from $L^\infty(\mathbb{C})$ to $L^\infty(\mathbb{C})$. Moreover, $C_B$ has the following estimate:
	\begin{align}
		\parallel C_B \parallel_{L^\infty}\lesssim t^{-1/p}, \ p>1,
	\end{align}
which implies that $\left( I-C_B\right)^{-1} $ exists as $t\to\infty$.
\end{Proposition}
\begin{proof}
	A direct calculation shows that
	\begin{align}
		\parallel C_B \parallel_{L^\infty}\lesssim \iint_{\mathbb{C}^+}\dfrac{|\bar{\partial}R^{B} (\eta)|}{|\eta-z|}dm(\eta)+\iint_{\mathbb{C}^-}\dfrac{|\bar{\partial}R^{B} (\eta)|}{|\eta-z|}dm(\eta).
	\end{align}
Take the first term as an example. Let $\eta=u+vi$, $z=x+yi$.  Using H\"{o}lder's inequality, Lemma \ref{p1the12}, and the following basic inequalities
\begin{align}
	\parallel |\eta-z|^{-1}\parallel_{L^q(0,+\infty)}
	\lesssim |v-y|^{1/q-1},\label{s-z}
\end{align}
where $1< q<+\infty$ and $\frac{1}{p}+\frac{1}{q}=1$, it follows that
\begin{align}
	\iint_{\mathbb{C}^+}\dfrac{|\bar{\partial}R^{B} (\eta)|}{|\eta-z|}dm(\eta)&\le \int_{0}^{2\varepsilon}\int_{1-2\varepsilon}^{1+2\varepsilon}|\eta-z|^{-1}|\bar{\partial}R^{B} (\eta)|e^{- \frac{\sqrt{3}(7-3 \tilde{k}_1^2)}{(\tilde{k}_1+1)^2}t(u-z_b)^2v} du dv\nonumber\\
	&\lesssim\int_{0}^{2\varepsilon}v^{-1/q}e^{- \frac{\sqrt{3}(7-3 \tilde{k}_1^2)}{(\tilde{k}_1+1)^2}(1+2\varepsilon-z_b)^2tv}dv\lesssim t^{-1/p}.
\end{align}
\end{proof}
Hence, from $\tilde{M}^{B}_1=\left( I-C_B\right)^{-1}\cdot I$, we get the existence and uniqueness of $\tilde{M}^{B}_1(z)$. Take $z=e^{\frac{\pi i}{6}}$ in (\ref{tMB}), then
 \begin{equation}
 	\tilde{M}^{B}_1(z)-I=\frac{1}{\pi}\iint_\mathbb{C}\dfrac{\tilde{M}^{B}_1(\eta)\bar{\partial}R^{B} (\eta)}{\eta-e^{\frac{\pi i}{6}}}dm(\eta).
 \end{equation}
\begin{Proposition}
	There exists  a constant $T_1$, such that for all $t>T_1$,   $\tilde{M}^{B}(z)$   admits the following estimate
	\begin{align}
		\parallel \tilde{M}^{B}(e^{\frac{i\pi}{6}})-I\parallel\lesssim t^{-1}.\label{mBi}
	\end{align}
\end{Proposition}
\begin{proof}
	Since $e^{\frac{\pi i}{6}}\notin\mathbb{B}_1$, $|s-e^{\frac{\pi i}{6}}|\lesssim1$ in $\mathbb{B}_1$.  Then, we have
	\begin{align}
		\frac{1}{\pi}\iint_\mathbb{C}\dfrac{\tilde{M}^{B}_1(\eta)\bar{\partial}R^{B} (\eta)}{\eta-e^{\frac{\pi i}{6}}}dm(\eta)\lesssim\iint_{\mathbb{C}^+}|\bar{\partial}R^{B} (\eta)|dm(\eta)+\iint_{\mathbb{C}^-}|\bar{\partial}R^{B} (\eta)|dm(\eta).\nonumber
	\end{align}
Here we take the estimate of the  first term as an example
	\begin{align}
	&\iint_{\mathbb{C}^+}|\bar{\partial}R^{B} (\eta)|dm(\eta)\lesssim\int_{0}^{2\varepsilon}\int_{1-2\varepsilon}^{1+2\varepsilon}|\bar{\partial}R^{B} (\eta)|e^{- \frac{\sqrt{3}(7-3 \tilde{k}_1^2)}{(\tilde{k}_1+1)^2}(1+2\varepsilon-z_b)^2tv} dudv\nonumber\\
	&\lesssim\int_{0}^{2\varepsilon}e^{- \frac{\sqrt{3}(7-3 \tilde{k}_1^2)}{(\tilde{k}_1+1)^2}(1+2\varepsilon-z_b)^2tv}dv\lesssim t^{-1}.
\end{align}
The  estimate of the second term can be given in a similar way.
\end{proof}
Finally, we obtain
\begin{align}
	M^{B}_1(e^{\frac{i\pi}{6}})=
	\tilde{M}^{B}_1(e^{\frac{i\pi}{6}})R^{B}(e^{\frac{i\pi}{6}})=I+\mathcal{O}(t^{-1}).\label{asyMB}
\end{align}
In addition, for $z \in \partial \mathbb{B}_1$, $R^{B}(z)=I$ and then $M^{B}_1(z)=\tilde{M}^{B}_1(z)$.
Furthermore, for $z\in\partial\mathbb{B}_1$, when $\eta \in\partial\mathbb{B}_1$,
we still have $ {|\bar{\partial}R^{B} (\eta)|}{|\eta-z|}^{-1}=0$, which  implies that $  {|\bar{\partial}R^{B} (\eta)|}{|\eta-z|}^{-1}$ is bounded for $z\in\partial\mathbb{B}_1$ and  $\eta \in\mathbb{B}_1$.
Then through direct calculations, we obtain  for $z\in\partial\mathbb{B}_1$,
\begin{align}
	M^{B}_1(z)=I+\mathcal{O}(t^{-1}).
\end{align}
Moreover,  $M^{B}_1(1)= R^B(1)^{-1} +\mathcal{O}(t^{-1/p})$.
Here, $p$ is a  arbitrary  constant with $p>1$. For ease of use, rewrite $p$ as  $1/p=1-\rho$ where $\rho$ is a small enough positive constant with $\rho<\frac{1}{4}$. Then we obtain
 $$M^{B}_1(1)=R^B(1)^{-1}
+\mathcal{O}(t^{-1+\rho}).$$ Similarly,  for $j=2,\cdots,6$,
\begin{align}
	 \lim_{z\to \kappa_j} \frac{M^{B}_j(z)-M^{B}_j(\kappa_j )}{z-\kappa_j} =
	\mathcal{O}(t^{-1+\rho}).\label{MBkap}
\end{align}

\subsubsection{Small norm RH problem for the residual error  }\label{sec7}

\quad  In this section,  we consider the error matrix-function $E(z)$.
From the  definition (\ref{transm4}), we can obtain the  RH problem  for    $E(z)$:

\begin{RHP}\label{p1e1}
	 Find a matrix-valued function $E(z)$  with  properties:
\begin{itemize}	

\item $E(z)$ is analytical  in $\mathbb{C}\setminus  \Sigma^{E} $, where
\begin{align*}
\Sigma^{E}=\partial  U\cup ( \cup_{j=1}^6 \partial\mathbb{B}_j) \
(\Sigma^{(3)}\setminus (U \cup  (\cup_{j=1}^6 \mathbb{B}_j) ).
\end{align*}

\item $E(z)$  has continuous boundary values $E_\pm(z)$  satisfying
$$E_+(z)=E_-(z)V^{E}(z), \ \ z\in \Sigma^{E},$$
where the jump matrix $V^{E}(z)$ is given by
\begin{equation}
V^{E}(z)=\left\{\begin{array}{llll}
M^{ O}(z)V^{(3)}(z)M^{ O }(z)^{-1}, & z\in \Sigma^{(3)}\setminus (U \cup  (\cup_{j=1}^6 \mathbb{B}_j) ),\\[4pt]
M^{ O }(z)M^{L,l}_j(z)^{-1}M^{ O }(z)^{-1},  & z\in \partial U_j^l, \ j=a,b,c,d, \ l=0,1,2,\\[4pt]
M^{ O }(z)M^{B}_j(z)^{-1}M^{ O }(z)^{-1},  & z\in \partial \mathbb{B}_j,\ j=,1,\cdots,6.
\end{array}\right. \label{deVE}
\end{equation}
See  Figure \ref{figE}.

\item $E(z)$  has the following asymptotic behavior:
\begin{align*}
	&E(z) = I+\mathcal{O}(z^{-1}),\hspace{0.5cm} z \rightarrow \infty.
\end{align*}

\item As $z \to \varkappa_j = e^{\frac{i\pi(j-1)}{3}}, j = 1,\cdots,6$,
the limit  of $E(z)$
 has pole singularities with leading terms of a specific
matrix structure
\begin{align}
	&\lim_{z\to \varkappa_j}E(z)=\lim_{z\to \varkappa_j}M^R(z)M^B_j(z)^{-1}M^{O }(z)^{-1}=\mathcal{O} (  (z-\varkappa_j)^{-2}).
\end{align}
\end{itemize}
\end{RHP}

\begin{figure}[htp]
	\centering
	\begin{tikzpicture} [scale=0.7]
		\draw[dotted](0,0)--(5.2,0)node[right]{ $L_1$};
		\draw[dotted](0,0)--(-5.2,0)node[left]{ $L_4$};
		\draw[dotted,rotate=60](0,0)--(5.2,0)node[above]{ $L_2$};
		\draw[dotted,rotate=60](0,0)--(-5.2,0)node[below]{ $L_5$};
		\draw[dotted,rotate=120](0,0)--(5.2,0)node[above]{ $L_3$};
		\draw[dotted,rotate=120](0,0)--(-5.2,0)node[below]{ $L_6$};
		
		\draw[thick](2.3,0.1)--(2.887,0.4)--(3.5,0.1);
		\draw[thick](4.5,0.1)--(5.087,0.4);
        \filldraw (1.8,0)circle(0.04cm);
        \filldraw (4,0)circle(0.04cm);
        \node [below] at (1.9,0.15) {\tiny{$z_b$}};
        \node [below] at (4.1,0.15) {\tiny{$z_a$}};
        \filldraw (-1.8,0)circle(0.04cm);
        \filldraw (-4,0)circle(0.04cm);
        \node [below] at (-1.8,0.15) {\tiny{$z_c$}};
        \node [below] at (-4,0.15) {\tiny{$z_d$}};

		\draw [red,thick] (1.8,0)circle(0.5cm);
        \draw [red,thick] (4,0)circle(0.5cm);
        \draw [red,thick,rotate=120] (4,0)circle(0.5cm);
        \draw [red,thick,rotate=120] (1.8,0)circle(0.5cm);
        \draw [red,thick,rotate=180] (4,0)circle(0.5cm);
        \draw [red,thick,rotate=180] (1.8,0)circle(0.5cm);
        \draw [red,thick,rotate=240] (4,0)circle(0.5cm);
        \draw [red,thick,rotate=240] (1.8,0)circle(0.5cm);
        \draw [red,thick,rotate=300] (4,0)circle(0.5cm);
        \draw [red,thick,rotate=300] (1.8,0)circle(0.5cm);
        \draw [red,thick,rotate=60] (1.8,0)circle(0.5cm);
		\draw [red,thick,rotate=60] (4,0)circle(0.5cm);

		\draw[thick,rotate=60](2.3,0.1)--(2.887,0.4)--(3.5,0.1);
		\draw[thick,rotate=60](4.5,0.1)--(5.087,0.4);
	
      \filldraw [rotate=60](1.8,0)circle(0.04cm);
      \filldraw [rotate=60](4,0)circle(0.04cm);

      \node [right] at (2.3,3.4) {\tiny{$\omega^2z_d$}};
       \node [right] at (1.3,1.5) {\tiny{$\omega^2z_c$}};
	
		\draw[thick,rotate=120](2.3,0.1)--(2.887,0.4)--(3.5,0.1);
		\draw[thick,rotate=120](4.5,0.1)--(5.087,0.4);
	
		\filldraw [rotate=120](1.8,0)circle(0.04cm);
		\filldraw [rotate=120](4,0)circle(0.04cm);
		
        \node [left] at (-1.3,1.5) {\tiny{$\omega z_b$}};
        \node [left] at (-2.3,3.4) {\tiny{$\omega z_a$}};
	
		\draw[thick,rotate=180](2.3,0.1)--(2.887,0.4)--(3.5,0.1);
		
		\draw[thick,rotate=180](4.5,0.1)--(5.087,0.4);
		
		\draw[thick,rotate=240](2.3,0.1)--(2.887,0.4)--(3.5,0.1);
		
		\draw[thick,rotate=240](4.5,0.1)--(5.087,0.4);

		\filldraw [rotate=240](1.8,0)circle(0.04cm);
		\filldraw [rotate=240](4,0)circle(0.04cm);
		
		\node [left] at (-1.3,-1.5) {\tiny{$\omega^2z_b$}};
		\node [left] at (-2.3,-3.4) {\tiny{$\omega^2z_a$}};

		\draw[thick,rotate=300](2.3,0.1)--(2.887,0.4)--(3.5,0.1);
		
		\draw[thick,rotate=300](4.5,0.1)--(5.087,0.4);
		
		\filldraw [rotate=300](1.8,0)circle(0.04cm);
		\filldraw [rotate=300](4,0)circle(0.04cm);
		\node [right] at (1.3,-1.5) {\tiny{$\omega z_c$}};
		\node [right] at (2.3,-3.4) {\tiny{$\omega z_d$}};

		\draw[thick](-1.3,0.1)--(-0.7,0.4);
		\draw[thick](1.3,0.1)--(0.7,0.4);
		\draw [thick](0,0.8)--(0,-0.8);
		\draw[thick](2.3,-0.1)--(2.887,-0.4)--(3.5,-0.1);
		\draw[thick](4.5,-0.1)--(5.087,-0.4);
	
		\draw [thick](2.887,-0.4)--(2.887,0.4);
		
		\draw[thick,rotate=60](-1.3,0.1)--(-0.7,0.4);
		\draw[thick,rotate=60](1.3,0.1)--(0.7,0.4);
		\draw [thick,rotate=60](0,0.8)--(0,-0.8);
	
		\draw[thick,rotate=60](2.3,-0.1)--(2.887,-0.4)--(3.5,-0.1);
		\draw[thick,rotate=60](4.5,-0.1)--(5.087,-0.4);
		\draw [thick,rotate=60](2.887,-0.4)--(2.887,0.4);

		\draw[thick,rotate=120](-1.3,0.1)--(-0.7,0.4);
		\draw[thick,rotate=120](1.3,0.1)--(0.7,0.4);
		\draw [thick,rotate=120](0,0.8)--(0,-0.8);
	
		\draw[thick,rotate=120](2.3,-0.1)--(2.887,-0.4)--(3.5,-0.1);
		
		\draw[thick,rotate=120](4.5,-0.1)--(5.087,-0.4);
		
		\draw [thick,rotate=120](3.487,-0.4)--(3.487,0.4);

		\draw[thick,rotate=180](-1.3,0.1)--(-0.7,0.4);
		\draw[thick,rotate=180](1.3,0.1)--(0.7,0.4);
		
		\draw[thick,rotate=180](2.3,-0.1)--(2.887,-0.4)--(3.5,-0.1);
		
		\draw[thick,rotate=180](4.5,-0.1)--(5.087,-0.4);
		
		\draw [thick,rotate=180](2.887,-0.4)--(2.887,0.4);

		\draw[thick,rotate=240](-1.3,0.1)--(-0.7,0.4);
		\draw[thick,rotate=240](1.3,0.1)--(0.7,0.4);
		\draw[thick,rotate=240](-1.3,-0.1)--(-0.7,-0.4);
		\draw[thick,rotate=240](1.3,-0.1)--(0.7,-0.4);

		\draw[thick,rotate=240](2.3,-0.1)--(2.887,-0.4)--(3.5,-0.1);
		\draw[thick,rotate=240](4.5,-0.1)--(5.087,-0.4);
	
		\draw [thick,rotate=240](2.887,-0.4)--(2.887,0.4);

		\draw[thick,rotate=300](-1.3,0.1)--(-0.7,0.4);
		\draw[thick,rotate=300](1.3,0.1)--(0.7,0.4);
		\draw[thick,rotate=300](-1.3,-0.1)--(-0.7,-0.4);
		\draw[thick,rotate=300](1.3,-0.1)--(0.7,-0.4);
		\draw[thick,rotate=300](2.3,-0.1)--(2.887,-0.4)--(3.5,-0.1);
		\draw[thick,rotate=300](4.5,-0.1)--(5.087,-0.4);
		\draw [thick,rotate=300](2.887,-0.4)--(2.887,0.4);

        \draw[thick,orange](-2.44,-0.13)--(-2.44,0.13)--(-2.7,0.13)--(-2.7,-0.13)--(-2.44,-0.13);
		\draw[thick,orange,rotate=120](-2.44,-0.13)--(-2.44,0.13)--(-2.7,0.13)--(-2.7,-0.13)--(-2.44,-0.13);
\draw[thick,orange,rotate=60](-2.44,-0.13)--(-2.44,0.13)--(-2.7,0.13)--(-2.7,-0.13)--(-2.44,-0.13);
		\draw[thick,orange,rotate=180](-2.44,-0.13)--(-2.44,0.13)--(-2.7,0.13)--(-2.7,-0.13)--(-2.44,-0.13);
\draw[thick,orange,rotate=240](-2.44,-0.13)--(-2.44,0.13)--(-2.7,0.13)--(-2.7,-0.13)--(-2.44,-0.13);
\draw[thick,orange,rotate=300](-2.44,-0.13)--(-2.44,0.13)--(-2.7,0.13)--(-2.7,-0.13)--(-2.44,-0.13);
 \filldraw[orange] (2.57,0)circle(0.04cm);
    \filldraw [orange,rotate=60](2.57,0)circle(0.04cm);
    \filldraw [orange,rotate=120](2.57,0)circle(0.04cm);
    \filldraw [orange,rotate=240](2.57,0)circle(0.04cm);
    \filldraw [orange,rotate=300](2.57,0)circle(0.04cm);
    \filldraw [orange,rotate=180](2.57,0)circle(0.04cm);
		\filldraw [red] (0,0) circle [radius=0.04];
		\node [right] at (-0.01,-0.35) {\textcolor{red}{\tiny$0$}};
	\end{tikzpicture}
	\caption{\footnotesize  The jump contour $\Sigma^{E}$. The red circles represents $U$ and the orange  rectangular box stands for $\mathbb{B}_j$, $j=1,\cdots,6$. }
	\label{figE}
\end{figure}
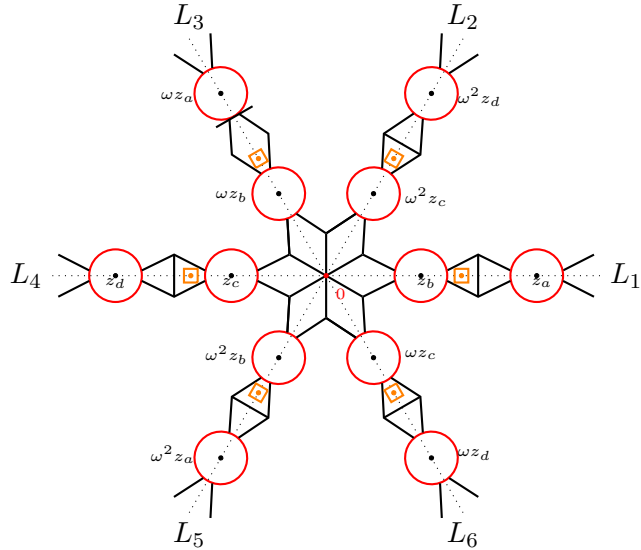

Next we prove that the above RH problem \ref{p1e1} can be approximated by the following RH problem, which owns the same jump condition with RH problem \ref{p1e1}, but without pole singularities.

\begin{RHP}\label{p1e2}
	 Find a matrix-valued function $E^{(1)}(z)$  with  properties:
\begin{itemize}	

\item $E^{(1)}(z)$ is analytical  in $\mathbb{C}\setminus  \Sigma^{E} $.

\item $E^{(1)}(z)$  has continuous boundary values $E^{(1)}_\pm(z)$  satisfying
$$E^{(1)}_+(z)=E^{(1)}_-(z)V^{E}(z), \ \ z\in \Sigma^{E},$$
where the jump matrix $V^{E}(z)$ is given by \eqref{deVE}.

\item $E^{(1)}(z) = I+\mathcal{O}(z^{-1}),\hspace{0.5cm} z \rightarrow \infty.$
\end{itemize}
\end{RHP}

  We will show  that  the  RH problem \ref{p1e2}  for  $E^{(1)}(z)$ is solvable for large $t$ as a small norm  RH problem.
  By  Proposition \ref{asymlo},  the jump matrix  $V^{E}(z)$ admits   the following estimates.
  	\begin{equation}
  		\parallel V^{E}(z)-I\parallel_{L^p(\Sigma^E)}\lesssim\left\{\begin{array}{llll}
  		e^{-ct^{3\delta_1}},  & z\in \Sigma^{E}\setminus (U \cup (\cup_{j=1}^6 \mathbb{B}_j)),\\[4pt]
  		 t^{-K_p},  & z\in  \partial U,\\
          t^{-1}, & z \in \partial \mathbb{B}_j,
  		\end{array}\right. \label{VE-I}
  	\end{equation}
 for some positive $c$ with $K_\infty = \delta_1$ and $K_2 = 1/6+\delta_1$.
Therefore,    the   existence and uniqueness  of   RH problem \ref{p1e1} can be shown  by using  a  small-norm RH problem \cite{RN9,RN10}.  Moreover,  its solution   can be given by
\begin{equation}
E^{(1)}(z)=I+\frac{1}{2\pi i}\int_{\Sigma^{E}}\dfrac{\varpi(\eta) (V^{E}(\eta)-I)}{\eta-z}ds,\label{Ez}
\end{equation}
where $\varpi\in I + L^2(\Sigma^{E})$ is the unique solution of the following equation
\begin{equation}
  \varpi=I + \mathcal{C}_E\varpi.
\end{equation}
Here $\mathcal{C}_E$ is  a integral operator  defined by
\begin{equation}
\mathcal{C}_E(f)(z)=\mathcal{C}^-\left( f(V^{E}(z) -I)\right) ,
\end{equation}
and  $\mathcal{C}^-$ is the  Cauchy projection operator on $\Sigma^{E}$.
By  (\ref{VE-I}),   we have
\begin{equation}
\parallel \mathcal{C}_E\parallel_{L^2(\Sigma^{E})}\leq\parallel \mathcal{C}^-\parallel_{L^2(\Sigma^{E})} \parallel V^{E}(z)-I\parallel_{L^\infty(\Sigma^{E})} \lesssim t^{-\delta_1},
\end{equation}
which implies that  $1-\mathcal{C}_E$ is invertible  for   sufficiently large $t$.    So  $\varpi$  exists and is unique with
\begin{align}
  \varpi = I + (1-\mathcal{C}_E)^{-1} (\mathcal{C}_EI).
\end{align}
Moreover,  the following estimates hold:
\begin{equation}
\parallel \mathcal{C}_EI \parallel_{L^2(\Sigma^{E})} \lesssim t ^{-1/6-\delta_1/2}, \quad \parallel \varpi -I \parallel_{L^2(\Sigma^{E}) } \lesssim t ^{-1/6-\delta_1/2}. \label{normrho}
\end{equation}

In order to reconstruct the solution $u(y,t)$ of (\ref{Novikov}), it is  necessary to consider the large time asymptotic behavior of $E^{(1)}(e^{\frac{i\pi}{6}})$. Note that when  estimating its  asymptotic behavior, based on (\ref{Ez}) and (\ref{VE-I}), our calculations need only to focus on  $\partial U$, as  it  exponentially tends to zero on the remaining boundary.
\begin{Proposition}\label{p1asye}
	When $z= e^{\frac{i\pi}{6}}$, we have
	\begin{align*}
	E^{(1)}(e^{\frac{i\pi}{6}})=I+\frac{1}{2\pi i}\int_{\Sigma^{E}}\dfrac{\varpi(\eta) (V^{E}(\eta)-I)}{\eta-e^{\frac{i\pi}{6}}}d\eta,
	\end{align*}
	with the large time asymptotic behavior
	\begin{equation}\label{asyE}
	E^{(1)}(e^{\frac{i\pi}{6}})=I+t^{-1/3}E_1+\mathcal{O}(t^{-2/3+2\delta_1}),
	\end{equation}
where $E_1$  is explicitly  computed by
	\begin{align}
	E_1
	&=- \sum_{l=0}^2 \sum_{j\in\{a,b,c,d \}}c_j^{-1/3} \frac{M^O( \omega^lz_j)M^{L,j}_{j1}  M^O( \omega^l\xi_j)^{-1} }{\omega^lz_j -e^{\frac{i\pi}{6}} } \nonumber \\
    &=- \sum_{j\in\{a,b,c,d \}}c_j^{-1/3} \Bigl(  \frac{M^O(  z_j)M^{L,0}_{j1}  M^O(z_j)^{-1} }{z_j -e^{\frac{i\pi}{6}} } -  \frac{\omega M^O( \omega z_j) \Gamma_3 \overline{M^{L,0}_{j1}} \Gamma_3 M^O( \omega z_j)^{-1} }{\omega z_j -e^{\frac{i\pi}{6}} }\nonumber
    \\
    &-  \frac{\omega^2 M^O( \omega^2 z_j)\Gamma_2 \overline{M^{L,0}_{j1}} \Gamma_2  M^O( \omega^2z_j)^{-1} }{\omega^2z_j -e^{\frac{i\pi}{6}} } \Bigr).\label{H0}
	\end{align}
\end{Proposition}

\begin{proof}
From \eqref{p1deml} and \eqref{deVE}, it follows that
\begin{align*}
E^{(1)}(e^{\frac{i\pi}{6}}) &= I +\frac{1}{2\pi i} \sum_{l=0}^2 \sum_{j \in \{a,b,c,d \}} \oint_{\partial U^l_j} \frac{M^O(\eta) (M_j^l(\eta)^{-1} - I) M^O(\eta)^{-1} }{\eta-e^{\frac{i\pi}{6}}} d\eta + \mathcal{O}(t^{-1/3-\delta_1}) \\
&= I -\frac{1}{2\pi i}  \sum_{j \in \{a,b,c,d \}} (c_j t)^{-1/3}\oint_{\partial U^0_j} \frac{M^O(\eta) M^{L,0}_{j1} M^O(\eta)^{-1} }{(\eta-e^{\frac{i\pi}{6}})(\eta-z_j)}  d\eta\\
 &- \frac{1}{2\pi i}  \sum_{j \in \{a,b,c,d \}} (c_j t)^{-1/3} \oint_{\partial U^1_j} \omega \frac{M^O(\eta)  \Gamma_3 \overline{M^{L,0}_{j1}} \Gamma_3 M^O(\eta)^{-1} }{(\eta-e^{\frac{i\pi}{6}})(\eta-\omega z_j)}  d\eta\\
 &- \frac{1}{2\pi i}  \sum_{j \in \{a,b,c,d \}} (c_j t)^{-1/3}\oint_{\partial U^2_j} \omega^2 \frac{M^O(\eta)  \Gamma_2 \overline{M^{L,0}_{j1}} \Gamma_2 M^O(\eta)^{-1} }{(\eta-e^{\frac{i\pi}{6}})(\eta-\omega^2 z_j)}  d\eta\\
 & + \mathcal{O}(t^{-2/3+2\delta_1}) \\
 & = I -  \sum_{j\in\{a,b,c,d \}}(c_j t)^{-1/3} M^O(e^{\frac{i\pi}{6}}) \left( \frac{M^{L,0}_{j1}}{z_j -e^{\frac{i\pi}{6}} } + \frac{ \omega  \Gamma_3 \overline{M^{L,0}_{j1}} \Gamma_3}{\omega z_j- e^{\frac{i\pi}{6}}} + \frac{\omega^2   \Gamma_2 \overline{M^{L,0}_{j1}} \Gamma_2}{\omega^2 z_j- e^{\frac{i\pi}{6}}}  \right)        M^O(e^{\frac{i\pi}{6}})^{-1} \\
 & + \mathcal{O}(t^{-2/3+2\delta_1}).
\end{align*}

\end{proof}

Moreover, for $j=1,\cdots,6$,
\begin{equation}
 E^{(1)}(\varkappa_j) = I + \mathcal{O}(t^{-1+\rho}).
\end{equation}
Finally we consider the error between $E(z)$ and $E^{(1)}(z)$. Define the error function
\begin{equation}
E^{(2)}(z) = E(z) E^{(1)}(z)^{-1},
\end{equation}
which is a solution of a RH problem only has the following singularities
\begin{align}
	\lim_{z\to \varkappa_j}E^{(2)}(z)=&\lim_{z\to \varkappa_j}M^R(z)M^B_j(z)^{-1}M^{O }(z)^{-1}E^{(1)}(z)^{-1},
\end{align}
which leads to
\begin{align}
	E^{(2)}(z)=I+\sum_{j=1 }^6\left( \frac{E^{(2),j}_{-2}}{(z-\varkappa_j)^2}+\frac{E^{(2),j}_{-1}}{z-\varkappa_j}\right),
\end{align}
where $E^{(2),j}_{-l}$, $j=1,\cdots,6$ and $l=1,2$, represents the coefficients of the term of $(z-\varkappa_j)^l$ in the expansion of $E^{(2)}(z)$.
Then, from \eqref{resmr1} and \eqref{mosin}, it follows that
\begin{align}
	E^{(1),j}_{-2}=&\left(\begin{array}{ccc}
		\alpha^{R}_\pm &	\alpha^{R}_\pm & \beta^{R}_\pm \\
		-\alpha^{R}_\pm & -\alpha^{R}_\pm & -\beta^{R}_\pm\\
		0	&	0 & 0
	\end{array}\right)M^B_j(\varkappa_j)^{-1}\nonumber\\
&\left(\begin{array}{ccc}
	\tilde{\alpha}^{O}_\pm &	\tilde{\alpha}^{O}_\pm & \tilde{\beta}^{O}_\pm \\
	-\tilde{\alpha}^{O}_\pm & -\tilde{\alpha}^{O}_\pm & -\tilde{\beta}^{O}_\pm\\
	0	&	0 & 0
\end{array}\right)E^{(1)}(\varkappa_j)^{-1}.
\end{align}
As $t\to\infty$, bring (\ref{MBkap})  into the above formula leads to
\begin{align}
		E^{(1),j}_{-2}=\mathcal{O}(t^{-1+\rho}).
\end{align}
Analogously, as $t\to\infty$, the  coefficient of $(z-\varkappa_j)^{-1}$ satisfies
\begin{align}
	E^{(1),j}_{-1}=\mathcal{O}(t^{-1+\rho}).
\end{align}
Summarizing above   results   gives the following proposition.
\begin{Proposition}\label{asyE}
	Taking $z= e^{\frac{i\pi}{6}}$, we have
	\begin{align}
		E(e^{\frac{i\pi}{6}})
		=&I+t^{-1/3}E_1+\mathcal{O}(t^{-2/3+2\delta_1}),
	\end{align}
where $E_1$ is given by    (\ref{H0}).
\end{Proposition}

\subsection{Contribution from  $\bar\partial$-components }\label{sec8}
\quad Now we consider the   asymptotics  of $M^{(4)}$ of the  $\bar{\partial}$-problem \ref{RHP4},
whose solution   can be given by an integral equation
\begin{equation}
M^{(4)}(z)=I+\frac{1}{\pi}\iint_\mathbb{C}\dfrac{M^{(4)}(\eta)W^{(4)} (\eta)}{\eta-z}dm(\eta),\label{m3}
\end{equation}
where $m(\eta)$ is the Lebesgue measure on the $\mathbb{C}$.
Define the left Cauchy-Green integral  operator,
\begin{equation*}
fC_z(z)=\frac{1}{\pi}\iint_{\mathbb{C}}\dfrac{f(\eta)W^{(4)} (\eta)}{\eta-z}dm(\eta),
\end{equation*}
then  the above equation (\ref{m3})  can be rewritten as
\begin{equation}
\left(I-C_z \right) M^{(4)}(z)=I.\label{deM3}
\end{equation}
Aiming at estimating $M^{(4)}(z)$, we need to evaluate the norm of the integral operator
 $\left(I-  {C}_z \right)^{-1}$
 in this transition zone.

By  Lemma \ref{p1the12}, we can  show that for sufficiently large $t$ the operator
$\mathcal{C}_z$ is small-norm, so that the resolvent operator
 $(I-\mathcal{C}_z)^{-1}$ exists and can be expressed as a Neumann series.

\begin{Proposition}\label{Cz1}
The norm of the integral operator $C_z$ satisfies the inequality
	\begin{equation}\label{p1cz1}
	\parallel C_z\parallel_{L^\infty\to L^\infty}\lesssim t^{-1/3}, \ \ t\to\infty.
	\end{equation}

\end{Proposition}
\begin{proof}
	For any $f\in L^\infty$,
	\begin{align}
	\parallel fC_z \parallel_{L^\infty}&\leq \parallel f \parallel_{L^\infty}\frac{1}{\pi}\iint_\mathbb{C}\dfrac{|W^{(4)} (\eta)|}{|z-\eta|}dm(\eta)\nonumber.
	\end{align}

We detail the case for matrix functions having support in the region $\Omega_1^0$, the case for the other regions follows similarly.	
Recall the definition of $W^{(4)} (z)=M^{R}(z)\bar{\partial}\mathcal{R}^{(3)}(z)M^{R}(z)^{-1}$.  Note that $W^{(4)} (z)\equiv0$ out of $\overline{\Omega}$.
	Proposition  \ref{asymlo}  and \ref{p1asye} implies  the boundedness of $M^{R}(z)$ and $M^{R}(z)^{-1}$ for $z\in \overline{\Omega}_{1}$. By \eqref{p1dbarr3} and \eqref{p1w4}, it follows that
	\begin{equation*}
	\frac{1}{\pi}\iint_{\Omega_{1}}\dfrac{|W^{(4)} (\eta)|}{|z-\eta|}dm(\eta)\lesssim \frac{1}{\pi}\iint_{\Omega_{1}}\dfrac{|\bar{\partial}R_{1} (\eta) e^{-it\theta_{12}}|}{|z-\eta|}dm(\eta).
	\end{equation*}
	From Lemma \ref{p1the12},  we divide $\Omega_1^0$ in two regions:
	\begin{itemize}
		\item  $\{ z \in \Omega_1^0: \re z(1-|z|^{-2}) \le 2 \} \subseteq  \Omega_A := \{ z \in \Omega_1^0: \re z \le 3\}$,
		\item  $\{ z \in \Omega_1^0: \re z(1-|z|^{-2})>2 \} \subseteq  \Omega_B := \{ z \in \Omega_1^0: \re z \ge 2\}$.
	\end{itemize}
	  Set $\eta = u+1/p_1 +vi$, $z=z_R + iz_I$, $u,v,z_R, z_I \in \mathbb{R}$, and $1/q+1/p=1$ with $p>2$.

	Referring  to (\ref{dbarRj3}) and (\ref{dbarRj4}) in Proposition \ref{proR1}, then the following integral 	can be divided into three  part:
	\begin{align*}
	\iint_{\Omega_{1}^0}\dfrac{|\bar{\partial}R_{1} (\eta)|e^{t\text{Im}\theta_{12}}}{|z-\eta|}dm(\eta)\lesssim \hat{I}_1+\hat{I}_2+\hat{I}_3,	
	\end{align*}
	with
	\begin{align*}
	&\hat{I}_1:=\iint_{\Omega_{1}^0}\dfrac{\left(|\tilde{r}'(u+1/p_1)|+ |\mathcal{X}'(u+1/p1)| \right)e^{t\text{Im}\theta_{12}}}{|z-\eta|}dm(\eta), \\
	&\hat{I}_2:=\iint_{\Omega_{B}}\dfrac{|u|^{-1/2}e^{t\text{Im}\theta_{12}}}{|z-\eta|}dm(\eta).
	\end{align*}
Our task now is to estimate the above integrals $\hat{I}_i$, $i=1,2,3$, respectively.

	\begin{align*}
	 \hat{I}_1\leq& \iint_{\Omega_A} \dfrac{\left(|\tilde{r}'(u+1/p_1)|+ |\mathcal{X}'(u+1/p_1)| \right)e^{t\text{Im}\theta_{12}}}{|z-\eta|}dm(\eta) \\
	+ & \iint_{\Omega_B} \dfrac{\left(|\tilde{r}'(u+1/p_1)|+ |\mathcal{X}'(u+1/p_1)| \right)e^{t\text{Im}\theta_{12}}}{|z-\eta|}dm(\eta)\\
	 \leq& \int_{0}^{(3-1/p_1)\tan \varphi_0}\int_{v}^{3}\dfrac{e^{-t \frac{\sqrt{3}(7-3 \tilde{k}_1^2)}{(\tilde{k}_1+1)^2}  v u^2 }}{|z-\eta|}dudv
	+\int_{2}^{+\infty}\int_{u \tan (\varphi_0/3) }^{u \tan \varphi_0}\dfrac{e^{-t\frac{\sqrt{3}}{8}  v}}{|z-\eta|}dvdu.
\end{align*}

    Note that the following basic inequalities hold
    \begin{align}
    	\| |z-\eta|^{-1}\|_{L^q_u(v,\infty)} \lesssim |v-y|^{-1+1/q}, \quad \| e^{-tvu^2}\|_{L^p_u(v,\infty)}  \lesssim (tv)^{-1/(2p)}, \quad \| e^{-tv}\|_{L^p_u(v,\infty)}  \lesssim (pt)^{-1/p}.
    \end{align}
  	Then using Cauchy-Schwartz's inequality, we have
\begin{align}
	& \int_{0}^{(3-1/p_1)\tan \varphi_0}\int_{v}^{3}\dfrac{e^{-t \frac{\sqrt{3}(7-3 \tilde{k}_1^2)}{(\tilde{k}_1+1)^2}  v u^2 }}{|z-\eta|}dudv \nonumber \\
	&\lesssim t^{-1/4} \int_{0}^{(3-1/p_1)\tan \varphi_0} |v-z_I|^{-1/2} v^{-1/4}e^{-t\frac{\sqrt{3}(7-3 \tilde{k}_1^2)}{(\tilde{k}_1+1)^2}v^3}dv \lesssim t^{-1/3} \label{p1csdbar1}.
		\end{align}
		Using H\"{o}lder's inequality, we obtain
	\begin{align}\label{p1csdbar2}
	\int_{2}^{+\infty}\int_{u \tan (\varphi_0/3) }^{u \tan \varphi_0}\dfrac{e^{-t\frac{\sqrt{3}}{8}  v}}{|z-\eta|}dvdu  \lesssim t^{-1/p} \int_{2}^{\infty} e^{-t\frac{\sqrt{3}}{8}   u \tan (\varphi_0/3)}|u-z_R|^{-1+1/q}  du\lesssim t^{-1}.	
	\end{align}
Together \eqref{p1csdbar1} with \eqref{p1csdbar2} gives us that
\begin{align}
   \hat{I}_1 \lesssim t^{-1/3}.
\end{align}

To estimate $\hat{I}_2$,  it follows from H\"{o}lder's inequality again that
\begin{align*}
 \hat{I}_2 \le  &  \iint_{\Omega_B}  \dfrac{|u|^{-1/2}e^{t\text{Im}\theta_{12}}}{|z-\eta|}dm(\eta)
   \le \int_{2}^{+\infty}\int_{u \tan (\varphi_0/3) }^{u \tan \varphi_0} \dfrac{|u|^{-1/2}e^{-t\frac{\sqrt{3}}{8}  v}}{|z-\eta|}dvdu \\
   \lesssim  & \int_{2}^{+\infty} u^{1/p-1/2} |u-z_R|^{-1+1/q} e^{-t \frac{\sqrt{3}}{8} u \tan (\varphi_0/3)} du \lesssim t^{-1/2}.
\end{align*}

Then combining the  estimates for $\hat{I}_1$ and $\hat{I}_2$  yields the estimate
\begin{align*}
\frac{1}{\pi}\iint_{\Omega_{1}}\dfrac{|W^{(4)} (\eta)|}{|z-\eta|}dm(\eta)\lesssim  t^{-1/3}.
\end{align*}

The integral over other regions can be estimated in similar manners, which finally confirms \eqref{p1cz1}.

\end{proof}

Then from (\ref{deM3}), we immediately arrive at  the existence and uniqueness of $M^{(4)}(z)$ for $z\in\mathbb{C}$.
Take $z=e^{\frac{i\pi}{6}}$ in (\ref{m3}), then
\begin{equation}\label{p1m3pi6}
M^{(4)}(e^{\frac{\pi i}{6}})=I+\frac{1}{\pi}\iint_{\mathbb{C}}\dfrac{M^{(4)}(\eta)W^{(4)} (\eta)}{\eta-e^{\frac{\pi i}{6}}}dm(\eta).
\end{equation}
To reconstruct the solution of (\ref{Novikov}), we need the local behaviors of \eqref{p1m3pi6} as $t \to \infty$.
\begin{Proposition}\label{asyM3i1}
	There exists a constant $T_1$, such that for all $t>T_1$, the solution $M^{(4)}(z)$  of  the $\bar{\partial}$-problem  admits the following estimate:
	\begin{align}\label{p1asym3}
	| M^{(4)}(e^{\frac{\pi i}{6}})-I| \lesssim t^{-2/3}.
	\end{align}
\end{Proposition}
\begin{proof}
Similar to the proof of Proposition \ref{Cz1}, we only give the proof for the integral over $\Omega_1^0$ and the integral on other regions can be obtained in the same way.	
Proposition \ref{Cz1} and (\ref{deM3}) imply that the boundedness of  $M^{(4)}(z)$ for $z \in \Omega_1^0$ as $t \to \infty$.  Then it is sufficient to consider the integral
 \begin{align*}
 \iint_{\mathbb{C}}\dfrac{|M^{(4)}(\eta)W^{(4)} (\eta)|}{\eta-e^{\frac{\pi i}{6}}}dm(\eta) \lesssim \iint_{\mathbb{C}}\dfrac{|W^{(4)} (\eta)|}{\eta-e^{\frac{\pi i}{6}}}dm(\eta).
 \end{align*}

  Referring  (\ref{dbarRj4}) in Proposition \ref{proR1}, and note that  $|\mathcal{X}'(\text{Re}z)|=0$ in $\overline{\Omega_{1}^0}$, this integral 	can be divided into two parts
	\begin{align}
	\iint_{\Omega_1^0}\dfrac{|W^{(4)} (\eta)|}{\eta-e^{\frac{\pi i}{6}}}dm(\eta)\le \left( \iint_{\Omega_A} +\iint_{\Omega_B} \right) \dfrac{|\bar\partial R_{1}(\eta)| e^{2t\text{Im}\theta_{12}}}{\eta-e^{\frac{\pi i}{6}}}dm(\eta).
	\end{align}

Let $\eta = u+1/p_1 + vi$ with $u,v \in \mathbb{R}$.  From \eqref{dbarRj4} and $|\eta-e^{\frac{\pi i}{6}}|^{-1}$ is bounded for $z\in \Omega_A$, we have
\begin{align*}
\iint_{\Omega_A} \dfrac{|\bar\partial R_{1}(\eta)|e^{2t\text{Im}\theta_{12}}}{\eta-e^{\frac{\pi i}{6}}}dm(\eta)& \le \int_{0}^{(3-1/p_1)\tan \varphi_0}\int_{v}^{3} |\tilde{r}'(u+1/p_1)| e^{-t \frac{\sqrt{3}(7-3 \tilde{k}_1^2)}{(\tilde{k}_1+1)^2}  v u^2 }dudv \\
& \lesssim \int_{0}^{\infty} \int_{v}^{\infty} e^{-t \frac{\sqrt{3}(7-3 \tilde{k}_1^2)}{(\tilde{k}_1+1)^2}  v u^2 }dudv \lesssim t^{-2/3}.
\end{align*}

To estimate the integral over $\Omega_B$,  we use the fact $|\eta-e^{\frac{\pi i}{6}}|^{-1} \le|\eta|^{-1} $ and the estimate \eqref{dbarRj4} to obtain that
\begin{align}
	\iint_{\Omega_B}  \dfrac{|\bar\partial R_{1}(\eta)| e^{2t\text{Im}\theta_{12}}}{\eta-e^{\frac{\pi i}{6}}}dm(\eta) \le \int_{2}^{+\infty} \int_{0}^{u} (u^2+v^2)^{-1/2} e^{-t \frac{\sqrt{3}}{8} v}dv du \lesssim t^{-1}.
\end{align}
Combining the above two estimates, we arrive at the desired result  \eqref{p1asym3}.
\end{proof}

\subsection{Proof of Theorem \ref{last} }\label{p1sec9}

\quad Now we begin to construct the large time asymptotics of the Novikov equation (\ref{Novikov}) in the transition zone $|\xi+\frac{1}{8}|t^{2/3}<C$.
Inverting the sequence of transformations (\ref{transm1}), \eqref{p1transm2}, (\ref{transm2}), (\ref{transm4}) and (\ref{transm3}), we have for $z \in \mathbb{C} \setminus U$,
\begin{align}
M(z)=&M^{(4)}(z)E(z)M^{ O}(z)\mathcal{R}^{(3)}(z)^{-1}T(z)^{-1}G(z)^{-1} + \mathcal{O}(e^{-ct}).
\end{align}
To  reconstruct the solution $u(x,t)$ by using (\ref{recons u}),   we take $z=e^{\frac{\pi i}{6}}$. In this case,  $\mathcal{R}^{(3)}(z)=G(z)=I$, and then we can obtain that
\begin{align}\label{p1sol1}
	M(e^{\frac{\pi i}{6}})=&M^{(4)}(e^{\frac{\pi i}{6}}) E(e^{\frac{\pi i}{6}})M^{O}(e^{\frac{\pi i}{6}})
	T(e^{\frac{\pi i}{6}})^{-1} + \mathcal{O}(e^{-ct}).
\end{align}
Using Proposition \ref{asyM3i1}, \eqref{p1sol1} comes down to
\begin{align}\label{p1sol11}
	M(e^{\frac{\pi i}{6}})= E(e^{\frac{\pi i}{6}})M^{O}(e^{\frac{\pi i}{6}})
	T(e^{\frac{\pi i}{6}})^{-1} + \mathcal{O}(t^{-2/3}).
\end{align}
Substitute the above estimates into  (\ref{recons u}) and (\ref{recons x}), and obtain
\begin{align*}
	u(y,t)&= u^\lozenge(y,t;\tilde{\mathcal{D}}_\lozenge)\left( T_1(e^{\frac{\pi i}{6}})T_3(e^{\frac{\pi i}{6}})\right) ^{-1/2}-1 \\
	&+\frac{1}{2}\left( T_1(e^{\frac{\pi i}{6}})T_3(e^{\frac{\pi i}{6}})\right) ^{-1/2} f_{11}t^{-1/3}+\mathcal{O}(t^{-2/3+2\delta_1}), \nonumber \\
	x(y,t)&= y+ \frac{1}{2} \ln\frac{M^O_{33}(e^{\frac{i\pi}{6}};y,t)}{M^O_{11}(e^{\frac{i\pi}{6}};y,t)}+\frac{1}{2} \ln\left(\frac{T_1(e^{\frac{\pi i}{6}})}{T_3(e^{\frac{\pi i}{6}})}\right) + \frac{1}{2} f_{12} t^{-1/3} +\mathcal{O}(t^{-2/3+2\delta_1}), \\
	&=x^\lozenge(y,t;\tilde{\mathcal{D}}_\lozenge)+\frac{1}{2} \ln T_{13}(e^{\frac{\pi i}{6}})+ \frac{1}{2} f_{12} t^{-1/3} +\mathcal{O}(t^{-2/3+2\delta_1}), \nonumber
\end{align*}
where
\begin{align}
f_{11} &= \frac{1}{2} \left( \tilde{m}^\lozenge_1(y,t) \left(\frac{M_{33}^O(e^{\frac{\pi i}{6}})}{M_{11}^O(e^{\frac{\pi i}{6}})} \right)^{1/2} -\tilde{m}^\lozenge_3 (y,t) \left(\frac{M_{33}^O(e^{\frac{\pi i}{6}})}{M_{11}^O(e^{\frac{\pi i}{6}})} \right)^{-1/2} \right)f_{12} \nonumber\\
	& + (E_1 M^O(e^{\frac{\pi i}{6}}))_1 \left(\frac{M_{33}^O(e^{\frac{\pi i}{6}})}{M_{11}^O(e^{\frac{\pi i}{6}})} \right)^{1/2}  + (E_1 M^O(e^{\frac{\pi i}{6}}))_3 \left(\frac{M_{33}^O(e^{\frac{\pi i}{6}})}{M_{11}^O(e^{\frac{\pi i}{6}})} \right)^{-1/2}, \label{p1f11}\\
f_{12} & = (E_1 M^O(e^{\frac{\pi i}{6}}))_{33} M^O_{33}(e^{\frac{\pi i}{6}})^{-1} -  (E_1 M^O(e^{\frac{\pi i}{6}}))_{11} M^O_{11}(e^{\frac{\pi i}{6}})^{-1},	\label{p1f12}
\end{align}
$(E_1 M^O)_{ij}$ represents the element in the $i$-th row and $j$-th column of the matrix $E_1 M^O$,  $(E_1 M^O)_{j} = \sum_{i=1}^{3}(E_1 M^O)_{ij}$,
and $u^\lozenge(y,t;\tilde{\mathcal{D}}_\lozenge)$, $x^\lozenge(y,t;\tilde{\mathcal{D}}_\lozenge)$, and $\tilde{m}^\lozenge_j(y,t)$ are defined in Corollary \ref{p1urxrsol}.
Bring \eqref{Tij} into the above formulas, we obtain the final result.

\appendix

\section{Modified Painlev\'{e} \uppercase\expandafter{\romannumeral2}  RH Problem} \label{appx}
The (homogeneous) Painlev\'{e} \uppercase\expandafter{\romannumeral2} equation is
\begin{equation}\label{p23}
	v_{ss} = 2v^3 +sv, \quad s \in \mathbb{R}.
\end{equation}
The standard Painlev\'{e} \uppercase\expandafter{\romannumeral2} equation is related to a $2 \times 2$ matrix-valued RH problem, here we give a modified $3 \times 3$ matrix-valued RH problem related to \eqref{p23} as follows.

Denote $\Sigma^P = \bigcup_{n=1}^6\left\{  \Sigma_n^P = e^{i\left(\frac{\pi}{6}+(n-1)\frac{\pi}{3}\right)} \mathbb{R}_+ \right\}$, see Figure \ref{Sixrays}. Let $\mathcal{C} =\{c_1,c_2,c_3\}$ be a set of complex constants such that
\begin{equation}
	c_1 -c_2 +c_3 +c_1 c_2 c_3=0,
\end{equation}
and define the matrices $\{ C_n\}_{n=1}^6$ by
\begin{equation*}
	C_n = 	\begin{pmatrix} 1 & 0 &0 \\ c_n e^{2i(\frac{4}{3}k^3+sk)}& 1 &0 \\ 0 & 0 & 1 \end{pmatrix}, \ n \ \text{odd}; \quad C_n = 	\begin{pmatrix} 1 & c_n e^{-2i(\frac{4}{3}k^3+sk)}&0 \\ 0 & 1 &0 \\ 0 & 0 & 1 \end{pmatrix}, \ n \ \text{even},
\end{equation*}
where $c_{n+3}=-c_n,\ n=1,2,3$.
Then there exists a countable set $\mathcal{S}_{\mathcal{C}} =\{ s_j\}_{j=1}^\infty \subset \mathbb{C}$ with $s_j \to \infty$ as $j\to\infty$, such that the following RH problem
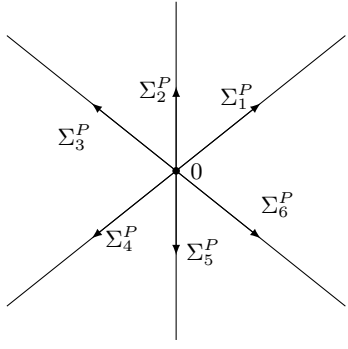
\begin{figure}
	\begin{center}
		\begin{tikzpicture}[scale=0.9]
			\node[shape=circle,fill=black,scale=0.15] at (0,0) {0};
			\node[below] at (0.3,0.25) {\footnotesize $0$};
			\draw [] (0,-2.5 )--(0,2.5);
			\draw [-latex] (0,0)--(0,1.25);
			\draw [-latex] (0,0 )--(0,-1.25);
			\draw [] (0,0 )--(2.5,2);
			\draw [-latex] (0,0)--(1.25,1);
			\draw [] (0,0 )--(2.5,-2);
			\draw [-latex] (0,0)--(1.25,-1);
			\draw [] (0,0 )--(-2.5,2);
			\draw [-latex] (0,0)--(-1.25,1);
			\draw [] (0,0 )--(-2.5,-2);
			\draw [-latex] (0,0)--(-1.25,-1);

			\node at (0.9,1.1) {\footnotesize$\Sigma^P_1$};
			\node at (1.5,-0.5) {\footnotesize$\Sigma^P_6 $};
			\node at (-1.5,0.5) {\footnotesize$\Sigma^P_3$};
			\node at (-0.8,-1) {\footnotesize$\Sigma^P_4$};
			\node at (-0.3,1.2) {\footnotesize$\Sigma^P_2$};
			\node at ( 0.4,-1.2) {\footnotesize$\Sigma^P_5$};
			
			
		\end{tikzpicture}
		\caption{ \footnotesize { The jump contour $\Sigma^P$.}}
		\label{Sixrays}
	\end{center}
\end{figure}

\begin{RHP}\label{1modp2}
	Find   $M^{P}(k)=M^{P}(k,s)$ with properties
	\begin{itemize}
		\item Analyticity: $M^{P}(k)$ is analytical in $\mathbb{C}\setminus \Sigma^{P}$.
		\item Jump condition:
		\begin{equation*}
			M^{P}_+( k)=M^{P}_-(k)C_n, \quad k \in \Sigma_n^P.
		\end{equation*}
		
	\item Asymptotic behavior:
\begin{align*}
	&M^{P}( k)=I+\mathcal{O}(k ^{-1}),	\quad k \to  \infty,\\
	& M^{P}( k) = \mathcal{O}(1),\quad k \to 0.
\end{align*}

\end{itemize}
\end{RHP}
\noindent has a unique solution $M^P(k,s)$ for each $s \in \mathbb{C} \setminus  \mathcal{S}_\mathcal{C}$. For each $n$, the restriction of $M^P(k,s)$ to $\arg k \in \left(\frac{\pi(2n-3)}{6}, \frac{\pi(2n-1)}{6}\right)$ admits an analytic continuation to $\left( \mathbb{C} \setminus  \mathcal{S}_\mathcal{C} \right) \times \mathbb{C}$ and  there are smooth function $\{M_j^P(s)\}_{j=1}^\infty$ of $s \in \mathbb{C} \setminus  \mathcal{S}_\mathcal{C}$ such that, for each integer $N \ge 0$,
\begin{equation}\label{stanp}
M^P(k) = I + \sum_{j=1}^N \frac{M_j^P(s)}{k^j} + \mathcal{O}(k^{-N-1}),\quad k \to \infty,
\end{equation}
uniformly for $s$ in compact subsets of  $\mathbb{C} \setminus  \mathcal{S}_\mathcal{C}$ and for $\arg k \in [0, 2\pi]$.
Moreover,
\begin{align}\label{up2}
\left(M_1^P(s)\right)_{12} =  \left(M_1^P(s)\right)_{21}= \frac{1}{2} u(s),
\end{align}
solves the Painlev\'{e} \uppercase\expandafter{\romannumeral2} equation.
The map $(c_1,c_2,c_3) \in \mathcal{C} \to u(\cdot;c_1,c_2,c_3)$ is a bijection
\begin{equation}
\{(c_1,c_2,c_3) \in \mathbb{C}^3 | c_1 -c_2 +c_3 +c_1 c_2 c_3 =0   \} \to \{\text{solutions of} \ \eqref{p23}\},
\end{equation}
and $\mathcal{S}_\mathcal{C}$ is the set of poles of $ u(\cdot;c_1,c_2,c_3)$.
Moreover, if $\mathcal{C} = (c_1,0,-c_1)$ where $c_1 \in i \mathbb{R}$ with $|c_1| <1$, then the leading coefficient $M_1^P$ is given by
\begin{align}\label{posee}
M_1^P(s) = \frac{1}{2} \begin{pmatrix} -i\int_{s}^\infty u(\eta)^2d\eta& u(s) &0 \\ u(s) & i\int_{s}^\infty u(\eta)^2d\eta &0  \\ 0& 0 & 0 \end{pmatrix}.
\end{align}
For each $C_1 > 0$,
\begin{align}\label{mPbounded}
\sup_{k \in \mathbb{C}\setminus \Sigma^P} \sup_{s \geq -C_1} |M^P(k)|  < \infty.
\end{align}
The solution $u(s)$ of the Painlev\'{e} \uppercase\expandafter{\romannumeral2} equation  is specified by its asymptotics as $s \to \infty$
\begin{equation}
v(s) \sim -\im c_1 \mathrm{Ai}(s) \sim - \frac{\im c_1}{2\sqrt{\pi}} k^{-\frac{1}{4} e^{-\frac{2}{3}k^{3/2}}}.
\end{equation}
where $\mathrm{Ai}(s)$ denotes the Airy function.

\section{Model RH Problem for the Transition Zone} \label{appx2}
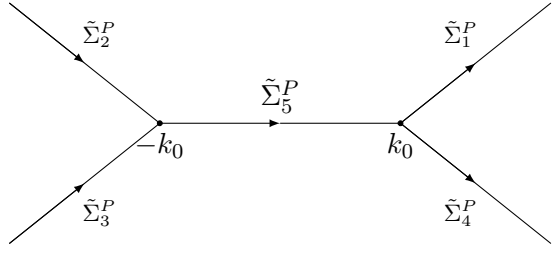
\begin{figure}[htbp]
	\centering
	\begin{tikzpicture}[scale=0.8]
		\filldraw (-2,0) circle [radius=0.04];
		\filldraw (2,0) circle [radius=0.04];
		\node  [above]  at (0,0) {$\tilde{\Sigma}^{P}_5$};
		\node  [below]  at (-2,0) {$-k_0$};
		\node  [below]  at (2,0) {$k_0$};
		\draw[-latex] (-2,0)--(0,0);
		\draw[] (0,0)--(2,0);
		\draw[] (-4.5,2)--(-2,0);
		\draw[-latex] (-4.5,2)--(-3.25,1);
		\draw[] (-4.5,-2)--(-2,0);
		\draw[-latex] (-4.5,-2)--(-3.25,-1);
		
		\draw[] (4.5,2)--(2,0);
		\draw[-latex] (2,0)--(3.25,1);
		\draw[] (4.5,-2)--(2,0);
		\draw[-latex] (2,0)--(3.25,-1);
		
		\node  [above]  at (-3,1.1) {\footnotesize$\tilde{\Sigma}^{P}_2$};
		\node  [below]  at (-3,-1.1) {\footnotesize$\tilde{\Sigma}^{P}_3$};

		\node  [above]  at (3,1.1){\footnotesize$ \tilde{\Sigma}^{P}_1$};
		\node  [below]  at (3,-1.1){\footnotesize$\tilde{\Sigma}^{P}_4$};
	\end{tikzpicture}
	\caption{\footnotesize The jump contours $\tilde{\Sigma}^{P}$.}
	\label{fmodelp2}
\end{figure}

Let $\tilde{\Sigma}^{P} = \tilde{\Sigma}^{P}(k_0)$ denote the contour $\tilde{\Sigma}^{P} = \cup_{j=1}^5 \tilde{\Sigma}^{P}_j $, as depicted in Figure \ref{fmodelp2}, where
\begin{align*}
	&\tilde{\Sigma}^{P}_1 = \{k| k=k_0 + r e^{\frac{\pi i}{6}},\ 0\le r <\infty \}, \quad \tilde{\Sigma}^{P}_2  = \{k|k=-k_0 +r e^{\frac{5\pi i}{6}},\ 0\le r <\infty  \},\\
	&\tilde{\Sigma}^{P}_3  = \{k|k \in \overline{\Sigma^{L}_2}\}, \quad \tilde{\Sigma}^{P}_4  =  \{k|k \in \overline{\Sigma^{L}_1}\}, \quad
	\tilde{\Sigma}^{P}_5  = \{k|-k_0 \le k \le k_0 \}.
\end{align*}
The model RH problem for the  transition zone is given as follows:
\begin{RHP}\label{modelp2}
	Find   $N^{P}(k)=N^{P}(k,s,c_1,k_0)$ with properties
	\begin{itemize}
		\item Analyticity: $N^{P}(k)$ is analytical in $\mathbb{C}\setminus \tilde{\Sigma}^{P}$.
		\item Jump condition:
		\begin{equation*}
			N^{P}_+( k)=N^{P}_-(k)V^{P}(k), \quad k \in \tilde{\Sigma}^{P},
		\end{equation*}
		where
		\begin{equation}\label{jumpl}
			V^{P}(k) = \begin{cases}
				\begin{pmatrix} 1 & 0 &0 \\ c_1 e^{2i (\frac{4k^3}{3} + s k)} & 1 &0 \\ 0 & 0 & 1 \end{pmatrix}, \quad k \in \tilde{\Sigma}^{P}_1 \cup \tilde{\Sigma}^{P}_2,\\
				\begin{pmatrix} 1 & -\bar{c}_1 e^{-2i (\frac{4k^3}{3} + s k)} &0 \\ 0 & 1 &0 \\ 0 & 0 & 1 \end{pmatrix},\quad k \in \tilde{\Sigma}^{P}_3\cup \tilde{\Sigma}^{P}_4,\\
				\begin{pmatrix} 1 -|c_1|^2 & - \bar{c}_1e^{-2i (\frac{4k^3}{3} + s k)} &0 \\ c_1 e^{2i (\frac{4k^3}{3} + s k)} & 1 &0 \\ 0 & 0 & 1 \end{pmatrix},\quad k \in \tilde{\Sigma}^{P}_5.
			\end{cases}
		\end{equation}
		\item Asymptotic behavior: $N^{P}( k)=I+\mathcal{O}(k ^{-1}),	\quad k \to  \infty.$

	\end{itemize}
\end{RHP}
Define the parameter subset $\mathcal{P}_T$ of $\mathbb{R}^3$ by
\begin{equation}\label{subpt}
	\mathcal{P}_T = \{(s,t,k_0) \in \mathbb{R}^3 | -C_1 \le s \le 0, t \ge T, \sqrt{|s|}/2 \le k_0 \le C_2  \},
\end{equation}
where $C_1,C_2 >0$ are constants.
Then there exists a $T \ge 1$ such that the above RH problem has a unique solution.
 and for each integer $N \ge 1$,
 \begin{equation}\label{stanp2}
N^P(k) = I + \sum_{j=1}^N \frac{N^{P}_{j}(s)}{k^j} + \mathcal{O}(k^{-N-1}),\quad k \to \infty,
\end{equation}
uniformly with respect to $\arg k \in [0,2\pi]$ and $(s,t,k_0) \in \mathcal{P}_T$ as $k \to \infty$, where $\{ N^{P}_{j}(s)\}_{j=1}^N $ are smooth functions in \eqref{stanp}.

\begin{proof}
	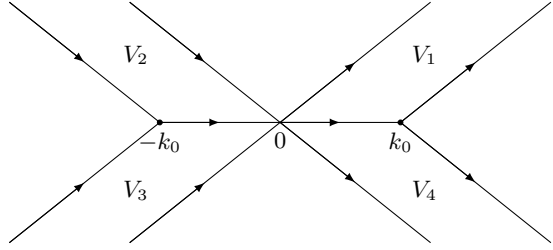
\begin{figure}[htbp]
		\centering
		\begin{tikzpicture}[scale=0.8]
			\filldraw (-2,0) circle [radius=0.04];
			\filldraw (2,0) circle [radius=0.04];
			\node  [below]  at (0,0) {\footnotesize$0$};
			\node  [below]  at (-2,0) {\footnotesize$-k_0$};
			\node  [below]  at (2,0) {\footnotesize$k_0$};
			
			\draw[] (-2,0)--(2,0);
			\draw[-latex] (-2,0)--(-1,0);
			\draw[-latex] (0,0)--(1,0);
			
			\draw[] (-4.5,2)--(-2,0);
			\draw[-latex] (-4.5,2)--(-3.25,1);
			\draw[] (0,0)--(-2.5,2);
			\draw[-latex](-2.5,2)--(-5/4,1);
			\draw[] (-4.5,-2)--(-2,0);
			\draw[-latex] (-4.5,-2)--(-3.25,-1);
			\draw[] (0,0)--(-2.5,-2);
			\draw[-latex] (-2.5,-2)--(-5/4,-1);
			
			\draw[] (4.5,2)--(2,0);
			\draw[-latex] (2,0)--(3.25,1);
			\draw[] (0,0)--(2.5,2);
			\draw[-latex] (0,0)--(5/4,1);
			\draw[] (4.5,-2)--(2,0);
			\draw[-latex] (2,0)--(3.25,-1);
			\draw[] (0,0)--(2.5,-2);
			\draw[-latex] (0,0)--(5/4,-1);
			
			\node  [above]  at (-2.4,0.8) {\footnotesize$V_2$};
			\node  [below]  at (-2.4,-0.8) {\footnotesize$V_3$};

			\node  [above]  at (2.4,0.8){\footnotesize$V_1$};
			\node  [below]  at (2.4,-0.8){\footnotesize$V_4$};
		\end{tikzpicture}
		\caption{\footnotesize The open subsets $\{V_j\}_{j=1}^4$.}
		\label{fmatp2}
	\end{figure}

	Let $u(s;c_1,0,-c_1)$ denote the smooth real-valued solution of \eqref{p23} corresponding to $(c_1,0,-c_1)$ and $M^{P}(k) = M^{P}(k,s;c_1,0,-c_1)$ be the corresponding solution of RH problem \ref{modelp2}. Denote the open subsets $\{V_j\}_{j=1}^4$, as shown in Figure \ref{fmatp2}.
	To match with the modified Painlev\'{e} \uppercase\expandafter{\romannumeral2}  RH problem, we make a transformation  $N(k)=N(k,s,c_1,k_0)$ by
	\begin{equation*}
		N(k)= M^{P}(k) \times \begin{cases}
			\begin{pmatrix} 1 & 0 &0 \\ c_1e^{2i (\frac{4k^3}{3} + s k)} & 1 &0 \\ 0 & 0 & 1 \end{pmatrix}, \quad k \in V_1 \cup V_2,\\
			\begin{pmatrix} 1 & \bar{c}_1 e^{-2i (\frac{4k^3}{3} + s k)} &0 \\ 0 & 1 &0 \\ 0 & 0 & 1 \end{pmatrix},\quad k \in V_3 \cup V_4.\\
		\end{cases}
	\end{equation*}
    Then $N(k)$ satisfies the above RH problem \ref{modelp2}.

\end{proof}

\noindent\textbf{Acknowledgements}

This work is supported by  the National Science
Foundation of China (Grant No. 11671095,  51879045).


\end{document}